\documentclass[11pt]{CGC2}
\usepackage[english]{babel}
\usepackage{verbatim}
\usepackage{psfrag}
\usepackage{graphicx}
\usepackage{amsmath}
\usepackage{xcolor}
\usepackage{setspace}

\newcommand{\sms}{\setminus}
\newcommand{\lex}{\mathrm{lex}}

\begin{document}
\Logo{BCRI--CGC--preprint, http://www.bcri.ucc.ie}

\begin{frontmatter}

\title{Improved decoding of affine--variety codes}

{\author{Chiara Marcolla}} {\tt{(chiara.marcolla@unitn.it)}}\\
 {{Department of Mathematics, University of Trento, Italy}}

{\author{Emmanuela Orsini}}
{\tt{(orsini@posso.dm.unipi.it) }}\\
{{Department of Mathematics,
 University of Pisa, Italy }}

{\author{Massimiliano Sala}} {\tt{(maxsalacodes@gmail.com)}}\\
{ Department of Mathematics, University of Trento, Italy}

\runauthor{C.~Marcolla, E.~Orsini, M.~Sala}

\begin{abstract}
General error locator polynomials are polynomials able to decode 
any correctable syndrome for a given linear code. Such polynomials 
are known to exist for all cyclic codes and for a large class of linear codes.
We provide some decoding techniques for affine-variety codes using some
multidimensional extensions of general error locator polynomials.
We prove the existence of such polynomials for any correctable affine-variety
code and hence for any linear code. We propose two main different approaches, that depend on the underlying geometry. We compute some interesting
cases, including Hermitian codes. To prove our coding theory results, we develop a theory for special classes of zero-dimensional ideals, that can
be considered generalizations of stratified ideals. Our improvement with respect to stratified ideals is twofold: we generalize from one variable to many variables and we introduce points with multiplicities.
\end{abstract}

\end{frontmatter}

\section{Introduction}
\label{intro}
 \vspace{-.3cm}
Affine-variety codes were introduced by Fitzgerald and Lax
in \cite{CGC-cd-art-lax}
and provide a way to represent any linear code as an
evaluation code for a suitable polynomial ideal. Unsurprisingly, this rather
general description does not provide immediately any efficient
decoding algorithm. The lack of an efficient decoding algorithm is one of
the main drawbacks of this nice approach, which has unfortunately
still not received the attention it deserves, with few exceptions
(\cite{CGC-cd-art-Geil-1}, \cite{CGC-cd-art-salazar}).
Some \Gr \ basis techniques have been proposed in
\cite{CGC-cd-art-lax} to decode these codes, which may be
efficient depending on the underlying algebraic structure.

General error locator polynomials are polynomials introduced by us
in \cite{CGC-cd-art-gelp1} to decode cyclic codes. Their roots,
when specialized to a syndrome, give the error locations.
They can be used to decode any linear code, if it possesses them.
Giorgetti and Sala in \cite{CGC-cd-prep-martamax,CGC-cd-art-martamax} 
have  found a
large family of linear codes possessing such polynomials. When
the general error locator polynomial admits a sparse
representation, the decoding for the code is very fast.
Experimental evidence (and theoretical proofs for special cases)
suggests their sparsity in many interesting cases
(\cite{CGC-cd-prep-teomaxmanu}, \cite{CGC-cd-art-teomaxmanu}, \cite{CGC-cd-art-YaotsChung10}, \cite{CGC2-cd-art-LeeChangChen10}, \cite{CGC2-cd-art-LeeChangJingChen10}).\\
\indent We report several other approaches on decoding linear and cyclic codes with  \Gr\ bases:
\cite{CGC-BulPell}, \cite{CGC-augot}, \cite{CGC-cd-inbook-D1eleanna}.\\

In this paper we generalize our formerly proposed locator
polynomials to cover also the multi-dimensional case and hence the
affine-variety case. By adapting the \Gr \ techniques  in
\cite{CGC-cd-art-lax}, \cite{CGC-cd-art-gelp1},
\cite{CGC-cd-art-martamax}, we can prove their existence for any
affine-variety code. Excluding this introduction, this paper contains the following sections.

\begin{itemize}
    \item[-] In Section \ref{prelimenaries} we recall definitions and properties
for affine-variety codes, stratified ideals (a special class of zero-dimensional ideals), general error locator polynomials and the Hasse derivative.

    \item[-] In Section  \ref{cooper} we summarize the decoding proposed in
\cite{CGC-cd-art-lax} and we propose several alternatives, discussing 
their merits and drawbacks, especially taking into consideration 
the underlying geometric situation. In particular, we introduce the notion
of ``ghost points'', which are points added to the variety to play the role
of non-valid error locations. This way we can define a first generalization
of general error locator polynomials to the multivariate case (Definition \ref{locDebole}),
which provides a first decoding strategy. We also introduce evaluator
polynomials (Definition \ref{eval}) that permits a second strategy. While
the existence of evaluator polynomials can be proved directly using 
the theory of stratified ideals, unfortunately
in this section we lack the theoretical background to prove the existence
of these multivariate locators. 

    \item[-] In Section \ref{zeroideal} we extend  the results in \cite{CGC-cd-art-gelp1} for stratified ideals to cover also the ``multi-dimensional case'', 
that is precisely the theoretical background that we need for any
multivariate generalization of general locators.
Unexpectedly, there is no obvious ``natural'' way to extend the core notion of stratified ideals. We present three generalizations in  Definition~\ref{weaklyideal}, \ref{strongly}. We discuss their implications and provide some preliminary results. \\
Given a zero-dimensional ideal we can consider an order on its elimination ideals with a decreasing number of variables.
By choosing two consecutive elimination ideals $I$ and $I'$, we have a natural
projection from $\mathcal{V}(I)$ to $\mathcal{V}(I')$. At this stage, we highlight the role of two natural numbers: the maximum degree of some special 
polynomials in suitable \Gr\ basis of $I'$ and the maximum number of extensions to $\mathcal{V}(I)$ for points in $\mathcal{V}(I')$. It is convenient to present these values
as functions, respectively, $\eta$ (Definition~\ref{levelf}) and $\zeta$ (\ref{zeta}).
Section \ref{zeroideal} ends with the statement of Proposition~\ref{GBstructure}, which is the  main result claimed in this section (but not proved here).
Proposition \ref{GBstructure} is, in some sense, the multivariate analogue of Proposition~\ref{kk} on stratified ideals, and shows that for our ideals
$\eta$ and $\zeta$ coincides in this setting.

    \item[-]  Section~\ref{dim} is devoted to the long  proof of Proposition \ref{GBstructure}. This proposition describes some features of the \Gr\ basis
of (the elimination ideals of) a zero-dimensional radical ideal $J$.
The proof is constructive and relies on iterated applications of some versions of the Buchberger-M\"{o}ller algorithm.
To be more precise, we can start from the vanishing ideal of a single point.
For any monomial ordering it is trivial  
to determine its \Gr\ basis. In particular, $\zeta$ and $\eta$
coincide.
By adding more points, the shape of the \Gr\ basis becomes more complex, but we can follow what happens to the leading terms, if we are only interested in the variable involved in the projection $\mathcal{V}(I)\rightarrow \mathcal{V}(I')$.
When we have added enough points, we will get again $J$, since $J$ is radical.
To apply the Buchberger-M\"{o}ller algorithm, we need to add the points one by one. The difficult part is choosing the point in such a way
that $\eta$ and $\zeta$ grow exactly by the same amount.

    \item[-] Unfortunately,
our result in the multidimensional case, Proposition \ref{GBstructure}, is not as strong as our result in the one-variable case, Proposition  \ref{kk}.
In Section \ref{Loc}, it does allow us to prove the existence of our first generalization of locators
in Theorem \ref{bomba}, but we show that better locators can be found,
as in Definition \ref{zeroaf}. We discuss with examples 
a new decoding strategy by applying these locators, but again 
for the moment we are unable
to prove their existence, since they use multiplicities. This will be done
in the next section.

\item[-]
 In Section \ref{hasse} we develop the theory for generalizing
stratified ideals to the multivariate case with multiplicities.\\
As usual, we are interested in suitable \Gr\ bases of elimination
ideals of some zero-dimensional ideals.
First, we introduce the notion  of \textit{stuffed ideals} 
(Definition \ref{stuffed}), which basically means that the roots of some
polynomials in these \Gr\ bases have the ``expected'' multiplicity.
We give a constructive method (``stuffing'') to obtain stuffed ideals 
from special 
classes of ideals (in particular, radical ideals will do).
Our main results here are Theorem~\ref{TeoHasse}, that ensures
that the desired shape of our \Gr\ bases is unchanged under stuffing,
and Theorem~\ref{bombastaffato}, that ensures the existence of our sought-after
locators (in our \Gr\ bases).

    \item[-] In Section \ref{families} we compute some examples from different
families of affine-variety codes. In particular, we formally determine the shape for multivariate locator polynomials in the Hermitian case, for any $q\geq 2$ and $t=2$ (Theorem \ref{evviva}), both in our weaker version and in our
stronger version. 

\item[-] In Section \ref{conc} we provide  further comments and draw some conclusions. 
\end{itemize}


\section{Preliminaries}
\label{prelimenaries}
In this section we fix some notation and recall some known results.

We denote by $\FF_q$   the field with $q$ elements, where $q$ is a
power of a prime, and by  $n \geq 1$ a natural number. Let
$(\FF_q)^n$ be the vector space of dimension $n$ over $\FF_q$. Any
 vector subspace $C \subset (\FF_q)^n$ is a linear code (over
$\FF_q$).\\
From now on, we denote by  $\mathbb{K}$  any (not necessarily
finite) field and by $\overline{\mathbb{K}}$ its algebraic closure.

For any ideal $I$ in a polynomial ring $\mathbb{K}[X]$,
$X=\{x_1,\dots, x_m\}$, we denote by $\mathcal{V}(I)$ its variety
in  $\overline{\mathbb{K}}$. For any $Z \subset
\overline{\mathbb{K}}^m$ we denote by $\mathcal{I}(Z)$ the vanishing ideal of
$Z$, $\mathcal{I}(Z)\subset \mathbb{K}[X]$.

For any $f \in \mathbb{K}[X]$ and any term order $>$ on 
$\mathbb{K}[X]$, we denote by ${\bf T}(f)$ the leading 
term of $f$ with respect to $>$. We assume the reader familiar with the standard theory of \Gr\ bases, see e.g. 
\cite{CGC-alg-book-teo2},
\cite{CGC-alg-book-teo3},
\cite{CGC-cd-inbook-D1moratech}. 
When we have fixed the polynomial ring and the term order, we write $\mathrm{GB}(I)$ for the (unique) reduced \Gr\ basis of $I$.\\

\noindent We briefly recall the notion of ``block order", since it is less frequently met in literature.
Let $X=\{x_1,\ldots,x_m\}$ and $Y=\{y_1,\ldots,y_r\}$ be two variable sets. Let  $<_X$ and $<_Y$ be two orders, on the monomials of $X$ and on the monomials of $Y$, respectively.
We can define an order $<=(<_X,<_Y)$ on the monomials of $X\cup Y$ ({\em{block order}}) as follows:
$$
\begin{array}{l}
    
    x_1^{i_1}\dots x_m^{i_m}\, y_1^{j_1}\dots y_r^{j_r} <  x_1^{a_1}\dots x_m^{a_m}\, y_1^{b_1}\dots y_r^{b_r}\mbox{ if }\\
\mbox{  \texttt{either} } y_1^{j_1}\dots y_r^{j_r} <_Y  y_1^{b_1}\dots y_r^{b_r}\\ 
\mbox{  \texttt{or} } y_1^{j_1}\dots y_r^{j_r} =  y_1^{b_1}\dots y_r^{b_r}   \mbox{  and } x_1^{i_1}\dots x_m^{i_m} <_X  x_1^{a_1}\dots x_m^{a_m}

\end{array}
$$
The definition of a block order for more variable sets is a direct generalization.\\

If $f\in\mathbb{K}[X,y]$, with $y>X$, we can write $f$ as
$$f=a_{u} y^{u}+\ldots+a_1y+a_0, \quad \mbox{ with }  a_i\in\mathbb{ K}[X] \mbox{ and } a_{u}\ne 0 \,.$$
Then we say that $a_{u}$ is the \textit{leading polynomial} of $f$.

\subsection{Affine--variety codes}\label{2.1}

Let $m \geq 1$ and  $I \subseteq \FF_q[X]=\FF_q[x_1,\dots,x_m]$ be an
ideal such that $$\{x_1^q-x_1,x_2^q-x_2,\dots,x_m^q-x_m\} \subset
I.$$ Let  $\mathcal{V}(I)=\{P_1, P_2, \dots, P_n\}$.
Since $I$ is a zero-dimensional radical ideal,  
we have an isomorphism of
$\FF_q$ vector spaces (an evaluation):
\begin{align*}
\phi:  &\,\,R=\FF_q[x_1,\dots,x_m]/I \; \,\,\longrightarrow \,\,
\hspace{0.2cm} (\mathbb{F}_q)^n\\
 & \hspace{2.15cm} f \hspace{1.8cm}   \longmapsto
\hspace{0.3cm} (f(P_1), \dots, f(P_n)).
\end{align*}

Let $L \subseteq R$ be an $\FF_q$ vector subspace of $R$ of dimension $r$.\\

\begin{definition}[\cite{CGC-cd-art-lax}]
 The {\bf
affine--variety code} $C(I,L)$ is the image $\phi(L)$ and the
affine--variety code
 $C^{\perp}(I,L)$ is its dual code.
\end{definition}

If $b_1,\dots,b_r$ is a linear basis for $L$ over $\FF_q$, then
the matrix
{\small{$$\left(
\begin{array}{cccc}
b_1(P_1) &b_1(P_2)&\dots&b_1(P_n)\\
\vdots & \vdots & \dots & \vdots\\
b_r(P_1) & b_r(P_2) & \dots & b_r(P_n)
\end{array}
\right)
$$}}
is a generator matrix for $C(I,L)$ and a parity--check matrix for
$C^{\perp}(I,L)$.
\begin{theorem}[\cite{CGC-cd-art-lax}]
Every linear code may be represented as an affine--variety code.
\end{theorem}

From now on, $q, m, n, I$ and $L$ are understood to be defined as
above.

For any $1\leq i\leq m$, $\widehat{\pi}_i$ denotes the natural projection $\widehat{\pi}_i: (\mathbb{F}_q)^m\rightarrow \mathbb{F}_q$, such that $\widehat{\pi}_i(\bar x_1,\ldots,\bar x_m)=\bar x_i$.

\subsection{Stratified ideals}\label{2.2}

In this subsection we summarize some definitions and results  from
\cite{CGC-cd-art-martamax}.

Let  $J \subset
\mathbb{K}[\mathcal{\mathcal{S}},\mathcal{A},\mathcal{\mathcal{T}}]$
be a zero--dimensional radical ideal, with variables
$\mathcal{\mathcal{S}}=\{{\sf s}_1,\dots,$ ${\sf s}_N\}$,
$\mathcal{A}=\{{\sf a}_1,\dots,{\sf a}_L\}$, $\mathcal{T}=\{{\sf
t}_1,\dots,{\sf t}_K\}$. We fix a term ordering $<$  on
$\mathbb{K}[\mathcal{\mathcal{S}},\mathcal{A},\mathcal{\mathcal{T}}]$,
with $\mathcal{\mathcal{S}}<\mathcal{A}<\mathcal{T}$, such that
the $\mathcal{A}$ variables are   ordered by
${\sf a}_L < {\sf a}_{L -1}<\ldots < {\sf a}_1$.

Let us define the elimination ideals
$J_\mathcal{S}\,=\,J\cap\mathbb{K}[\mathcal{S}],\quad J_{\mathcal{S},{\sf
a}_L}\,=\,J\cap\mathbb{K}[\mathcal{S},{\sf a}_L],\ldots,$
$J_{\mathcal{S},{\sf a}_L,\ldots,{\sf
a}_1}=J\cap\mathbb{K}[\mathcal{S},{\sf a}_L,\,\ldots,$ $\,{\sf
a}_1]$ $=J\cap\mathbb{K}[\mathcal{S},\mathcal{A}]$.

We want to view ${\mathcal{V}}(J_{\mathcal{S}})$ as a disjoint union of some
sets. The way we define these sets is linked to the fact that any point $P$
in ${\mathcal{V}}(J_{\mathcal{S}})$ can be extended to at least one point
in ${\mathcal{V}}(J_{\mathcal{S},{\sf a}_L})$.
But the number of all possible extensions of $P$ in 
${\mathcal{V}}(J_{\mathcal{S},{\sf a}_L})$ is finite, since the ideal is 
zero-dimensional, so we can partition ${\mathcal{V}}(J_{\mathcal{S}})$ in sets 
such that all points in the same set share the same number of extensions.
We denote by $\lambda(L)$ the maximum number of extensions in
${\mathcal{V}}(J_{\mathcal{S},{\sf a}_L})$ for any $P\in {\mathcal{V}}(J_{\mathcal{S}})$.
The same principle applies when we consider the variety of another elimination 
ideal, e.g. $\mathcal{ V}(J_{\mathcal{S},{\sf a}_L,\dots,{\sf a}_h})$.
We can partition it into subsets such that all points in the same subset
share the same number of extensions in $\mathcal{ V}(J_{\mathcal{S},{\sf a}_L,\dots,{\sf {a}}_{h},{\sf {a}}_{h-1}})$. The maximum number of extensions is
denoted by $\lambda(h-1)$.

We write our partitioning in a formal way, as follows:
{\small{
$$\begin{array}{l}
\mathcal{ V}(J_\mathcal{S})= \sqcup_{l=1}^{\lambda(L)}\Sigma_l^{L},\,\,
        \mbox{with} \\
    \Sigma_l^{L}= \{ (\overline{{\sf s}}_1,\ldots,\overline{{\sf s}}_N)\in \mathcal{ V}(J_\mathcal{S}) \mid \exists
    {\mbox{\ exactly }} l {\mbox{ distinct values }}
    \bar{{\sf a}}^{(1)}_{L}, \ldots,  \bar{{\sf a}}^{(l)}_{L} \\
     \quad\,\,\,\mathrm{s.t.} \,\,(\overline{{\sf s}}_1,\ldots,\overline{{\sf s}}_N,\bar{{\sf a}}^{(\ell)}_{L})
   \in \mathcal{ V}(J_{\mathcal{S},{\sf a}_{L}}), 1 \leq \ell \leq l  \};\\
\mathcal{ V}(J_{\mathcal{S},{\sf a}_{L},\dots,{\sf a}_h})=\sqcup_{l=1}^{\lambda(h-1)}
\Sigma_l^{h-1},\; 2 \leq h \leq L, \, \mbox{with}
\end{array}$$
$$\begin{array}{l}
    \Sigma_l^{h-1}= \{ (\overline{{\sf s}}_1,\ldots,\overline{{\sf s}}_N,
                           \overline{{\sf a}}_{L},\dots,\overline{{\sf a}}_h) \in
    \mathcal{ V}(J_{\mathcal{S},{\sf a}_L,\dots,{\sf a}_h}) \mid
    \exists \,\,\,\,{\mbox{exactly}}\,\, l {\mbox{ distinct values}} \\
   \quad\quad\,\,\, \bar{{\sf a}}^{(1)}_{h-1}, \ldots,
     \bar{{\sf a}}^{(l)}_{h-1} \,\,
     \mathrm{ s.t.} \,\, (\overline{{\sf s}}_1,\ldots,\overline{{\sf s}}_N,\overline{{\sf a}}_{L},
           \dots,\overline{{\sf a}}_h,\
    \bar{{\sf a}}^{(\ell)}_{h-1}) \in \mathcal{ V}(J_{\mathcal{S},{\sf a}_{L},\dots,{\sf a}_{h-1}}),\, 1 \leq \ell \leq l  \}.
\end{array}$$}}
\noindent
For an arbitrary zero--dimensional ideal $J$, nothing can be said
about $\lambda(h)$, except that $\lambda(h)\geq 1$ for any $1\leq
h \leq L$.

\begin{definition}[\cite{CGC-cd-art-martamax}]\label{stratificato}
With the above notation, let $J$ be a zero-dimensional radical
ideal. We say that $J$ is \textbf{stratified}, with respect to the
$\mathcal{A}$ variables, if:
\begin{itemize}
\item[(a)] $\lambda(h)=h$, $1 \leq h \leq L$, and \label{stratificato1}
\item[(b)]  $\sum_l^{h}\neq \emptyset$, $1 \leq h \leq L,\, 1 \leq l \leq h$.\label{stratificato2}
\end{itemize}
\end{definition}
To explain conditions
$(a)$ and $(b)$ in the above definition, let us consider $h=L$ and
think of the projection
\begin{equation} \label{mappa}
\pi: \mathcal{V}(J_{\mathcal{S},{\sf a}_L})  \rightarrow
\mathcal{V}(J_{\mathcal{S}}).
\end{equation}
In this case, $(a)$ in Definition \ref{stratificato} is equivalent to saying that any point in
$\mathcal{V}(J_{\mathcal{S}})$ has at most $L$  pre-images in
$\mathcal{V}(J_{\mathcal{S},{\sf a}_L})$
via $\pi$, and that there is at least one point with (exactly) $L$ pre-images.
On the other hand, $(b)$ implies that, if for a point $P
\in \mathcal{V}(J_{\mathcal{S}})$ we have  $|\pi^{-1}(P)| = m \geq
2$, then there is at least another point  $Q \in
\mathcal{V}(J_{\mathcal{S}})$  such that $|\pi^{-1}(Q)| = m-1$.

\begin{example}
Let $\mathcal{S}=\{{\sf s}_1\},\,\mathcal{A}=\{{\sf a}_1,{\sf
a}_2,{\sf a}_3\}$ ($L=3$) and $\mathcal{T}=\{{\sf t}_1\}$ such
that $\mathcal{S}<\mathcal{A}<\mathcal{T}$ and ${\sf a}_3<{\sf
a}_2<{\sf a}_1$. Let us consider $\mathcal{J}=\mathcal{I}(Z)
\subset \mathbb{C}[{\sf s}_1,{\sf a}_3,{\sf a}_2,{\sf a}_1,{\sf
t}_1]$ with $Z= \{(1,2,1,0,0), (1,2,2,0,0),$ $(1,4,0,0,0),$
$(1,6,0,0,0),$ $(2,5,0,0,0),$ $(3,1,0,0,0),$ $(3,3,0,0,0),$
$(5,2,0,0,0)\}.$ Then: {\scriptsize{
\begin{align*}
\mathcal{V}(J_\mathcal{S})&=\{1,2,3,5\}\\
\mathcal{V}(J_{\mathcal{S},{\sf
a}_3})&=\{(1,2),(1,4),(1,6),(2,5),(3,1),(3,3),(5,2)\}\\
\mathcal{V}(J_{\mathcal{S},{\sf
a}_3,{\sf a}_2})&=\{(1,2,1),(1,2,2)(1,4,0),(1,6,0),(2,5,0),(3,1,0),(3,3,0),(5,2,0)\}\\
\mathcal{V}(J_{\mathcal{S},{\sf a}_3,{\sf a}_2,{\sf
a}_1})&=\{(1,2,1,0),(1,2,2,0)(1,4,0,0),(1,6,0,0),(2,5,0,0),(3,1,0,0),(3,3,0,0),(5,2,0,0)\}
\end{align*}}}
Let us consider the projection $\pi:
\mathcal{V}(\mathcal{J}_{\mathcal{S},{\sf a}_3})  \rightarrow
\mathcal{V}(\mathcal{J}_{\mathcal{S}})$. Then:
$$
|\pi^{-1}(\{5\})|=1,\,\,|\pi^{-1}(\{2\})|=1,\,\,|\pi^{-1}(\{3\})|=2,\,\,|\pi^{-1}(\{1\})|=3 \,,
$$
so $\sum_1^3=\{2,5\}$, $\sum_2^3=\{3\}$, $\sum_3^3=\{1\}$ and
$\sum_{i}^3=\emptyset,\,\, i>3$. This means that
$\lambda(L)=\lambda(3)=3$ and $\sum_l^3$ is not empty, for
$l=1,2,3$. Thus the conditions  of Definition \ref{stratificato} are
satisfied for $h=L=3$ (see Fig. \ref{fig1}). In the same way,  it
is easy to verify said conditions also for $h=1,2$, and hence the
ideal $\mathcal{J}$ is stratified with respect to the $\mathcal{A}$
variables.
\begin{figure}[h!]
\centering
\includegraphics[width=7cm]{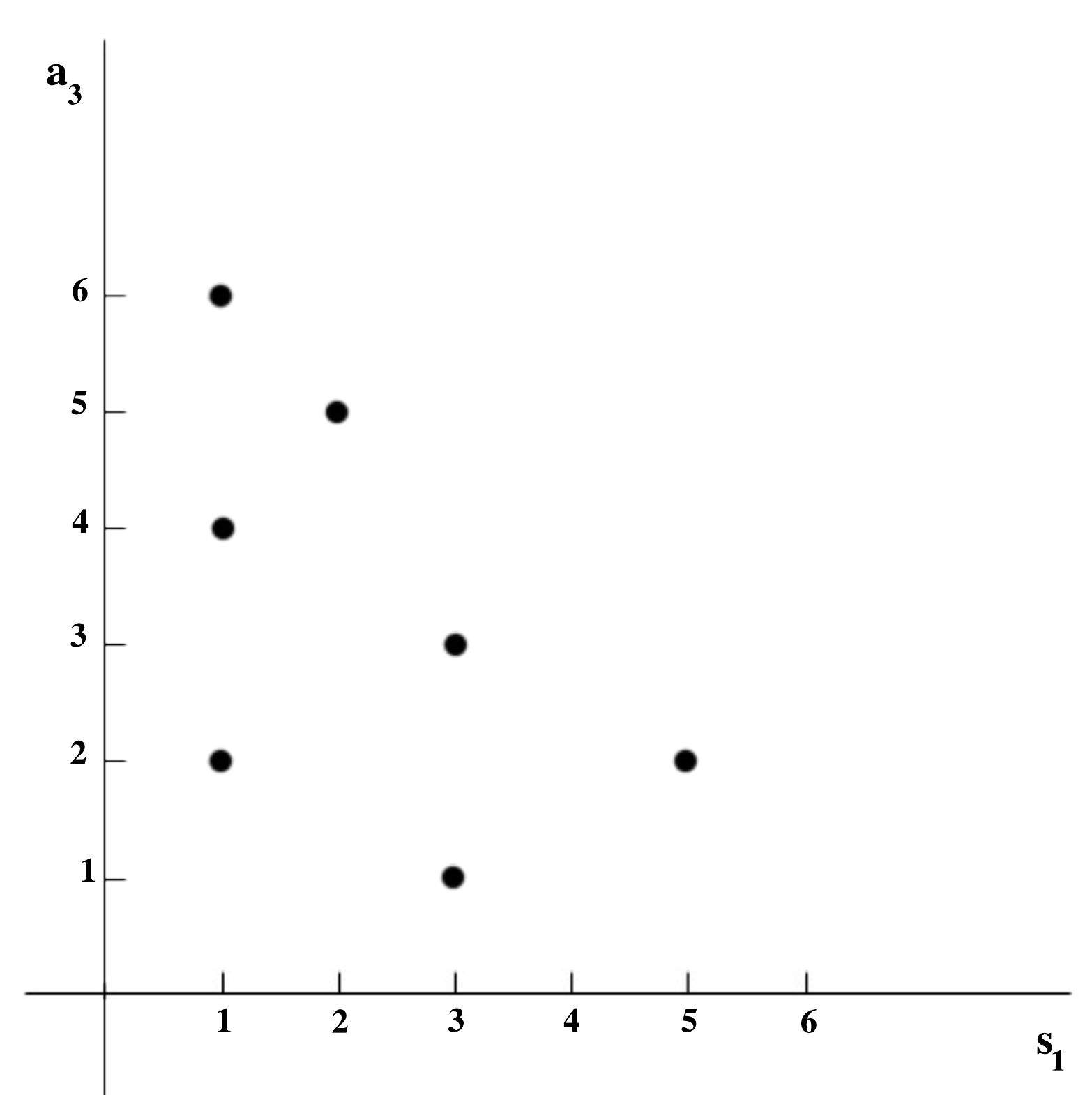}
\caption{A variety in a stratified case}\label{fig1}
\end{figure}
\end{example}

With the above notation, an immediate consequence of Theorem $3.6$
in \cite{CGC-cd-prep-martamax} (Theorem 32 in \cite{CGC-cd-art-martamax})
is the following proposition.
\begin{proposition} \label{kk}
Let $<$ be any lexicographic term order with  $\mathcal{\mathcal{S}}<\mathcal{A}<\mathcal{T}$
and ${\sf a}_L < {\sf a}_{L-1} < \dots < {\sf a}_1$.
Let $J$ be a stratified ideal with respect to the $\mathcal{A}$ variables. 
Let $G=\mathrm{GB}(J)$. Then $G$ contains one and only one polynomial $g$
such that:
$$
  g \in \KK[\mathcal{S},{\sf a}_L],\qquad {\bf T}(g)={\sf a}_L^L.
$$
\end{proposition}


\subsection{Root multiplicities and Hasse derivative}\label{hassesub}

\begin{definition}
Let $g=\sum_i a_ix^i \in \mathbb{K}[x]$. Then the \textbf{$\mathbf{n}$-th Hasse derivative} of $g$ is $\varphi^{(n)}(g)$ and the $\mathbf{n}$\textbf{-th formal derivative} of $g$ is $g^{(n)}$, where

$$\varphi^{(n)}(g)=\sum_i {i\choose n} a_ix^{i-n} \mbox{ and }\,\, g^{(n)}=n ! \sum_i {i\choose n} a_ix^{i-n}.$$
\end{definition}
We can note that
$g^{(n)}=n!\varphi^{(n)}(g)$. In a field with characteristic $p$, it is more convenient to use the Hasse derivative, because $n!=0$ for all $n\geq p$.\\
\indent Note that $\varphi^{(2)}(g)\ne \varphi^{(1)}\left(\varphi^{(1)}(g)\right)$.
\begin{definition}
Let $g\in \mathbb{K}[x]$, $g\ne 0$, $P\in \mathbb{K}$ and $g(P)=0$. The \textbf{ multiplicity} of $P$ as a root of $g$ is the largest integer $r\geq 1$ such that
$$\varphi^{(k)}(g)(P)=\varphi^{(k)}(g)_{\big{|_{x=
P}}}=0, \quad\textrm{ for }\, 0\leq k\leq r-1.$$
\end{definition}
The following theorem is well-known, see e.g. \cite{CGC-cd-book-niederreiter97}.
\begin{theorem}\label{teomult}
Let $g,f\in \mathbb{K}[x]$ and let $g$ be irreducible. Then
$$g^r|f \qquad\iff\qquad g|\varphi^{(k)}(f)\,\,\textrm{ for } \, 0\leq k\leq r-1 \,.$$
\end{theorem}
As a consequence of the previous theorem when $g=(x-P)$ for any $P\in \mathbb{K}$, we have 
$$(x-P)^r|f \quad\iff \quad\varphi^{(k)}(f)_{\big{|_{x=
P}}}=0\,\,\textrm{ for  } \, 0\leq k\leq r-1\,.$$


\subsection{General error locator polynomials} \label{seclocator}

Let $C$ be an $[n,k,d]$ linear code over $\FF_q$ with correction
capability $t \geq 1$. Choose any  parity-check matrix with
entries in an appropriate extension field $\FF_{q^M}$ of $\FF_q$,
$M \geq 1$. 
Its  syndromes lie in $(\FF_{q^M})^{n-k}$ and form a
vector space of dimension $r=n-k$ over $\FF_q$. Let $\alpha$ be a
primitive $n$-th root of unity in $\FF_{q^M}$.
\begin{definition}\label{zero}
Let $\mathcal{ L}$ be a polynomial in $\FF_q[S,x]$, where
$S=(s_1,\dots,s_r)$. Then $\mathcal{ L}$ is a {\bf general error
locator polynomial} of $C$ if
\begin{enumerate}
\item $\mathcal{ L}(S,x)= x^t+a_{t-1}x^{t-1}+ \cdots +a_0$,
  with $a_j \in \FF_q[S]$, $0 \leq j \leq t-1$, that is, $\mathcal{L}$ is
  a monic polynomial with degree
  $t$ with respect to the variable $x$
  and its coefficients are in $\FF_q[S]$;\label{loc1}
\item given a syndrome
  ${\bf s}=(\overline{s}_1,\dots \overline{s}_r)\in (\FF_{q^M})^r$,
  corresponding to an  error vector of weight
  $\mu \leq t$ and error positions $\{k_1, \dots, k_{\mu}\}$,
  if we evaluate the $S$ variables at  ${\bf s}$, then the roots of
  $\mathcal{L}({\bf s},x)$ are exactly
  $\{\alpha^{k_1},\dots,\alpha^{k_\mu},0\}$, where the multiplicity of $0$ is $t-\mu$.\label{loc2}
\end {enumerate}
\end{definition}
Given any (correctable) linear code $C$, 
the existence of a general error locator
polynomial is not known. In \cite{CGC-cd-art-gelp1} the authors
prove its existence for any cyclic code and recently in
\cite{CGC-cd-prep-martamax,CGC-cd-art-martamax,CGC-cd-phdthesis-marta} 
its existence has been proved for a large class of linear codes.

We can extend Definition \ref{zero} to the case when there are
also erasures.
\begin{definition}
\label{zeronu} Let $\mathcal{L}$ be a polynomial  in
$\FF_q[S,W,x]$, $S=(s_1,\dots,s_r)$ and  $W=(w_1,\dots,w_{\nu})$,
where $\nu$ is the number of occurred erasures.
Let $2\tau + \nu<d$.
Then
$\mathcal{L}$ is a {\bf general error locator polynomial of type}
${\boldsymbol{ \nu}}$ of $C$ if
\begin{enumerate}
\item $ \mathcal{ L}(S,W, x)= x^{\tau}+a_{\tau-1}x^{\tau-1}+ \cdots +a_0$,
  with $a_j \in \FF_q[S,W]$, for any\\ $0 \leq j \leq \tau-1$, that
  is,
  $\mathcal{ L}$ has degree $\tau$ w.r.t. $x$
 and coefficients in $\FF_q[S,W]$;
\item for any syndrome ${\bf s}=(\overline{s}_1,\dots,
  \overline{s}_r)$ and any erasure location vector \\
       {\bf w}$=
       (\overline{w}_{1},\dots,$ $\overline{w}_{\nu})$, corresponding to an
       error vector  of weight $\mu \leq \tau$ and error locations $\{k_1, \dots,
       k_{\mu}\}$, if we evaluate the $S$ variables at ${\bf s}$ and the
       $W$ variables at {\bf w}, then the roots of $\mathcal{ L}({\bf s},{\bf w},x)$ are
       $\{\alpha^{k_1},\dots,\alpha^{k_\mu},0\}$,  where the multiplicity of $0$ is $\tau-\mu$.
\end{enumerate}
\end{definition}
%
For the benefit of readers unfamiliar with simultaneous correction of 
errors and erasures, we sketch how it works.
When some (say $\nu$) symbols  are not recognised by the receiver, the decoder
treats them as {\em erasures}. The decoder knows the positions of these erasures
$i_1,\ldots,i_\nu$, which means in our notation that the decoders knows 
the erasure locations grouped for convenience in the {\em erasure location
vector} $\mathbf{w}=(\bar w_1,\ldots,\bar w_\nu)=
(\alpha^{i_1},\ldots,\alpha^{i_\nu})$. A standard result in coding theory
is that it is possible to correct simultaneously $\nu$ erasures and
$\tau$ errors, provided that $2\tau+\nu < d$.

To be consistent with our notation, we may refer to a polynomial
in Definition~\ref{zero} also as  a
{\em general error locator polynomial of type 0}.

For a code $C$, the possession of a polynomial of each type $0\leq \nu<d$
might be a stronger condition than the possession of a polynomial of
type $0$, but in
\cite{CGC-cd-art-gelp1} the authors prove that any cyclic code
admits a polynomial of any type $\nu$,     for $0\leq \nu<d$.
In \cite{CGC-cd-art-martamax} the existence of general error locator polynomials (of any type) for a large class of linear codes was proved,
 but it is still unknown whether such a result holds for general linear codes.


\section{Decoding the affine variety code with the Cooper philosophy}
\label{cooper}
\subsection{The approach by Fitzgerald and Lax}
In \cite{CGC-cd-art-lax} a decoding technique was proposed
following what is known as the ``Cooper
philosophy''.
Although this terminology has been established only
recently (\cite{CGC-cd-inbook-D1moraorsini}), 
this decoding approach has a quite wide literature, e.g.
\cite{CGC-cd-art-cooper2},\cite{CGC-cd-art-cooper1},
\cite{CGC-cd-art-caboaramora},\cite{CGC-cd-art-cooper3},
\cite{CGC-cd-art-CRHT1}.
We describe this technique for affine-variety codes, as follows (see Subsection \ref{2.1}).
Let $C^{\perp}(I,L)$ be an affine-variety code with dimension
$n-r$ and let $I=\langle g_1, \dots, g_{\gamma} \rangle$. Let $L$ be linearly
generated by $b_1,\ldots,b_r$.
Then we can denote
by $J^{C,t}_{{\mathcal{FL}}}$ the ideal (${\mathcal F}{\mathcal L}$ is for
``FitzgeraldLax'')
\vspace{-0.4cm}
{\footnotesize{$$
J^{C,t}_{\mathcal{FL}} \subset
\FF_q[s_1,\dots,s_r,x_{t,1},\dots,x_{t,m},\dots,x_{1,
1},\dots,x_{1,m}, e_1,\dots,e_t]= \FF_q[S,X_t,\dots,X_1,E]
$$}}
where\footnote{To speed up the basis computation we can add
{\scriptsize{$\left\{x_{j,\iota}^q-x_{j,\iota}\right\}_{\substack{1\leq j \leq t,\\ 1\leq \iota \leq m}}$}} to the ideal.}
{\footnotesize{
\begin{equation}\label{idealLax}
\begin{array}{ll}
J_{\mathcal{FL}}^{C,t} = \Big\langle &\left\{\sum_{j=1}^t e_j b_{\rho}
(x_{j,1},\dots,x_{j,m})-s_{\rho}\right\}_{1\leq \rho \leq r},\\
                 & \left\{e_j^{q-1}-1\right\}_{1\leq j  \leq t}, \,\left\{g_h(x_{j,1},\dots,x_{j,m})\right\}_{\substack{1 \leq h \leq \gamma,\\ 1\leq j \leq t}} \Big\rangle \quad .
\end{array}
\end{equation}}}
Let $<_S$ be any term ordering on the variables $s_1, \dots, s_r$
and $<_{\lex}$ be the lexicographic  ordering on the variables $X_t, \dots,
X_1$, such that
$$
x_{t,1}<_{\lex} \dots <_{\lex} x_{t,m} <_{\lex}  \dots<_{lex} x_{1,1} <_{\lex} \dots <_{\lex} x_{1,m}.
$$
\noindent  Let $<_E$ be any term ordering on the variables $e_1,\ldots, e_t$.\\
Then let $<$ be the block order $(<_S,<_{\lex},<_E)$. We denote by $\mathcal{G}_{\mathcal{FL}}^{C,t}$ a \GR\ basis of
$J_{\mathcal{FL}}^{C,t}$ with respect to $<$.
In \cite{CGC-cd-art-lax} we can find a method  describing how  to find the
error locations and values, by applying elimination theory to the
polynomials in $\mathcal{G}_{FL}^{C,t}$.

\begin{example}
\label{ex:hermitiano}
Let $C=C^{\perp}(I,L)$ be the Hermitian code from the curve $y^2+y=x^3$
over $\FF_4$ and with defining monomials $\{1,x,y,x^2,x y\}$.
The eight points of the variety defined by $I$ are
{\tiny{\begin{align*}
&P_1=(0,0), P_2=(0,1), P_3=(1,\alpha), P_4=(1,\alpha^2), P_5=(\alpha,\alpha), P_6=(\alpha,\alpha^2), P_7=(\alpha^2,\alpha), P_8=(\alpha^2,\alpha^2),
\end{align*}}}
where $\alpha$ is any primitive element of $\FF_4$. 
It is well--known that $C$ corrects up
to $t=2$ errors. The ideal $J_{\mathcal{FL}}^{C,2} \subset
\FF_4[s_1,\dots,s_5,x_2,y_2,x_1,y_1,e_1,e_2]$ is
{\small{\begin{align*}
J_{\mathcal{FL}}^{C,2}=\langle\{&x_1^4-x_1, x_2^4-x_2,y_1^4-y_1,y_2^4-y_2, e_1^3-1,e_2^3-1,y_1^2+y_1-x_1^3,\\
&y_2^2+y_2-x_2^3, e_1+e_2-s_1, e_1x_1+e_2x_2-s_2, e_1y_1+e_2y_2-s_3,\\
&e_1x_1^2+e_2x_2^2-s_4, e_1x_1y_1+e_2x_2y_2-s_5\}\rangle.
\end{align*}}}
\end{example}

Typically the \Gr \ basis of $J_{\mathcal{FL}}^{C,t}$ that has been obtained
using the block order $<$ contains a large number of polynomials
and  most  are not useful for  decoding purposes. We
would have to choose a polynomial in $\FF_q[S,x_{t,1}]$ that, once
specialized in the received syndrome, could be used to find
the first coordinates of all the errors.  It is important to observe
that in this situation we do not know which polynomial is the
right one, because after the specialization  we can obtain a polynomial
which vanishes identically.

 \subsection{Rationale for our decoding ideals}

The approach presented in the previous section shares the same
problem with other similar approaches
(\cite{CGC-cd-art-CRHT2},\cite{CGC-cd-art-louYork},\cite{CGC-cd-art-caboaramora}).
In the portion of the \Gr\ basis corresponding to the
elimination ideal $I_{S,x_{t,1}}$, one should choose a polynomial
$g$ in $\FF_q[S,x_{t,1}]\setminus \FF_q[S]$, specialize it to the
received syndrome, and then find its $x_{t,1}$-roots. The problem
is that it is not possible to know in advance which polynomial has
to be chosen, and there might be hundreds of ``candidate''
polynomials. Let us call ideal $J_{\mathcal{F}\mathcal{L}}^{C,t}$ the
``Cooper ideal for affine-variety codes'' (the convenience for
this historically inaccurate name will be clear in a moment) and
the ``Cooper variety'' its variety.

The same problem is present in the ideal for decoding cyclic codes
presented in \cite{CGC-cd-art-CRHT2}, which we will call the
``Cooper ideal for cyclic codes'' (although again its formal
definition was first presented in \cite{CGC-cd-art-CRHT2}), where
a huge number of polynomials can be found as soon as the code
parameters are not trivial.
In this case an improvement was proposed in
\cite{CGC-cd-art-caboaramora}. Instead of specializing the whole
polynomial, one can specialize only its leading polynomial, since 
it does not vanish identically  if and only if the whole polynomial does not
vanish (by the Gianni-Kalkbrener theorem
\cite{CGC-alg-art-gianni}, \cite{CGC-alg-art-kalkbrener}). We
could adopt exactly the same strategy for the ``Cooper ideal for
affine-variety codes'' and thus get a significant improvement on
the algorithm proposed in \cite{CGC-cd-art-lax}. This improvement
would reduce the cost of the specialization, but would still require an
evaluation (in the worst case) for any candidate polynomial.
%
%
In Section 7 of \cite{CGC-cd-art-caboaramora} a more refined
strategy has been investigated, that is, the vanishing conditions
coming from the leading polynomials were grouped and a decision
tree was formed. In the example proposed there, this resulted in a
drastic reduction of the computations required to identify the
right candidate. Unfortunately, this strategy has not
been deeply investigated in the general case, but we believe that
it is obvious how this could be done also for the Cooper ideal for
affine-variety codes, obtaining thus another improvement.

In \cite{CGC-cd-art-louYork} it was noted that the Cooper variety
for cyclic codes contains also points that do not correspond to
valid  syndrome-error location pairs and thus are useless. In
\cite{CGC-cd-art-gelp1} the authors enlarge the Cooper ideal in
order to remove exactly the non-valid pairs, which we call
``spurious solutions''. The new ideal turns out to be stratified
(although the notion of stratified ideal is established later in
\cite{CGC-cd-art-martamax}) and hence to contain the general
error locator polynomial, thanks to deep properties of some \Gr\ bases of stratified ideals, which is the {\em only} polynomial that
needs to be specialized. We are now going to explain how this
improvement can be obtained also for the Cooper variety
for affine-variety codes.

\noindent We define several modified  versions of the Cooper ideal for decoding
affine-variety codes.
We summarize what we are going to do:
\begin{itemize}
    \item[-] In  Subsection \ref{Ghostpoint} we define a decoding ideal (\ref{fineideal}) that is able to correct any correctable error, even not knowing in advance the number of errors.
    \item[-] However, in Subsection \ref{Locdebole} we show why this decoding ideal does not necessarily contain locator polynomials that play the same role of generator error locator polynomials for cyclic codes. Still, these weak forms of locators (Definition \ref{locDebole}) can be used to decode.
    \item[-] In Section  \ref{zeroideal} we develop the commutative algebra necessary to show the existence of weak locators, with Section  \ref{dim} devoted to the long proof of the main result, and then in Section  \ref{Loc} we will finally be able to define a set of multi-dimensional general error locator polynomials (see Definition~\ref{zeroaf}). We define a suitable ideal containing this set as we show in Theorem~\ref{bombastaffato}.
\end{itemize}

\subsection{Decoding with ghost points}\label{Ghostpoint} 

Note that Fitzgerald and Lax consider the possible error locations as $t$ points in $\mathcal{V}(I)$, that we call $P_{\sigma_1},\ldots, P_{\sigma_t}$, but they denote their components dropping the reference to $\sigma$, that is, $P_{\sigma_l}=(x_{l,1},\ldots,x_{l,m})$ for $1\leq l\leq t$. We adhere to this notation from now on.\\

We observe that in the Cooper ideal (\ref{idealLax}) there is not any constraint on point pairs. But we want that all error locations are distinct. We have to force this, i.e. any two locations  $P_{\sigma_j}=(x_{j,1},\ldots,x_{j,m})$ and  $P_{\sigma_k}=(x_{k,1},\ldots,x_{k,m})$ must differ in at least one component. So we add this condition:
$$\prod_{1 \leq \iota \leq m}((x_{j,\iota}-x_{k,\iota})^{q-1}-1)=0 \qquad\textrm{ for }1\leq j< k\leq t. $$
In fact, if $\alpha\in \mathbb{F}_q$, then $\alpha\not=0\iff \alpha^{q-1}=1$.
Therefore, the product $\prod_{1 \leq \iota \leq m}((x_{j,\iota}-x_{k,\iota})^{q-1}-1)$ is zero if and only if at least for one $\iota$ we have $(x_{j,\iota}-x_{k,\iota})^{q-1}=1$, i.e. $x_{j,\iota}\not=x_{k,\iota}$ and thus $P_{\sigma_j}\ne P_{\sigma_k}$.
Our ideal becomes

\begin{equation}\label{Inostro}
    \begin{array}{ll}
\widehat{J}_{\mathcal{FL}}^{\,C,t} = \Big\langle &\left\{\sum_{j=1}^t e_j b_{\rho}
(x_{j,1},\dots,x_{j,m})-s_{\rho}\right\}_{1\leq \rho \leq r},
                 \left\{e_j^{q-1}-1\right\}_{1\leq j  \leq t},\\
                 & \left\{g_h(x_{j,1},\dots,x_{j,m})\right\}_{\substack{1 \leq h \leq \gamma,\\ 1\leq j \leq t}},\\
&\left\{\prod_{1 \leq \iota \leq m}
                 ((x_{j,\iota}-x_{k,\iota})^{q-1}-1) \right\}_{\substack{1 \leq j < k \leq t}}  \Big\rangle \quad .
\end{array}
\end{equation}

\begin{remark}\label{no}
Ideal $\widehat{J}_{\mathcal{FL}}^{\,C,t} $ can be used to correct and it will work better than $J_{\mathcal{FL}}^{\,C,t} $, since its variety does not contain spurious solutions. However, we cannot expect that $\widehat{J}_{\mathcal{FL}}^{\,C,t} $ contains polynomials with a role similar to that of the generic error locator in the cyclic case, because $\widehat{J}_{\mathcal{FL}}^{\,C,t} $ still depends on the knowledge of the error number. 
\end{remark}
In the following we modify (\ref{Inostro}) to allow for different-weight syndromes.
\begin{itemize}
    \item[(a)]First, we  note that in  $\widehat{J}_{\mathcal{FL}}^{\,C,t} $ the following condition is verified
 $$e_j^{q-1}=1\textrm{ with }j=1,\ldots,t.$$ 
\noindent This is equivalent to saying that  exactly $t$ errors occurred, which are $e_1,\ldots,e_t\in \mathbb{F}_q^*$. We must allow for some $e_j$ with $j=1,\ldots,t$ to be equal to zero.
We would obtain a new ideal where the conditions $e_j^{q-1}=1$ are replaced with $e_j^q=e_j$ for any $j=1,\ldots,t$.

 \item[(b)] We recall  the changes made to the Cooper   ideal in  \cite{CGC-cd-art-gelp1} for cyclic codes.  We consider the error vector  

$$e=(\underbrace{0,\dots,0}_{k_1-1},
\underset{\underset{k_1}{\uparrow}}{e_{1}},
0,\dots,0,\underset{\underset{k_l}{\uparrow}}{e_l}, 0,\dots,0,\underset{\underset{k_{\mu}}{\uparrow}}{e_{\mu}},
\underbrace{0,\dots,0}_{n-1-k_\mu})\quad \textrm{ with } \mu\leq t,$$
where $k_1,\ldots,k_{\mu}$ are the error positions and $e_1\ldots,e_{\mu}$ are the error values. We consider the $j$-th syndrome and we obtain the following equation 
\begin{equation}\label{sin1}
    \sum_{l=1}^{\mu}e_l(\alpha^{i_j})^{k_l}=s_j.
\end{equation}
(For the n-th rooth codes in \cite{CGC-cd-prep-martamax,CGC-cd-art-martamax} the formulas are slightly more complicated). To arrive at the desired equation
\begin{equation}\label{sin2}
    \sum_{l=1}^{t}e_l(\alpha^{i_j})^{k_l}=s_j
\end{equation}
we have to add the ``virtual error position'' $\textsf{k}$ defined as $\alpha^{\textsf{k}}=0 \,\,\forall \alpha \in \mathbb{F}$. Using the location $z_l=\alpha^{k_l}$ (and so the ``virtual error location'' is $\alpha^{\textsf{k}}=0$), equation  (\ref{sin2}) becomes
$$s_j=\sum_{l=1}^{\mu}e_l(z_l)^{i_j}+\sum_{l=\mu+1}^{t}e_l(\alpha^{\textsf{k}})^{k_l}
=\sum_{l=1}^{\mu}e_l(z_l)^{i_j}+\sum_{l=\mu+1}^{t}e_l(0)^{k_l}=\sum_{l=1}^{t}e_l(z_l)^{i_j}.$$

\noindent We can rephrase what we did by saying that we are using $0$ as a \textit{ghost error location}, meaning that if we find $\nu$ zero roots in the error location polynomial, then $\mu=t-\nu$ ($\nu$ error locations are ghost locations and so they do not correspond to actual errors).

 \item[(c)] Let us come back to the affine-variety case. The error vector is
$$
e=(\underbrace{0,\dots,0}_{\sigma_1-1},
\underset{\underset{P_{\sigma_1}}{\uparrow}}{e_{1}},
0,\dots,0,\underset{\underset{P_{\sigma_l}}{\uparrow}}{e_l}, 0,\dots,0,\underset{\underset{P_{\sigma_{\mu}}}{\uparrow}}{e_{\mu}},
\underbrace{0,\dots,0}_{n-1-\sigma_\mu}).
$$
The valid error locations are the points $P_{\sigma_l}=(x_{l,1},\dots,x_{l,m})$, $1\leq l\leq \mu$. The equation corresponding to (\ref{sin1}) is
\begin{equation}\label{sin1new}
s_{\rho}=\sum_{l=1}^{\mu} e_l b_{\rho}(P_{\sigma_l})=\sum_{l=1}^{\mu} e_l b_{\rho}(x_{l,1},\dots,x_{l,m}).
\end{equation}
We want a sum like (\ref{sin2}), something like
$s_{\rho}=\sum_{l=1}^{t} e_l b_{\rho}(P_{\sigma_l}).$
In order to do that, we would need $\sum_{l=\mu+1}^{t} e_l b_{\rho}(P_{\sigma_l})=0$, for some convenient \textit{ghost points} $\{P_{\sigma_l}\}_{\mu+1\leq l\leq t}$. 
Actually, we can use just one ghost point, that we call $P_0$. But it must \textit{not} lie on the variety, otherwise it could be confused with valid locations. In particular, we cannot hope to use always the ghost point $P_0=(x_{0,1},\dots,x_{0,m})=(0\ldots,0)$, since $(0\ldots,0)$ could be a point on the variety. For example, the Hermitian curve $\chi: x^{q+1}=y^q+y$ contains $(0,0)$ for any $q$.\\
Let $P_0$ be a ghost point. Not only we need to choose $P_0$ outside the variety, but we must also force $e_j=0$ for the error values in $P_0$, since we cannot hope that $b_{\rho}(P_{0})=0$ for each $\rho$. With these assumptions, we obtain
{\small\begin{equation}\label{eq6}
\begin{array}{lll}  s_{\rho}&=&\sum_{l=1}^{\mu}e_lb_{\rho}(x_{l,1},\dots,x_{l,m})+\sum_{l=\mu+1}^{t} e_l b_{\rho}(P_0)\\
&=&\sum_{l=1}^{\mu}e_lb_{\rho}(x_{l,1},\dots,x_{l,m})+\sum_{l=\mu+1}^{t} 0 \, b_{\rho}(P_0)\\
&=&\sum_{l=1}^{\mu}e_lb_{\rho}(x_{l,1},\dots,x_{l,m}) \,.
\end{array}
\end{equation}}

\item[(d)] For us a ghost point is any point $P_0\in (\mathbb{F}_q)^m\setminus \mathcal{V}(I)$. Depending on the variety, there can be clever ways to choose $P_0$.
\begin{definition}
Let $P_0=(\overline{x}_{0,1},\ldots,\overline{x}_{0,m})\in (\mathbb{F}_q)^m\setminus \mathcal{V}(I)$. We say that $P_0$ is an \textbf{optimal ghost point} if there is a $1\leq j\leq m$ such that the hyperplane $x_j=\overline{x}_{0,j}$ does not intersect the variety. We call $j$ the \textbf{ghost component}.
\end{definition}
\noindent In other words, for any optimal ghost point there is at least a component not shared with any variety point. See Figure \ref{gp} for an example.
\begin{center}
\begin{figure}[!htbp]
\includegraphics[width=14 cm]{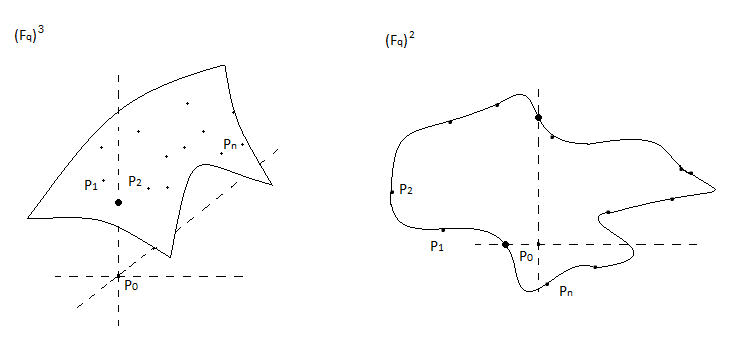}
\caption{In the first picture we have a optimal ghost point with two ghost components. In the second, a non-optimal ghost point.}\label{gp}
\end{figure} 
\end{center}
\begin{remark}
The advantage of using optimal ghost points is that it is enough to look at any ghost component in order to discard non-valid locations.
\end{remark}
\noindent If a curve is smooth and maximal (e.g., an Hermitian curve), it will probably intersect any hyperplane and so no optimal ghost point will exist in this case.

\item[(e)] We are ready to define a new ideal, summarising the above argument.\\
We start from equations (\ref{eq6}):
$$\left\{\sum_{j=1}^t e_j b_{\rho}
(x_{j,1},\dots,x_{j,m})-s_{\rho}\right\}_{1\leq \rho \leq r}
$$
\noindent We choose a ghost point $P_0=(x_{0,1},\dots,x_{0,m})\not\in \mathcal{V}(I)$. We need to find a generator set for the radical ideal $I'$ vanishing on $\mathcal{V}(I)\sqcup \{P_0\}$. The easiest way of doing this is to start from any \Gr\ basis $G$ of $I$ and to use the Buchberger-M\"{o}ller algorithm (see \cite{CGC-alg-art-buchmoeller,CGC-cd-inbook-D1morafglm}) to compute the \Gr\ basis $G'$ of $I'$. We will summarize the Buchberger-M\"{o}ller algorithm in Theorem \ref{teoBM}.  Let $G'=\{g'_h\}_{1\leq h\leq \gamma '}$. We can insert in our new ideal the following polynomials
$$
\left\{g'_h(x_{j,1},\dots,x_{j,m})\right\}_{\substack{1 \leq h \leq \gamma ',\\ 1\leq j \leq t}}
$$
\noindent In our new system we put $\{e_j^{q}=e_j\}$, because there can be zero values (corresponding to ghost locations).
We enforce $(x_{j,1},\dots,x_{j,m})\not=P_0$ for all $j$ corresponding to actual error locations. In order to do that, when $e_j\not=0$ we must have at least one component of $P_{\sigma_j}$ different from that of $P_0$, that is, $e_j\prod_{1 \leq \iota \leq m}\left((x_{j,\iota}-x_{0,\iota}\right)^{q-1}-1)=0$. 
So we can add
$$
\left\{e_j\prod_{1 \leq \iota \leq m}
                 ((x_{j,\iota}-x_{0,\iota})^{q-1}-1) \right\}_{\substack{1 \leq j \leq t}}.
$$

On the other hand, when $e_j=0$ we want  $(x_{j,1},\dots,x_{j,m})=P_0$. To enforce it, we add $$\left\{(e_j^{q-1}-1)(x_{j,\iota}-x_{0,\iota}) \right\}_{\substack{1\leq j \leq t,\\ 1\leq \iota \leq m}}.$$

Finally, if two points correspond to valid locations then they must be distinct. However, if at least one is a ghost point, then this requirement does not hold:$$\left\{e_je_k\prod_{1 \leq \iota \leq m}
                 ((x_{j,\iota}-x_{k,\iota})^{q-1}-1) \right\}_{\substack{1 \leq j < k\leq t}}.$$

We denote by $J_{*}^{C,t}$ the ideal in
$\FF_q[s_1,\dots,s_r,X_t, \dots,X_1,e_1,\dots,e_t]$,
with $X_1=\{x_{1,1},\dots,x_{1,m}\}, \dots, X_t=\{x_{t,1},\dots,x_{t,m}\}$ s.t.
\begin{center}
{\footnotesize\begin{equation}\label{fineideal}
\begin{array}{ll}
J_{*}^{C,t} = \Big\langle &\left\{\sum_{j=1}^t e_j b_{\rho}
(x_{j,1},\dots,x_{j,m})-s_{\rho}\right\}_{1\leq \rho \leq r}, \left\{e_j^{q}-e_j\right\}_{1\leq j  \leq t},\\
&    \left\{g'_h(x_{j,1},\dots,x_{j,m})\right\}_{\substack{1 \leq h \leq \gamma ',\\ 1\leq j \leq t}}, \left\{(e_j^{q-1}-1)(x_{j,\iota}-x_{0,\iota}) \right\}_{\substack{1\leq j \leq t,\\ 1\leq \iota \leq m}}, \\
&  \left\{e_j\prod_{1 \leq \iota \leq m}
                 ((x_{j,\iota}-x_{0,\iota})^{q-1}-1) \right\}_{\substack{1 \leq j \leq t}}, \\
&        
\left\{e_je_k\prod_{1 \leq \iota \leq m}
                 ((x_{j,\iota}-x_{k,\iota})^{q-1}-1) \right\}_{\substack{1 \leq j < k
                 \leq t}} \Big\rangle .
\end{array}
\end{equation}}
\end{center}
Since $I'=\langle\left\{g'_h\right\}_{\substack{1 \leq h \leq H}}\rangle$ contains the field equations, we may add them to reduce the computation of the \Gr\ basis of $J_{*}^{C,t}$.
\end{itemize}

\subsection{Weak locator polynomials}\label{Locdebole}

We would like to define some locator polynomials for affine-variety
codes that play the same role as those in Definition \ref{zero}. 
We would expect to find them in our ideal (\ref{fineideal}). These locators might look like 
\begin{equation}\label{polocSbagliato}
    {\mathcal{L}}_{i}(S,x_1,\ldots,x_i)= x_i^{t}+a_{t-1}x_i^{t-1}+\cdots +a_0,
\end{equation}
  with $a_j \in \FF_q[S,x_1,\ldots,x_{i-1}]$, $0 \leq j \leq t-1$, that is,
  ${\mathcal{L}}_{i}$ is a monic polynomial with degree $t$ with respect
  to the variable $x_i$ and its coefficients are in $\FF_q[S,x_1,\ldots,x_{i-1}]$. We would also want the following property.\\
Given a syndrome
  ${\bf s}=(\bar{s}_1,\dots \bar{s}_r)\in (\FF_{q})^{r}$,
  corresponding to an error vector of weight
  $\mu\leq t$ and $\mu$ error locations
  $(\bar{x}_{1,1},\ldots,\bar{x}_{1,m})$ $,\ldots,
   (\bar{x}_{\mu,1},\ldots,\bar{x}_{\mu,m})$, if we evaluate the $S$ variables at ${\bf s}$ and the variables $x_1,\ldots,x_{i-1}$ at $\bar{x}_{j,1},\ldots,\bar{x}_{j,i-1}$  for any $1\leq j\leq \mu$,  then the roots of
  ${\mathcal{ L}}_{i}({\bf s},\bar{x}_{j,1},\ldots,\bar{x}_{j,i-1},x_i)$ are either   $\{ \bar{x}_{1,i},\ldots,\bar{x}_{t,i}\}$, when $\mu=t$, or  $\{\bar{x}_{1,i},\ldots,\bar{x}_{\mu,i},\bar{x}_{0,i}\}$, when $\mu\leq t-1$.
Apart from the actual  location components and possibly the ghost component, polynomial  ${\mathcal{ L}}_{i}$ should not have other solutions.\\

To show that a polynomial of this kind does not necessarily  exist in $J_{*}^{C,t}$, we consider the following examples.
\begin{example}\label{Degenere1}
Let us consider an MDS code
$C=C^{\perp}(I,L)$ $[5,1,5]$ from the plane curve $\{y^5-y^4+y^3-y^2+y-x=0\} \cap \{x-1=0\}$ over $\FF_7$ and
with {\footnotesize{$$L=\{ y-3,y^2-1,y^3+3,y^4-1\},\, \mathcal{V}(I)=\{(1,1),\,(1,2),\,(1,3),\,(1,4),\,(1,5)\}.$$}}
\hspace{-0.3cm}
\noindent It is easy to see that $C$ can correct up to
$t=2$ errors. Let us consider the lex term-ordering with
$s_1<s_2<s_3<s_4<s_5<x_2<y_2<x_1<y_1<e_2<e_1$ in
$\FF_7[s_1,s_2,s_3,s_4,s_5,x_2,y_2,x_1,y_1,e_1,e_2]$ . Ideal $J_{*}^{C,t}$ is generated by
{\scriptsize
\begin{align*}
&\langle e_1^7-e_1,e_2^7-e_2,x_1-1,x_2-1,y_1^6-y_1^5+y_1^4-y_1^3+y_1^2-y_1,
y_2^6-y_2^5+y_2^4-y_2^3+y_2^2-y_2,\\
&
e_1(-y_1^4+y_1^3+y_1^2-2y_1+2)+e_2(-y_2^4+y_2^3+y_2^2-2y_2+2)-s_1, e_2((x_2-1)^6-1)(y_2^6-1),\\
&
e_1(3y_1^4-2y_1^3+3y_1^2+3y_1)+e_2(3y_2^4-2y_2^3+3y_2^2+3y_2)-s_2,
e_1(3y_1^4-y_1^2-2)+e_2(3y_2^4-y_2^2-2)-s_3,\\
&
e_1(-y_1^4+2y_1^3-y_1^2-3y_1+3)+e_2(-y_2^4+2y_2^3-y_2^2-3y_2+3)-s_4, e_1((x_1-1)^6-1)(y_1^6-1),\\
&
e_1e_2((x_1-x_2)^6-1)((y_1-y_2)^6-1),(e_2^6-1)(x_2-1),(e_2^6-1)y_2,(e_1^6-1)(x_1-1),(e_1^6-1)y_1
 \rangle,
\end{align*}
}
\hspace{-0.2cm}
\noindent where the ghost point is $P_0=(1,0)$.
The  reduced \Gr \ basis $G$ with respect to 
$s_1<s_2<s_3<s_4<s_5<x_2<y_2<x_1<y_1<e_2<e_1$
has $27$ elements and the \textit{ new locators} are $\mathcal{L}_{1}(s_1,\ldots,s_5,x_2)=\mathcal{L}_{x}$ and $\mathcal{L}_{2}(s_1,\ldots,s_5,x_2,y_2)=\mathcal{L}_{xy}$ (see Appendix for 
polynomials $a$ and $b$):
{\footnotesize{
$$
\mathcal{L}_{x}= \mathbf{x}-1\quad\textrm{ and }\quad
\mathcal{L}_{xy}=\mathbf{y}^2+\mathbf{y}a+b.
$$}}
\hspace{-0.2cm}
\noindent We can note that $\mathcal{L}_{x}$ does not play any role, because all $x$'s are equal to $1$. So to apply the decoding we evaluate only $\mathcal{L}_{xy}$ at $\overline{\textbf{s}}$ and we expect to obtain the (second) components of error locations. We show it in two cases:
\begin{itemize}
    \item[-] We suppose that two errors occur at the points {\footnotesize{$P_1=(1,1)$}} and {\footnotesize{$P_2=(1,2)$}}, both with  error values $1$, so the syndrome vector
corresponding to the error vector 
$(1, 1,0,0,0)$ is  $\overline{\mathbf{s}}=(2,1,0,0)$.\\
In order to find the error positions we evaluate  $\mathcal{L}_{xy}$ in $\overline{\mathbf{s}}$. We obtain two different solutions
$
\mathcal{L}_{xy}(\overline{\mathbf{s}},y)= y^2-3y+2=(y-2)(y-1),
$
that identify the two error locations.

\item[-] We consider ${\overline{\mathbf{s}}}=(0,4,4,0,1)$ corresponding to 
$(0,0,0,4,0)$, so only one error occurs in the point $(1,3)$. Evaluating
$\mathcal{L}_{xy}$ at $\overline{\mathbf{s}}$ we obtain
$
\mathcal{L}_{xy}(\overline{\mathbf{s}},y)=y^2-3y=y(y-3).
$
\noindent Also in this case we obtain a correct solutions ($0$ is the ghost component).  So the above choice of $\mathcal{L}_{x}$ and $\mathcal{L}_{xy}$ seems correct.
\end{itemize}

\end{example}
Now we consider the above code but with a different ghost point. Also in the following example, we take an optimal ghost point. 
\begin{example}\label{Degenere2} 
Let us consider the same MDS code
$C=C^{\perp}(I,L)$ as in Example
\ref{Degenere1}.
In this example we choose the (optimal) ghost point  $P_0=(0,0)$. The ideal $J_{*}^{C,t}$ is generated by
{\scriptsize
\begin{align*}
&\langle e_1^7-e_1,e_2^7-e_2,x_1y_1-y_1,x_2y_2-y_2,x_1^2-x_1,x_2^2-x_2,y_1^6-y_1^5+y_1^4-y_1^3+y_1^2-y_1,\\
&
y_2^6-y_2^5+y_2^4-y_2^3+y_2^2-y_2,e_1(-y_1^4+y_1^3+y_1^2-2y_1+2)+e_2(-y_2^4+y_2^3+y_2^2-2y_2+2)-s_1,\\
&
e_1(3y_1^4-2y_1^3+3y_1^2+3y_1)+e_2(3y_2^4-2y_2^3+3y_2^2+3y_2)-s_2,
e_1(3y_1^4-y_1^2-2)+e_2(3y_2^4-y_2^2-2)-s_3,\\
&
e_1(-y_1^4+2y_1^3-y_1^2-3y_1+3)+e_2(-y_2^4+2y_2^3-y_2^2-3y_2+3)-s_4, e_1(x_1^6-1)(y_1^6-1),\\
&
e_2(x_2^6-1)(y_2^6-1),e_1e_2((x_1-x_2)^6-1)((y_1-y_2)^6-1),(e_2^6-1)x_2,
(e_2^6-1)y_2,(e_1^6-1)x_1,(e_1^6-1)y_1.
 \rangle,
\end{align*}
}
\hspace{-0.25cm}
\noindent
The  reduced \Gr \ basis $G$ with respect to 
{\footnotesize{$s_1<s_2<s_3<s_4<s_5<x_2<y_2<x_1<y_1<e_2<e_1$}}
has $27$ elements and the \textit{ new locators} are {\footnotesize{$\mathcal{L}_{1}(\mathcal S,x_2)=\mathcal{L}_{x}$}} and {\footnotesize{$\mathcal{L}_{2}(\mathcal S,x_2,y_2)=\mathcal{L}_{xy}$}}, where $\mathcal S=\{s_1,\ldots,s_5\}$ (see Appendix for $c$ and $d$):
{\footnotesize{
\begin{eqnarray}\label{poldeg2a}
 \mathcal{L}_{x}= \mathbf{x}^2-\mathbf{x} \quad\textrm{ and }\quad
\mathcal{L}_{xy}= \mathbf{y}^2+\mathbf{y}c+d. 
\end{eqnarray}
}}
\noindent Also in this case $\mathcal{L}_{x}$ does not depend on any syndrome, so to apply the decoding  we just specialize $\mathcal{L}_{xy}(\overline{\textbf{s}},\overline{\textbf{x}},y)$. We would like that the solutions of $\mathcal{L}_{xy}(\overline{\textbf{s}},\overline{\textbf{x}},y)=0$  are exactly the second components of error locations, but this is not always the case.  Let us consider the same errors as in Example \ref{Degenere1}: 
\begin{itemize}
    \item[-] We suppose that two errors occur at the points $P_1=(1,1)$ and $P_2=(1,2)$, with both error values $1$, so the syndrome vector
corresponding to the error vector 
$(1, 1,0,0,0)$ is  $\overline{\mathbf{s}}=(2,1,0,0)$.
In order to find the error positions we evaluate  $\mathcal{L}_{xy}$ in $\overline{\mathbf{s}}$. We obtain three different solutions
{\footnotesize{
$$\begin{array}{l}
\mathcal{L}_{xy}(\overline{\mathbf{s}},1,y)= y^2-3y+2=(y-1)(y-2),\\
\mathcal{L}_{xy}(\overline{\mathbf{s}},0,y)= y^2-3y-3=(y+2)^2.
\end{array}$$}}
In this case, we are lucky, because $(0,5)$ is not a point coordinate and so we can discard $y=5$ finding the two error locations.
\item[-] We consider ${\overline{\mathbf{s}}}=(0,4,4,0,1)$ corresponding to 
$(0,0,0,4,0)$, so only one error occurs in the point $(1,3)$. Evaluating
$\mathcal{L}_{xy}$ in $(\overline{\mathbf{s}})$ we obtain
{\footnotesize{
$$\begin{array}{l}
\mathcal{L}_{xy}(\overline{\mathbf{s}},1,y)= y^2-y+1=(y-3)(y+2),\\
\mathcal{L}_{xy}(\overline{\mathbf{s}},0,y)= y^2-y=y(y-1).
\end{array}$$
}}
\noindent In this case we have four possible solutions $(1,3)$, $(1,5)$,$(0,0)$ and $(0,1)$. but only three are acceptable, which are $(1,3)$, $(1,5)$ and $(0,0)$.  To individuate those corresponding to the  syndrome vector $\overline{\mathbf{s}}$, we must compute the two syndromes and we will see that $(1,3)$ and $(0,0)$, are  correct. In this case, the above choice of $\mathcal{L}_{x}$ and $\mathcal{L}_{xy}$ is unfortunate.
\end{itemize}

\end{example}

One might think that the unpleasant behaviour of  \eqref{poldeg2a}  is due to the degenerate geometric situation. Unfortunately, this is not entirely true, as next example shows (we end this long example with a horizontal line).
\begin{example}\label{exhq2} 
Let us consider  the Hermitian code
$C=C^{\perp}(I,L)$ from the curve $y^2+y=x^3$ over $\FF_4$ and
with defining monomials $\{1,x,y,x^2,x y\}$, as in Example
\ref{ex:hermitiano}.
It is well-known that $C$ can correct up to
$t=2$ errors. Let us consider the lex term-ordering with
$s_1<\ldots<x_2<y_2<x_1<y_1<e_2<e_1$ in
$\FF_4[s_1,s_2,s_3,s_4,s_5,x_2,y_2,x_1,y_1,e_1,e_2]$. 
Ideal $J_{{\mathcal{FL}}}^{C,t}$ is 

\begin{scriptsize}
$$
\begin{array}{l}
\langle x_1^4-x_1,y_1^4-y_1, x_2^4-x_2,y_2^4-y_2, e_1^3-1, e_2^3-1, y_1^2+y_1-x_1^3, y_2^2+y_2-x_2^3,\\
e_1+e_2-s_1,\;\; e_1x_1+e_2x_2-s_2,\;\; e_1y_1+e_2y_2-s_3,\;\; e_1x_1^2+e_2x_2^2-s_4, \\
 e_1x_1y_1+e_2x_2y_2-s_5 \rangle, 
\end{array}
$$
\end{scriptsize}

\noindent and the  reduced \Gr \ basis $G$ (with  respect to $<$) has $53$
elements.\\
\noindent The authors of \cite{CGC-cd-art-lax} report $119$
polynomials because they do not use lex but a block order, which
is faster to compute but which usually possesses larger \Gr\
bases.
In $G\cap (\FF_4[S,x_2]\setminus \FF_4[S])$ there are $5$
polynomials of degree $2$ in $x_2$ and these are our candidate
polynomials:
\vspace{-0.2 cm}
\begin{scriptsize}
$$
\begin{array}{ll}
g_5=& \mathbf{x_2}^2 s_5+\mathbf{x_2}
    (s_5s_4s_2^2+s_4^2s_3^2s_2s_1^2+s_4^2s_2s_1+s_4s_3^2s_1+s_4s_3s_2^3s_1^2+
    s_4s_3s_1^2+s_4s_1^3+s_3^2s_2^2s_1^3+s_2^2s_1^2)+\\
&   s_5^3s_3+s_5s_4^2s_3^3s_2+s_5s_4^2s_2+s_4^3s_3^3s_2^3s_1+s_4^3s_3^3s_1+
s_4^3s_3^2s_2^3s_1^2+
    s_4^3s_3s_2^3+s_4^3s_1+s_4^2s_3^3s_2^2+s_4^2s_3^2s_2^2s_1+\\
&
s_4s_3^2s_2+s_4s_2s_1^2+s_3^3s_2^3s_1+s_3s_2^3s_1^3+s_3s_2^3+s_2^3s_1 \\
g_4=& \mathbf{x_2}^2 s_4+\mathbf{x_2}
       (s_4^2s_2^2+s_2^3s_1+s_1+s_4^2s_3^2s_1^3)+s_4^2s_3^2+
        s_4^2s_2^3s_1^2+s_4^2s_1^2+s_4s_3s_2^2s_1^3+s_4s_3s_2^2+s_2s_1^3+s_2 \\
g_3=& \mathbf{x_2}^2 s_3+\mathbf{x_2}
       (s_4^2s_3s_1+s_4s_3s_2^2s_1^3+s_4s_3s_2^2+s_3s_2s_1^2)+
        s_5^2 s_3^2+s_5s_3^2s_2+s_4^2s_3^3s_2s_1+s_4^2s_3^2s_2s_1^2+s_4s_3^3s_1^3+\\
&   
s_4s_3^2s_1+s_3^3s_2^2s_1^2+s_3^2s_2^2s_1^3+s_3^2s_2^2\\
g_2=& \mathbf{x_2}^2 s_2+\mathbf{x_2}
       (s_4^2s_2s_1+s_4s_1^3+s_4+s_2^2s_1^2)+s_4^2s_2^2+s_4s_3^2s_2s_1^3+
       s_4s_3^2s_2+s_4s_2s_1^2+s_3s_2^3s_1^3+s_3s_2^3+s_2^3s_1\\
g_1=& \mathbf{x_2}^2(s_1)+\mathbf{x_2}
       (s_4^2s_1^2+s_2s_1^3)+s_4^2s_2s_1+s_4s_1^3+s_2^2s_1^2
\end{array}
$$
\end{scriptsize}
\vspace{-0.2 cm}

\noindent   Of course,
there are other similar polynomials in $J_{\mathcal{FL}}^{C,t}\cap
(\FF_4[S,x_2]\setminus \FF_4[S])$ and they may be found for
example by computing \Gr \ bases with respect to other orderings. It
is immediate that the leading polynomials are just
$\{s_1,\ldots,s_5\}$. Suppose that we receive a syndrome ${\bf
s}=(\bar{s}_1,\ldots,\bar{s}_5)$. If it is zero, then no errors
occurred. Otherwise, we might follow the most obvious way to
correct, that is, we might substitute ${\bf s}$ in all $g_i$'s,
until we find one which does not vanish identically.
\noindent The improvement introduced by Caboara and Mora translates here in
checking only the leading polynomials, i.e. checking which of the
syndrome components $ \bar{s}_i$ is non-zero. Since clearly at
least one is non-zero, with a negligible computational effort we
are able to determine the right candidate.\\
\noindent Let us now follow our proposal. Ideal $J_{*}^{C,t}$ is generated by

\vspace{-0.2 cm}
{\scriptsize{
$$
\begin{array}{l}
\{x_1^4-x_1,\,  
y_1^4-y_1, \,
x_2^4-x_2, \,
y_2^4-y_2, \,
e_1^4-e_1, \,
e_2^4-e_2, \,
y_1^2x_1+y_1^2+
y_1x_1+y_1+x_1^3+\\
 x_1, y_2^2x_2+y_2^2+y_2x_2+y_2+x_2^3+x_2,
y_1^3+y_1x_1^3+y_1+x_1^3,y_2^3+y_2x_2^3+y_2+x_2^3, e_1+\\
e_2-s_1,\,
e_1x_1+e_2x_2-s_2,\,  
e_1y_1+e_2y_2-s_3, \,
e_1x_1^2+e_2x_2^2-s_4,e_1x_1y_1+e_2x_2y_2-s_5,\\
e_1((x_1-1)^3-1)((y_1-1)^3-1), e_2((x_2-1)^3-1)((y_2-1)^3-1),
(e_1^3-1)(x_1-1),\\
(e_1^3-1)(y_1-1),  
(e_2^3-1)(x_2-1),
(e_2^3-1)(y_2-1),
e_1e_2((x_1-x_2)^3-1)((y_1-y_2)^3-1)
\}.
\end{array}
$$}}

\noindent where the ghost point is $(1,1)$ (note that $1^3\not=1^2+1$).\\
\noindent The  reduced \Gr \ basis $G$ with respect to 
$s_1<s_2<s_3<s_4<s_5<x_2<y_2<x_1<y_1<e_2<e_1$
has $32$ elements and the \textit{ new locators} are $\mathcal{L}_{1}(s_1,\ldots,s_5,x_2)=\mathcal{L}_{x}$ and $\mathcal{L}_{2}(s_1,\ldots,s_5,x_2,y_2)=\mathcal{L}_{xy}$, that are the polynomials of degree two in, respectively, $x_2$ and $y_2$:

{\scriptsize{
$$
\begin{array}{ll}
\mathcal{L}_{x}=& \mathbf{x}^2+
\mathbf{x}(s_1^2s_2s_4^3+s_4^3+s_1s_2^3s_4^2+s_1^2s_2^2s_4^2+s_1s_4^2+s_2^2s_4+
s_1s_2s_4+s_2^3+s_1^2s_2+s_1^3)+\\
&s_3s_5^2+s_2s_3s_5+s_1s_2^2s_4^3+s_1^2s_2s_4^3+s_2s_3^3s_4^2+s_1s_2s_3^2s_4^2+
s_1^2s_2s_3s_4^2+s_1s_2^3s_4^2+ s_1^3s_2s_4^2+\\
&s_2s_4^2+s_1^2s_3^3s_4+s_1^3s_3^2s_4+s_1s_3s_4+s_1^2s_2^3s_4+s_1^3
s_2^2s_4+s_1^2s_4+s_1^3s_2^3s_3^3+s_2^3s_3^3+s_1s_2^2s_3^3+\\
&s_1^3s_3^3+ s_3^3+s_1^2s_2^2s_3^2+
s_1^3s_2^2s_3+s_2^2s_3+s_1^3s_2^3+s_2^3+s_1s_2^2+s_1^3+1\\
\\
\mathcal{L}_{xy}=&\mathbf{y}^2+\mathbf{y}(s_3^3+s_1s_3^2+s_1^2s_2^3s_3+s_1^2s_3+s_1^3)+
\mathbf{x}(s_2^2s_3s_4^3+s_1s_2^2s_4^3+s_1^2s_2s_3s_4^2+s_1^2s_3^3s_4+s_3^2s_4+\\
&s_1s_3s_4
+s_1^2s_2^3s_4)+s_5^3+s_2s_3^2s_4^2s_5+s_3s_4s_5+s_2^2s_5+s_3^3s_4^3+s_1s_2^3s_3^2s_4^3+
s_2^3s_4^3+s_1^2s_2^2s_3^3s_4^2+\\
&s_1^2s_2s_3^2s_4+s_1^3s_2s_3s_4+s_1s_2s_4+s_2^3s_3^3+s_3^3+s_1s_2^3s_3^2+s_1s_3^2
+s_1^2s_2^3s_3+s_1^2s_3+s_1^3s_2^3+s_1^3+1
\end{array}
$$}} 

\noindent We can apply the decoding in this way: we specialize $\mathcal{L}_x(s,x)$ to ${\bf \overline{s}}$ for any received syndrome. If the syndrome
corresponds to two errors, then we expect that  the roots of $\mathcal{L}_{x}({\bf \overline{s}},x)$ are the first components of error locations and the roots of $\mathcal{L}_{xy}(\overline{\textbf{s}},\overline{\textbf{x}},y)$ are exactly the second components of error locations. But it is not always true, we show it in three cases: 
\begin{itemize}
    \item[-] We suppose that two errors occur at the points $P_6=(\alpha,\alpha+1)$ and $P_7=(\alpha+1,\alpha)$, with both error values $1$, so the syndrome vector
corresponding to the error vector 
$(0,0,0,0,0, 1, 1, 0)$ is  $\overline{\mathbf{s}}=(0,1,1,1,0)$.\\
In order to find the error positions we evaluate 
$\mathcal{L}_{x}$ in $\overline{\mathbf{s}}$ and we obtain the correct values of $x$, in fact:
\vspace{-0.2 cm}
{\small{$$
\begin{array}{l}
\mathcal{L}_{x}(\overline{\mathbf{s}}, x)= x^2+x+1=(x-\alpha)(x-(\alpha+1)).
\end{array}
$$}}
Now we have to  evaluate  $\mathcal{L}_{xy}$ in $(\overline{\mathbf{s}},\overline{\mathbf{x}})$. We obtain four different solutions
{\small{$$
\begin{array}{l}
\mathcal{L}_{xy}(\overline{\mathbf{s}},\alpha,y)= y^2+y+1=(y-\alpha)(y-(\alpha+1))\\
\mathcal{L}_{xy}(\overline{\mathbf{s}},\alpha+1,y)= y^2+y+1=(y-\alpha)(y-(\alpha+1)).
\end{array}
$$}}
But this is a \textit{problem} for us, because all these solutions are curve points: $(\alpha,\alpha)$,$(\alpha,\alpha+1)$,$(\alpha+1,\alpha)$,$(\alpha+1,\alpha+1)$.
Only two are the correct locations. 
To individuate those corresponding to the  syndrome vector $\overline{\mathbf{s}}$, we must compute the two syndromes and we will see that  $(\alpha+1,\alpha), (\alpha,\alpha+1)$ are  correct. This method of try-and-see works nice because the code is small, but soon it becomes unfeasible. So the above choice of $\mathcal{L}_{x}$ and $\mathcal{L}_{xy}$ is unfortunate.
 \item[-] We suppose that the syndrome is  $(\alpha+1,0,\alpha,0,0)$,
corresponding to the error vector 
$(1,\alpha,0,0,0,0, 0, 0)$. So two errors have occurred and their values are $1$ and $\alpha$ in the point, respectively, $P_1=(0,0)$ and $P_2=(0,1)$.
In order to find the error locations we evaluate 
$\mathcal{L}_{x}$ in ${\overline{\mathbf{s}}}$ and we obtain
$\mathcal{L}_{x}(\overline{\mathbf{s}},x)=x^2+x=x(x-1)$,
then we evaluate 
$\mathcal{L}_{xy}$ in $(\overline{\mathbf{s}},0)$ and $(\overline{\mathbf{s}},1)$ and we get
$\mathcal{L}_{xy}(\overline{\mathbf{s}},0,y)=\mathcal{L}_{xy}(
\overline{\mathbf{s}},1,y)=y^2+y=y(y-1)$.
The equations 

\begin{equation}\label{lx0}
    \mathcal{L}_{x}(\overline{\mathbf{s}},x)
=\mathcal{L}_{xy}(\overline{\mathbf{s}},1,y)
=\mathcal{L}_{xy}(\overline{\mathbf{s}},0,y)=0
\end{equation} 
have four possible solutions: $(0,0),
(0,1)$, $(1,0)$ and $(1,1)$. Since the points $(1,0)$ and $(1,1)$ do not lie on the Hermitian curve, then only one solution couple is admissible: $\{(0,0), (0,1)\}$.
This situation is better than the above case, because we can immediately understand what  the correct solutions of system (\ref{lx0}) are. This happens by chance and in any case the solutions of equation $\mathcal{L}_{x}(\overline{\mathbf{s}})=0$ are not what we want.
    \item[-] Finally we consider ${\overline{\mathbf{s}}}=(\alpha+1,\alpha+1,1,\alpha+1,1)$ corresponding to 
$(0,0,\alpha+1,0,0,0,0,0)$, so only one error occurs. Evaluating
$\mathcal{L}_{x}$  and $\mathcal{L}_{xy}$, respectively, in $\overline{\mathbf{s}}$ and $(\overline{\mathbf{s}},\overline{\mathbf{x}})$, we obtain
\pagebreak
\begin{equation}\label{lxy0}
   \left\{\begin{array}{l}
 \mathcal{L}_{x}(\overline{\mathbf{s}},x)=x^2+1\\
\mathcal{L}_{xy}(\overline{\mathbf{s}},1,y)=y^2+(\alpha+1)y+\alpha=(y-1)(y-\alpha).   
\end{array}\right.
\end{equation}
In this case we are extremely lucky because the two polynomials 
$\mathcal{L}_{x}$  and $\mathcal{L}_{xy}$ answer correctly: the solutions of  system (\ref{lxy0}) are $(1,1)$, which is the ghost point, and $(1,\alpha)$, which is the error location.
\end{itemize}
\end{example}
\hrule
$\,$\\
 \begin{remark}\label{basefield}
Since, in Example \ref{exhq2}, the curve equation has all coefficients in
$\FF_2$, the ideal $J_{\mathcal{FL}}^{C,t}$ actually lies in
$\FF_2[s_1,s_2,s_3,s_4,s_5,x_2,y_2,x_1,y_1,e_1,e_2]$.
This is a special case of a more general fact: for any affine 
variety-code and any decoding ideal
that we are considering in the whole paper, 
all polynomials defining these ideals have no coefficient different
from $\{1,-1\}$, except possibly for the polynomials defining $I$.
Therefore, if it is possible
to have a basis for the ideal $I$ with all coefficients in a smaller
field, then any of its \Gr\ bases will have elements
with the same coefficient field, which means that the basis computation will be much faster.
\end{remark}

Since polynomials like $\mathcal{L}_{x}$  and $\mathcal{L}_{xy}$ in Example \ref{exhq2} contain the right solutions (together with unwelcome parasite solutions), they deserve a definition. See Section \ref{2.1} for our notation.
\begin{definition}\label{locDebole}
Let $C=C^{\perp}(I,L)$ be an affine-variety code. Let $1\leq
i\leq m$. \\  Let $P_0=(\bar x_{0,1},\ldots,\bar x_{0,m})\in (\mathbb{F}_q)^m\setminus \mathcal{V}(I)$ be a ghost point. Let 
$$
 t_i=\min\left\{t\,,\,|\{\hat{\pi}_i(P)\mid P\in \mathcal{V}(I)\cup P_0\}|\right\},
$$ 
and let ${\mathcal{P}}_{i}$ be a polynomial in
$\FF_q[S,x_1,\ldots,x_{i}]$, where $S=\{s_1,\dots,s_r\}$. 
Then $\{{\mathcal{ P}}_{i}\}_{\substack{1 \leq i \leq m}}$ is a set of  {\bf weak multi-dimensional general error locator polynomials} of $C$ if for any $i$
\begin{itemize}
\item
${\mathcal{P}}_{i}(S,x_1,\ldots,x_i)= x_i^{t_i}+a_{t_i-1}x_i^{t_i-1}+\cdots +a_0$,
  with $a_j \in \FF_q[S,x_1,\ldots,x_{i-1}]$, $0 \leq j \leq t_i-1$, that is,
  ${\mathcal{P}}_{i}$ is a monic polynomial with degree $t_i$ with respect
  to the variable $x_i$ and its coefficients are in $\FF_q[S,x_1,\ldots,x_{i-1}]$;
\item given a syndrome
  $\overline{{\bf s}}=(\bar{s}_1,\dots \bar{s}_r)\in (\FF_{q})^{r}$,
  corresponding to an error vector of weight
  $\mu\leq t$, $\mu$ error locations
  $(\bar{x}_{1,1},\ldots,\bar{x}_{1,m})$ $,\ldots,
   (\bar{x}_{\mu,1},\ldots,\bar{x}_{\mu,m})$. 
  If we evaluate the $S$ variables at $\overline{{\bf s}}$ and the variables $(x_1,\ldots,x_{i-1})$ at the truncated vectors $\overline{\mathbf{x}}^j=(\bar{x}_{j,1},\ldots,\bar{x}_{j,i-1})$ for $0\leq j\leq \mu$,  then the roots of
  ${\mathcal{ P}}_{i}(\overline{{\bf s}},\overline{\mathbf{x}}^j,x_i)$ contain:
\begin{itemize}
    \item[-] either  $\{ \bar{x}_{h,i}\mid \overline{\mathbf{x}}^h=\overline{\mathbf{x}}^j, \, 0\le h\le \mu\}$ (when $\mu< t$), 
    \item[-]  or  $\{ \bar{x}_{h,i}\mid \overline{\mathbf{x}}^h=\overline{\mathbf{x}}^j, \, 1\le h\le \mu\}$ (when $\mu=t$), 
\end{itemize}
 plus possibly  some  parasite solutions. 
\end{itemize}
\end{definition}
\noindent Note that the difference between  $\{ \bar{x}_{h,i}\mid \overline{\mathbf{x}}^h=\overline{\mathbf{x}}^j, \, 0\le h\le \mu\}$ and $\{ \bar{x}_{h,i}\mid \overline{\mathbf{x}}^h=\overline{\mathbf{x}}^j, \, 1\le h\le \mu\}$ is that the latter set does not consider the ghost point.\\

Now we consider an alternative strategy to calculate the error locations, using the weak
multi-dimensional general error locator polynomials and some other polynomials in  ideal $J_{*}^{C,t}$.

Since it is convenient to know in advance the error number and the error values, we provide the following definition for a general correctable linear code. Let $C$ be an $[n,k,d]$ linear code over $\FF_q$ with correction
capability $t \geq 1$. Choose any  parity-check matrix with
entries in an appropriate extension field $\FF_{q^M}$ of $\FF_q$,
$M \geq 1$. 
Its  syndromes lie in $(\FF_{q^M})^{n-k}$ and form a
vector space of dimension $r=n-k$ over $\FF_q$. 
\begin{definition}\label{eval}
Let $\mathcal{E} \in \FF_q[S,e]$, where $S=\{s_1,\dots,s_r\}$. Then
$\mathcal{E}$ is a  {\bf general error evaluator polynomial} of $C$ if
\begin{itemize}
\item
$\mathcal{E}(S,e)= a_te^{t}+a_{t-1}e^{t-1}+
\cdots +a_0$,
  with $a_j \in \FF_q[S]$, $0 \leq j \leq t$, that is,
  $\mathcal{E}$ is a polynomial with degree $t$ with respect
  to the variable $e$ and its coefficients are in $\FF_q[S]$;
\item Given a syndrome
  $\overline{{\bf s}}=(\bar{s}_1,\dots \bar{s}_r)\in (\FF_{q^M})^{r}$ corresponding to an error vector of weight
  $\mu \leq t$ and with $\bar{e}_1,\ldots,\bar{e}_{\mu}$ as error values, we evaluate the $S$ variables at $\overline{{\bf s}}$, then the roots of
  $\mathcal{E}$ are  the error values  $\bar{e}_1,\ldots, \bar{e}_{\mu}$ plus $0$ with multiplicity $t-\mu$.
\end{itemize}
\end{definition}
\noindent The importance of  $\mathcal{E}$ lies in fact that the error number is $\mu$ if and only if
$$e^{t-\mu}\vert\mathcal{E}(\overline{\mathbf{s}})\quad\textrm{ and }\quad e^{(t-\mu+1)}\not\vert\mathcal{E}(\overline{\mathbf{s}}).$$ 
The ideal $J_{*}^{C,t}\,\cap \,\mathbb{K}[S,e_1,\ldots,e_t]$ is easily seen to be stratified, as follows. There is
a bijective correspondence between correctable syndromes and
correctable errors (i.e., errors of weight $\tau\leq t$) and so if we
fix $1\leq l \leq t$ and $1\leq s\leq t-l$
we can always find $l$ error values $e_1,\ldots,e_l$ that have $s$
extensions at level $e_{l+1}$.
So we can apply Proposition \ref{kk} and obtain the existence of $\mathcal{E}$:
\begin{theorem}
For any affine-variety code $C=C^{\perp}(I,L)$, the general  error evaluator polynomial exists.
\begin{proof}
We apply  Proposition \ref{kk} to the stratified ideal $J_{*}^{C,t}\,\cap \,\mathbb{K}[S,e_1,\ldots,e_t]$. It is enough to take $ g $ with $ {\bf T}(g)={\sf a}_L^L$
with $\mathcal{A}=\{e_1,\ldots,e_t\}$ and $\mathcal{S}=S$.
\end{proof}
\end{theorem}
\noindent  Using  $\mathcal{E}$, we know not only $\tau$, but also the $\tau$ error values. In order to exploit this information, we can consider a straightforward generalisation of weak
multi-dimensional general error locator polynomials (see Definition \ref{locDebole}) where the locators are actually ${\mathcal{P}}_{i}^e\in\FF_q[S,e,x_1,\ldots,x_{i-1}]$. We do not give a long definition for these, since we think it is obvious.

\noindent We consider again Example \ref{exhq2} to show two alternative strategies.
\begin{example}\label{EXHerm1}
Let us consider  the Hermitian code
$C=C^{\perp}(I,L)$ from the curve $y^2+y=x^3$ over $\FF_4$ and
with defining monomials $\{1,x,y,x^2,x y\}$, as in the Example
\ref{exhq2}.  
The  reduced \Gr \ basis $G$ of $J_{*}^{C,t}$ with respect to lex with
$s_1<s_2<s_3<s_4<s_5<e_2<e_1<x_2<y_2<x_1<y_1$
 has $33$
elements and the general error evaluator polynomial ${\mathcal E}$ is
{\scriptsize{
\begin{align*}
{\mathcal E}= & \mathbf{e}^2+\mathbf{e}s_1+s_4^3s_3^2+s_4^3s_3s_1+s_4^3s_2^3s_1^2+s_4^3s_1^2+
s_4^2s_3^2s_2^2s_1^2+s_4^2s_3s_2^2s_1^3+s_4s_3^2s_2s_1+\\
 &s_4s_3s_2s_1^2+s_4s_2s_1^3+
s_4s_2+s_3^2s_2^3s_1^3+s_3^2+s_3s_2^3s_1+s_3s_1+s_2^3s_1^2
\end{align*}
}}
%
\noindent In $G$ there are also these polynomials:
$$
{\mathcal{P}}_{x}^e= \mathbf{x}^2+\mathbf{x}s_4s_2^2+\mathbf{e}\,a_x+b_x\textrm{ and }
g_{x}= {\mathbf x_1}+{\mathbf x_2}+ c_x,
$$
where $a_x, b_x, c_x\in \FF_4[s_1,s_2,s_3,s_4,s_5]$ (see Appendix for the full polynomials).
Now we change the lex ordering to {\small$s_1<\dots<s_5<e_2<e_1<y_2<x_2<y_1<x_1$}.
In the  new \Gr\ basis we have other two polynomials ${\mathcal{P}}_{y}^e$ and $g_{y}$.
$$
{\mathcal{P}}_{y}^e= \mathbf{y}^2+\mathbf{y}(s_4s_3s_2+s_2^3+1)+\mathbf{e}\,a_y+b_y \textrm{ and } g_{y}= \mathbf{y_1}+\mathbf{y_2}+ c_y,
$$
where $a_y, b_y, c_y\in \FF_4[s_1,s_2,s_3,s_4,s_5]$ (see Appendix for the full polynomials).
We can decode as follows. First we evaluate $\mathcal{E}(\overline{\mathbf{s}})$ and we find two error values $\mathbf{e_1,e_2}$ (when $\tau=1$, one is zero).
\begin{itemize}
    \item[-]  If the syndrome corresponds to two errors, then the roots of ${\mathcal{P}}_{x}^e({\bf
s, e_2},x)$ are the first components of error locations,
    \item[-] else if $\overline{{\mathbf s}}$ corresponds to one error, we specialize $g_{x}(s,e,x_1,x_2)$ in $(\mathbf{s,e_2},1)$, where $1$ is the ghost component, and again the root of $ g_{x}({\bf s,e_2, 1},x_2)$ is  the first component of the  error location.
\end{itemize} 
Similarly we use ${\mathcal{P}}_{y}^e$ and $g_{y}$ to find the second location components.
Let us explain in detail the above-mentioned decoding with the help of the three cases of  Example \ref{exhq2}.
\begin{itemize}
    \item[-] $\overline{{\mathbf s}}=(0,1,1,1,0)$ is the syndrome vector
corresponding to the error vector $(0,0,0,0,0, 1, 1, 0)$.
Evaluating  ${\mathcal E}$ in $\overline{{\mathbf s}}$ we obtain:
${\mathcal E}(\overline{{\mathbf s}})=e^2+1$,
so two errors have occurred and their values is $1$.
In order to find the error positions we evaluate 
${\mathcal{P}}_{x}^e$ and ${\mathcal{P}}^e_{y}$ in $({\mathbf s}, 1)$ and we obtain
$$\begin{array}{l}
{\mathcal{P}}^e_{x}(\overline{{\mathbf s}}, 1)= x^2+x+1=(x-\alpha)(x-(\alpha+1)) \\ 
{\mathcal{P}}^e_{y}(\overline{{\mathbf s}}, 1)= y^2+y+1 =(y-\alpha)(y-(\alpha+1)) .
\end{array}$$
The system ${\mathcal{P}}^e_{x}(\overline{{\mathbf s}}, 1)={\mathcal{P}}^e_{y}(\overline{{\mathbf s}}, 1)=0$ have four possible solutions: $(\alpha,\alpha),
(\alpha+1,\alpha+1)$, $(\alpha+1,\alpha)$ and $(\alpha,\alpha+1)$. But only two solution pairs are admissible: $\{(\alpha+1,\alpha), (\alpha,\alpha+1)\}$ and  $\{(\alpha,\alpha), (\alpha+1,\alpha+1)\}$, since both $\alpha$ and $\alpha +1$ must appear as first components (and as second components).
We are in the same ambiguous situation as in Example \ref{exhq2}.
 \item[-] Now we consider the syndrome  $\overline{{\mathbf s}}=(\alpha+1,0,\alpha,0,0)$,
corresponding to $(1,\alpha,0,0,0,0, 0, 0)$.
Evaluating  ${\mathcal E}$ in $\overline{{\mathbf s}}$ we obtain $
{\mathcal E}(\overline{{\mathbf s}})=(e-1)(e-\alpha),$
so two errors have occurred and their values are $1$ and $\alpha$.
In order to find the error positions we evaluate 
${\mathcal{P}}^e_{x}$ and ${\mathcal{P}}^e_{y}$ in $({\mathbf s}, 1)$ (or in $({\mathbf s}, \alpha)$) 
{\footnotesize{$$
{\mathcal{P}}^e_{x}(\overline{{\mathbf s}}, 1)=f_{x}(\overline{{\mathbf s}}, \alpha)=x^2\mbox{ and }
{\mathcal{P}}^e_{y}(\overline{{\mathbf s}}, 1)=f_{y}(\overline{{\mathbf s}}, \alpha)=  y^2+ y =y(y-1).
$$}}
\noindent The solutions of the system $f_{x}(\overline{{\mathbf s}}, 1)=f_{y}(\overline{{\mathbf s}}, 1)=0$ are $\{(0,0), (1,\alpha)\}$, in this case we find the correct error positions.
Note that this case is an ambiguous situation in Example \ref{exhq2}, while here it is not.
\item[-] Vector $\overline{{\mathbf s}}=(\alpha+1,\alpha+1,1,\alpha+1,1)$ is the syndrome corresponding to $(0,0,\alpha+1,0,0,0,0,0)$. We evaluate ${\mathcal E}$ and we get 
${\mathcal E}(\overline{{\mathbf s}})=e^2+(\alpha+1) e$.
So only one error occurred and its value is $\alpha+1$. We evaluate
$g_{x}$ and $g_y$  in $(\overline{{\mathbf s}},\alpha+1, 1)$,
where $1$ is the first ghost component, and we have
{\footnotesize{
$$
g_{x}(\overline{{\mathbf s}},\alpha+1, 1)=x_2+1\textrm{ and }g_{y}(\overline{{\mathbf s}},\alpha+1, 1)= y_2+\alpha.
$$}}
Therefore the error location is $(1,\alpha)$.
\end{itemize}
\noindent Now we consider another type of decoding, using $\mathcal{E}$ and taking polynomials from $\widehat{J}_{\mathcal{FL}}^{\,C,t}$ as in (\ref{Inostro}).
First, we evaluate $\mathcal{E}(\overline{\mathbf{s}})$ to know the number of errors. We do not need their values. Instead, we compute the \Gr\ basis of ideal $\widehat{J}_{\mathcal{FL}}^{\,C,\tau}$, with $1\leq \tau\leq t$ and we collect polynomials in $\mathbb{F}_q[S,x]$ and $\mathbb{F}_q[S,y]$. For example,if two errors occur we use {\small{$s_1<s_2<s_3<s_4<s_5<x_2<y_2<x_1<y_1<e_2<e_1$}} and $\ldots s_5 < y_2 <\ldots$ to get, for $\tau =2$,
{\scriptsize{
$$
\begin{array}{ll}
 f_{2,x}=& \mathbf{x}^2+\mathbf{x}(s_4^2s_1+s_4s_2^2s_1^3+s_4s_2^2+s_2s_1^2)+
s_5^2s_3+s_5s_3s_2+s_4^2s_3^3s_2+s_4^2s_3^2s_2s_1+s_4^2s_3s_2s_1^2+s_4^2s_2+\\
&s_4s_3^3s_1^2+s_4s_3^2s_1^3+s_4s_3s_1+s_4s_1^2+s_3^3s_2^2s_1+s_3^2s_2^2s_1^2+
s_3s_2^2s_1^3+s_3s_2^2+s_2^2s_1\\ 
&\\        
f_{2,y}=& \mathbf{y}^2+\mathbf{y}(s_4^3+s_4s_3s_2s_1^3+s_4s_3s_2+s_4s_2s_1+s_3^2s_1+s_3s_1^2+s_2^3s_1^3+s_1^3+1)+
s_5^3+s_5s_4^2s_3^2s_2+\\
&s_5s_3^3s_2^2+s_5s_2^2+s_4^3s_3s_2^3s_1^2+s_4^3s_3s_1^2+s_4^3s_2^3+
s_4^3+s_4^2s_3^3s_2^2s_1^2+s_4^2s_3^2s_2^2+s_4^2s_3s_2^2s_1+\\
&s_4s_3^3s_2s_1+s_4s_3s_2s_1^3+s_4s_3s_2+s_3^3+s_3s_2^3s_1^2+s_2^3s_1^3+s_2^3+1   
\end{array}
$$}}
\noindent and for $\tau=1$
\vspace{-0.2cm}
{\footnotesize{$$
f_{1,x}= \mathbf{x}+s_2s_1^2\textrm{ and  }   
f_{1,y}= \mathbf{y}+s_3s_1^2.
$$}}
\noindent The decoding with $\{f_{2,x},f_{2,y},f_{1,x},f_{1,y}\}$ is obvious.\\
These polynomials are not the ideal polynomials yet, because again we may find parasite solutions (except with $\tau=1$).
\end{example}
In the previous examples we have used some polynomials as  weak multidimensional general error locator polynomials, as for example $\mathcal{L}_x$ and $\mathcal{L}_{xy}$ in Example \ref{exhq2}.
It is not obvious that such polynomials exist for any (affine-variety) code. To prove this, we need to analyse in depth the structure of the
zero-dimensional ideal $J_*^{C,t}$. 
This ideal turns out to belong to several interesting classes of zero-dimensional ideals, defined as generalizations of
stratified ideals. 
These ideal classes are rigorously studied in Section \ref{zeroideal}, where it is claimed in full generality that the sought-after polynomials can be found in a suitable \Gr\ basis. Section \ref{dim} is devoted to the proof of this claim. In Section \ref{Loc} we will come back to the coding setting.
%

%


\section{Results on some zero--dimensional ideals}
\label{zeroideal}
Our aim in this section is to describe the structure of
the reduced \GR \ basis for some special classes of zero-dimensional
ideals which are generalizations of stratified ideals. We suggest that the reader have a look at \cite{CGC-cd-art-gelp1}, since now we are generalizing our argument in \cite{CGC-cd-art-gelp1}.

First we provide a generalization of the material in Section
\ref{prelimenaries}. In this section  $J \subset
\mathbb{K}[\mathcal{\mathcal{S}},\mathcal{A}_L,\dots,\mathcal{A}_1,\mathcal{\mathcal{T}}]$
is a zero--dimensional ideal, with
$\mathcal{\mathcal{S}}=\{{\sf s}_1,\dots,{\sf s}_N\}$,
$\mathcal{A}_j=\{{\sf a}_{j,1},\dots,{\sf a}_{j,m}\}$, $j=1,
\dots, L$, $\mathcal{T}=\{{\sf t}_1,\dots,{\sf t}_K\}$. We fix a
block order $<$  on
$\mathbb{K}[\mathcal{\mathcal{S}},\mathcal{A}_L,\dots,
\mathcal{A}_1,\mathcal{\mathcal{T}}]$, with
$\mathcal{\mathcal{S}}<\mathcal{A}_L< \dots <
\mathcal{A}_1<\mathcal{T}$, such that within $\mathbb{A}_{j}$ we use lex with
${\sf {a}}_{j,1} < {\sf a}_{j,2} < \dots < {\sf a}_{j,m}$
(for any $j$). Let $\mathbb{A}$ and $\mathbb{A}_{j,i}$ denote the affine spaces  $\mathbb{A}=\mathbb{K}^{N+mL+K}$ and $\mathbb{A}_{j,i}=\mathbb{K}^{N+m(L-j)+i}$.\\

With the usual notation for the  elimination ideals, we want to partition 
$\mathcal{ V}(J_\mathcal{S})$
according to the number of extensions in 
$\mathcal{ V}(J_{\mathcal{S},{\sf{a}}_{L,1}})$, similarly to what done in Section \ref{2.2} in the one-variable case, that is, when $m=1$.
The additional complication here is that the $\sf{a}$ variables are
not $L$ any more, but rather they are collected into $L$ blocks, 
each block having $m$ variables.
Since we order the $\sf{a}$ variables first according to their block
(block ${\mathcal{A}}_L$ is the least) and then within the block
from the least to the greatest, their first index denotes the block
and their second index denotes their position within the block itself.
So, the least $\sf{a}$ variable is ${\sf{a}}_{L,1}$ and
the greatest is ${\sf{a}}_{1,m}$.\\
The members of the  partition of $\mathcal{ V}(J_\mathcal{S})$ will be called $\{\Sigma_l^{L,1}\}$ 
(similarly to the previously defined $\Sigma_l^{L}$). The maximum number
of extensions will be called $\eta(L,1)$ (compare with $\lambda(L)$).
\begin{remark}\label{droprad}
It is essential to count the number of extensions in 
$\mathcal{ V}(J_{\mathcal{S},{\sf{a}}_{L,1}})$ discarding their multiplicities. In the definition of a stratified ideal we required radicality, so in that case multiplicities did not arise. However, in our following multidimensional generalisations of results and definitions from Sec.~\ref{2.2}-\ref{seclocator}, we must drop radicality and so we have to be very careful when handling multiplicities.
\end{remark}
In the general case, if we consider block $j$ and variable ${\sf{a}}_{j,i}$,
we partition $\mathcal{V}(J_{\mathcal{S},\mathcal{A}_L,\dots,
\mathcal{A}_{j+1},{\sf{a}}_{j,1},\dots,{\sf{a}}_{j,i}})$ into subsets
$\{\Sigma_l^{j,i+1}\}$ according to the number of extensions to
$\mathcal{V}(J_{\mathcal{S},\mathcal{A}_L,\dots,
\mathcal{A}_{j+1},{\sf{a}}_{j,1},\dots,{\sf{a}}_{j,i},{\sf{a}}_{j,i+1}})$,
that is, adding the next variable ${\sf{a}}_{j,i+1}$.
The maximum number of extensions will be called $\eta(j,i+1)$.
We meet a special case when we consider the last variable in a block
(i.e., $i=m$), since in that case we extend from
$\mathcal{V}(J_{\mathcal{S},\mathcal{A}_L,\dots,
\mathcal{A}_{j+1},{\sf{a}}_{j,1},\dots,{\sf{a}}_{j,m}})$
to
$\mathcal{V}(J_{\mathcal{S},\mathcal{A}_L,\dots,
\mathcal{A}_{j-1},{\sf{a}}_{j,1},\dots,{\sf{a}}_{j,m},{\sf{a}}_{j-1,1}})$.
However, no confusion will arise if we follow our convention of naming 
the partition members according to the {\em added} variable, so
they are called $\{\Sigma_l^{j-1,1}\}$ in this case, even if their
union is V=$\mathcal{V}(J_{\mathcal{S},\mathcal{A}_L,\dots,
\mathcal{A}_{j-1},{\sf{a}}_{j,1},\dots,{\sf{a}}_{j,m}})$. Coherently,
$\eta(j-1,1)$ denotes the maximum number of extensions for
points in $V$.

A formal description of the above discussion goes as follows, where 
$l,j$ and $m$ are integers such that $l\geq 1$, $1\leq j\leq L$ and $1\leq i\leq m$:
{\small{\begin{align*}
\Sigma_l^{L,1}=& \{ (\bar{{\sf s}}_1,\ldots,\bar{{\sf
    s}}_N)\in \mathcal{ V}(J_\mathcal{S}) \mid \exists
    {\mbox{\ exactly }} l {\mbox{ distinct values }}
     \bar{{\sf{a}}}_{L,1}^{(1)}, \ldots,  \bar{{\sf{a}}}_{L,1}^{(l)} \\
    & \mathrm{s.t}.\,\, (\bar{{\sf s}}_1,\ldots,\bar{{\sf s}}_N,\bar{{\sf{a}}}^{(\ell)}_{L,1})
    \in \mathcal{ V}(J_{\mathcal{S},{\sf{a}}_{L,1}})\,\, {\mbox{with}} \,\, 1 \leq \ell \leq l
    \},\\
 \Sigma_l^{j,1}=& \bigl\{ (\bar{{\sf s}}_1,\ldots,\bar{{\sf s}}_N,{\bar{\sf{a}}}_{L,1},\dots,
    \bar{{\sf{a}}}_{L,m},\dots,{\bar{\sf{a}}}_{j+1,1},\dots,
    \bar{{\sf{a}}}_{j+1,m})
    \in \mathcal{ V}(J_{\mathcal{S},\mathcal{A}_{L},\dots,\mathcal{A}_{j+1}})
    \mid\\
    &\exists {\mbox{\ exactly}}\,\, l {\mbox{ distinct values }}
    \bar{{\sf{a}}}_{j,1}^{(1)}, \ldots,  \bar{{\sf{a}}}_{j,1}^{(l)}
   \,\,
    \mathrm{s.t. \,\,for\,\, any}\,\,\,\, 1 \leq \ell \leq l\\
    &(\bar{{\sf{s}}}_1,\dots,\bar{{\sf{s}}}_{N},{\bar{\sf{a}}}_{L,1},\dots,
    \bar{{\sf{a}}}_{L,m},\dots,{\bar{\sf{a}}}_{j+1,1},\dots,
    \bar{{\sf{a}}}_{j+1,m},\bar{{\sf{a}}}_{j,1}^{(\ell)})
    \in \mathcal{
    V}(J_{\mathcal{S},\mathcal{A}_{L},\dots,\mathcal{A}_{j+1},{\sf{a}}_{j,1}})\bigr\}\\
    & j=1, \dots, L-2,\\
\Sigma_l^{j,i}=& \bigl\{ (\bar{{\sf
s}}_1,\ldots,\bar{{\sf s}}_N,{\bar{\sf{a}}}_{L,1},\dots,
    \bar{{\sf{a}}}_{L,m},\dots,{\bar{\sf{a}}}_{j+1,1},\dots,
    \bar{{\sf{a}}}_{j+1,m},\bar{{\sf{a}}}_{j,1},\dots,\bar{{\sf{a}}}_{j,i-1})
    \,\, \mathrm{in}\\
    &\mathcal{ V}(J_{\mathcal{S},\mathcal{A}_{L},\dots,\mathcal{A}_{j+1},{{\sf{a}}}_{j,1},\dots,{{\sf{a}}}_{j,i-1}})
    \mid \exists
    {\mbox{\ exactly}} \,\, l {\mbox{ distinct values }}
    \bar{{\sf{a}}}_{j,i}^{(1)}, \ldots,  \bar{{\sf{a}}}_{j,i}^{(l)}
    \,\,\,
    \mathrm{s.t.}\\
    &(\bar{{\sf{s}}}_1,\dots,\bar{{\sf{s}}}_{N},{\bar{\sf{a}}}_{L,1},\dots,
    \bar{{\sf{a}}}_{L,m},\dots,{\bar{\sf{a}}}_{j+1,1},\dots,
    \bar{{\sf{a}}}_{j+1,m},\bar{{\sf{a}}}_{j,1},\dots,\bar{{\sf{a}}}_{j,i-1},\bar{{\sf{a}}}_{j,i}^{(\ell)}) \rm{\ is \,\, in}\\
    & \mathcal{
    V}(J_{\mathcal{S},\mathcal{A}_{L},\dots,\mathcal{A}_{j+1},{\sf{a}}_{j,1},\dots,{\sf{a}}_{j,i}})
    \, 1 \leq \ell \leq l\bigr\},  \quad    i=2,\dots,m,\, j=1,\dots,L-1.
\end{align*}}}
\noindent The maximum number of extensions at any level, which is $\eta(j,i)$, plays an important
role for our approach and therefore deserves a precise definition.
Before defining it, we need an elementary result.
\begin{fact}\label{vecchiadef}
Given $J$, there is a set of natural numbers $\{\eta(j,i)\}_{\substack{1\leq j \leq L,\\ 1 \leq i \leq m}}$ such that
\begin{itemize}
\item[i)]$\mathcal{V}(J_{\mathcal{S}})=\sqcup_{l=1}^{\eta(L,1)}\Sigma_l^{L,1}$;
\item[ii)] $\mathcal{
V}(J_{\mathcal{S},{\sf{a}}_{L,1},\dots,{\sf{a}}_{L,i}})=\sqcup_{l=1}^{\eta(L,i+1)}\Sigma_l^{L,i+1}$,
$ i=1,\dots,m-1$;
\item[iii)] $\mathcal{
V}(J_{\mathcal{S},\mathcal{A}_L})=\sqcup_{l=1}^{\eta(L-1,1)}\Sigma_l^{L-1,1}$;
\item[iv)] $\mathcal{
V}(J_{\mathcal{S},\mathcal{A}_L,\dots,\mathcal{A}_{j+1}})=\sqcup_{l=1}^{\eta(j,1)}\Sigma_l^{j,1}$,
$j=1,\dots, L-2$;
\item[v)] $\mathcal{
V}(J_{\mathcal{S},\mathcal{A}_L,\dots,\mathcal{A}_{j+1},{\sf{a}}_{j,1},\dots,
{\sf{a}}_{j,i}})=\sqcup_{l=1}^{\eta(j,i+1)}\Sigma_l^{j,i+1}$,
$i=1,\dots,m-1, j=1,\dots,L-1$;
\item[vi)] $\Sigma^{j,i}_{{\eta(j,i)}}\neq\emptyset$, $\forall
i=1,\dots,m, \, \forall j=1,\dots,L$. 
\end{itemize}
\end{fact}
\begin{proof}
Since $I$ is zero-dimensional ideal, $\mathcal{V}(I)$ is finite and so any variety projection $V=\mathcal{V}(J_{\mathcal{S},\mathcal{A}_L,\dots,\mathcal{A}_{j+1},
{\sf{a}}_{j,1},\dots,{\sf{a}}_{j,i-1}})$ has a finite number of points. Obviously $V$ is the union of the corresponding $\Sigma_l^{j,i}$, which means that there can be only a finite number of non-empty $\Sigma_l^{j,i}$  and so we use the notation $\eta(j,i)$ to denote the largest $l$ such that  $\Sigma_l^{j,i}$ is non-empty.
\end{proof}

\begin{definition}\label{levelf}
The {\bf level function} of $J$ (with respect to  the
${\mathcal{A}_L,\dots,\mathcal{A}_1}$ variables) is the function $\eta: \{1\ldots L\}\times \{1\ldots m\} \rightarrow \mathbb{N}$  satisfying Fact~\ref{vecchiadef}.
\end{definition}

We want now to generalize our previous definition of stratified ideals (Defi-nition \ref{stratificato}) to the multivariate case, but dropping radicality (see Remark \ref{droprad}).
It turns out that there are two ways of doing it: we have a weaker notion
in the next definition and two stronger notions in the subsequent definition.
%
\begin{definition}\label{weaklyideal}
Let $J$ be a zero-dimensional  ideal with the above
notation. We say that $J$ is a {\bf weakly stratified ideal} if
$$
\Sigma^{j,i}_{l}\neq \emptyset \qquad\textrm{ for }1 \leq l \leq\eta(j,i),\,
1\leq i\leq m, \, 1\leq j\leq L.
$$
\end{definition}
\noindent Being weakly stratified means that when considering the elimination
ideal at level $(j,i)$ (block $j$ and variable ${\sf{a}}_{j,i}$)
if there is a  variety point with $l\geq 2$ extensions then there is another point with $l-1$ extensions.

The following definition of multi-stratified ideal 
is given at variable-block level, rather than at a single-variable level. It 
contains two  conditions: there is at least
one point with exactly $j$ extensions and there are no ``gaps''  in the
number of extensions (for any integer $1\leq l\leq j$ there is at least
one point with $l$ extensions).
So it is exactly the multidimensional analogue of the definition of stratified
ideals, except that we drop the radicality.
Unfortunately, this straightforward generalization does not guarantee
the existence of polynomials playing the role of ``ideal'' locators, and so in the same definition
we provide an even stronger notion ``strongly multi-stratified ideal''.
\begin{definition}\label{strongly}
Let $J$ be a zero-dimensional ideal with the above
notation.
Let us consider the natural projections
\begin{align*}
&\pi_L:\,
\hspace{0.4cm}{\mathcal{V}}(J_{\mathcal{S},\mathcal{A}_L})
\hspace{0.3cm} \longrightarrow
{\mathcal{V}}(J_{\mathcal{S}})\\
&\pi_j:\,
{\mathcal{V}}(J_{\mathcal{S},\mathcal{A}_L,\ldots,\mathcal{A}_{j+1},\mathcal{A}_{j}})
\longrightarrow
{\mathcal{V}}(J_{\mathcal{S},\mathcal{A}_L,\ldots,\mathcal{A}_{j+1}}),
\,\, j=1,\dots, L-1 \\
&\rho_j:\,
{\mathcal{V}}(J_{\mathcal{S},\mathcal{A}_L,\ldots,\mathcal{A}_{j+1},\mathcal{A}_{j}})
\longrightarrow
{\mathcal{V}}(J_{\mathcal{A}_{j}}),
\,\, j=1,\dots, L
\end{align*}
\indent
Ideal $J$ is a {\bf multi-stratified ideal}
$($in the $\mathcal{A}_L,\ldots,\mathcal{A}_{1}$ variables$)$ if
\begin{itemize}
\item[$1)$] for any $1\leq j\leq L-1$  and for any $P \in
{\mathcal{V}}(J_{\mathcal{S},\mathcal{A}_L,\dots,\mathcal{A}_{j+1}})$
we have that $|\pi_{j}^{-1}(\{ P \})| \leq j$. \\
Moreover, for any $\bar s   \in$
${\mathcal{V}}(J_{\mathcal{S}})$ we have that $|\pi_{L}^{-1}(\{ \bar s\})|$
$\leq L$;
\item[$2)$] for any $1\leq j\leq L-1$ there is $Q \in
{\mathcal{V}}(J_{\mathcal{S},\mathcal{A}_L,\dots,\mathcal{A}_{j+1}})$
s.t. $|\pi_j^{-1}(\{Q\})|=j$.\\ Moreover, there is $\bar s \in
{\mathcal{V}}(J_{\mathcal{S}})$ s.t. 
$|\pi_L^{-1}(\{\bar s\})|=L$.
\end{itemize}

\indent
For any $1\leq j\leq L$, let $Z_j=\rho_j({\mathcal{V}}(J))$. We say that ideal $J$ is a {\bf strongly multi-stratified ideal}
(in the ${\mathcal{A}_L,\dots,\mathcal{A}_1}$ variables) if 1) holds and
\begin{itemize}
\item[$3)$]
for any $1\leq j\leq L-1$, for any $T\subset Z_j$ s.t. $1\leq |T|\leq j$ 
there is 
a $Q \in {\mathcal{V}}(J_{\mathcal{S},\mathcal{A}_L,\dots,\mathcal{A}_{j+1}})$ s.t.
$\rho_j(\pi_j^{-1}(\{ Q\}))=T$.\\
Moreover, for any $T\subset Z_j$ s.t. $1\leq |T|\leq L$ there is
an $\bar s \in {\mathcal{V}}(J_{\mathcal{S}})$ s.t.
$\rho_j(\pi_L^{-1}(\{ \bar s\}))=T$.
\end{itemize}
\end{definition}
Again, in the previous definition, we do not count multiplicities.\\
\begin{remark} For any zero-dimensional ideal $J$ with the above notation, let $Z=Z_1$. Once $\rho_{j'}({\mathcal{V}}(J))=\rho_j({\mathcal{V}}(J))$ for any $1\leq j,j'\leq L$, we obviously have
$\rho_j(\pi_L^{-1}(\{ \bar s\})) \subset Z$. Assuming this, 1) and 3) could be replaced
by saying that there is a bijection between
the sets of $\rho_j(\pi_L^{-1}(\{  Q\}))$ and all (non-empty) subsets of
$Z$ with up to $j$ elements
(and a similar condition at level $L$).
\end{remark}

We note the following obvious fact.
\begin{fact} \label{fact1} 
Let $m\geq 1$. If $J$ is a strongly multi-stratified ideal  then
  $J$ is a  multi-stratified ideal.\\
\indent Let $m=1$. If   $J$ is a  multi-stratified ideal 
then $J$ is a weakly stratified ideal. 
If $J$ is radical, then $J$ is a multi-stratified ideal  if and only
if  $J$ is a stratified ideal.
\end{fact}
The next two examples clarify (in the case $m=1$) the notions of multi-stratified ideals and of weakly stratified ideals.
\begin{example}
Let  $\mathcal{S}=\{{\sf{s}}_1\}$,
$\mathcal{A}_1=\{{\sf{a}}_{1,1}\}$,
$\mathcal{A}_2=\{{\sf{a}}_{2,1}\}$, so that $m=1$, and
$\mathcal{T}=\{{\sf{t}}_1\}$. Let  $J=\mathcal{I}(Z) \subset
\mathbb{C}[{\sf{s}}_1,{\sf a}_{2,1},{\sf a}_{1,1},{\sf t}_1]$ with
$Z=$ $\{(0,0,0,0)$, $(0,1,1,0)$, $(0,2,2,0)\}$. The order $<$ is
$s_1<a_{2,1}<a_{1,1}<t_1$ and the varieties  are
 {\scriptsize{\begin{align*}
& \mathcal{V}(J_\mathcal{S})=\{0\}, \quad
\mathcal{V}(J_{\mathcal{S},{\sf a}_{2,1}})=\{(0,0),(0,1),(0,2)\},\quad\mathcal{V}(J_{\mathcal{S},{\sf a}_{2,1},{\sf
a}_{1,1}})=\{(0,0,0),(0,1,1),(0,2,2)\}.
\end{align*}}}
Let us consider the projection $\pi_2:
\mathcal{V}(\mathcal{J}_{\mathcal{S},{\sf a}_{2,1}})
\rightarrow \mathcal{V}(\mathcal{J}_{\mathcal{S}})$. Then
$|\pi_2^{-1}(\{0\})|=3$.
We have $\sum_3^{2,1}=\{0\}$ and $\sum_1^{2,1}=\emptyset$, $\sum_2^{2,1}=\emptyset$.
So $\eta(2,1)=3$ and $J$ is not  a  weakly stratified ideal (neither a stratified ideal).
\end{example}

\begin{example}
Let  $\mathcal{S}=\{{\sf{s}}_1\}$,
$\mathcal{A}_1=\{{\sf{a}}_{1,1}\}$,
$\mathcal{A}_2=\{{\sf{a}}_{2,1}\}$, $\mathcal{A}_3=\{{\sf{a}}_{3,1}\}$, $\mathcal{T}=\{{\sf{t}}_1\}$ so that $m=1$.
Let  $J=\mathcal{I}(Z) \subset
\mathbb{C}[{\sf{s}}_1,{\sf{a}}_{3,1},{\sf a}_{2,1},{\sf a}_{1,1},{\sf t}_1]$ with
$Z=$ $\{(0,1,0,0,0),$ $(0,2,1,1,2),$ $(2,2,2,0,0)\}$. The order $<$ is
$s_1<a_{3,1}<a_{2,1}<a_{1,1}<t_1$ and the varieties  are
 {\scriptsize{$$
\begin{array}{l}
\mathcal{V}(J_\mathcal{S})=\{0,2\},\quad
\mathcal{V}(J_{\mathcal{S},{\sf a}_{3,1}})=\{(0,1),(0,2),(2,2)\},\\   
\mathcal{V}(J_{\mathcal{S},{\sf a}_{3,1},{\sf
a}_{2,1}})=\{(0,1,0),(0,2,1),(2,2,2)\},\\
\mathcal{V}(J_{\mathcal{S},{\sf a}_{3,1},{\sf a}_{2,1},{\sf a}_{1,1}})=\{(0,1,0,0),(0,2,1,1),(2,2,2,0)\}. 
\end{array}
$$}}
Let us consider the projection $\pi_3:
\mathcal{V}(\mathcal{J}_{\mathcal{S},{\sf a}_{3,1}})
\rightarrow \mathcal{V}(\mathcal{J}_{\mathcal{S}})$. Then
$|\pi_3^{-1}(\{0\})|=2$ and $|\pi_3^{-1}(\{2\})|=1$, so $\sum_2^{3,1}=\{0\}$, $\sum_1^{3,1}=\{2\}$ and $\eta(3,1)=2$, but $\sum_3^{3,1}=\emptyset.$ 
Similarly,  $\eta(2,1)=\eta(1,1)=1$.
So $J$ is a weakly stratified ideal that is
not  multi-stratified (and not stratified).
\end{example}

However, if $m\geq 2$, a weakly stratified ideal is not necessarily a multi-stratified ideal and, viceversa, a multi-stratified ideal
is not necessarily a weakly stratified ideal, as shown in the following example.

\begin{example}
Let  {\small{$\mathcal{S}=\{{\sf{s}}_1,{\sf{s}}_2,{\sf{s}}_3\}$,
$\mathcal{A}_1=\{{\sf{a}}_{1,1},{\sf{a}}_{1,2}\}$,
$\mathcal{A}_2=\{{\sf{a}}_{2,1},{\sf{a}}_{2,2}\}$, 
$\mathcal{A}_3=\{{\sf{a}}_{3,1},{\sf{a}}_{3,2}\}$, 
$\mathcal{T}=\{{\sf{t}}_1\}$}} so that $m=2$.\\
Let  {\small{$J=\mathcal{I}(Z) \subset
\mathbb{C}[{\sf{s}}_1,{\sf{s}}_2,{\sf{s}}_3,{\sf{a}}_{3,1},{\sf{a}}_{3,2},
{\sf{a}}_{2,1},{\sf a}_{2,2},{\sf a}_{1,1},{\sf a}_{1,2},{\sf t}_1,{\sf t}_2]$}}, with 
$Z=$ $\{(0,0,1,1,$ $1,1,1,3,1,1,1),$ $(0,0,1,1,2,1,3,1,2,1,1),$ $(0,0,1,1,3,0,0,2,1,1,2),$ $(1,1,2,2,$ $1,2,1,0,0,0,1),$ $(1,1,2,0,1,1,1,0,1,2,1),$ $(1,1,2,0,1,1,0,1,0,0,1),$ $(2,3,0,3,$ $3,1,0,1,1,1,2)\}$. 
The order $<$ is
{\scriptsize{$s_1<s_2<s_3<a_{3,1}<a_{3,2}<a_{2,1}<a_{2,2}<a_{1,1}<a_{1,2}<t_1$}} and the varieties  are
{\scriptsize{
\begin{align*}
\mathcal{V}(J_\mathcal{S})=&\{(0,0,1), (1,1,2), (2,3,0)\},\\
\mathcal{V}(J_{\mathcal{S},\mathcal{A}_3})=&\{(0,0,1,1,1),(0,0,1,1,2),
(0,0,1,1,3),(1,1,2,2,1),(1,1,2,0,1),(2,3,0,3,3)\},\\   
\mathcal{V}(J_{\mathcal{S},\mathcal{A}_3,\mathcal{A}_2})=&\{
(0,0,1,1,1,1,1),(0,0,1,1,2,1,2),
(0,0,1,1,3,0,0),(1,1,2,2,1,2,1),(1,1,2,0,1,1,1),\\
&(1,1,2,0,1,1,0),(2,3,0,3,3,1,0)\},\\
\mathcal{V}(J_{\mathcal{S},\mathcal{A}_3,\mathcal{A}_2,\mathcal{A}_1})=&
\{(0,0,1,1,1,1,1,3,1),(0,0,1,1,2,1,2,1,2),
(0,0,1,1,3,0,0,2,1),(1,1,2,2,1,2,1,0,0),\\
&(1,1,2,0,1,1,1,0,1),(1,1,2,0,1,1,0,1,0),(2,3,0,3,3,1,0,1,1)\}. 
\end{align*}}}
\noindent Let us consider the projection {\small{$\pi_3:
\mathcal{V}(\mathcal{J}_{\mathcal{S},\mathcal{A}_3})
\rightarrow \mathcal{V}(\mathcal{J}_{\mathcal{S}})$}}.\\   
Then {\small{$|\pi_3^{-1}(\{(0,0,1)\})|=3$}},  {\small{$|\pi_3^{-1}(\{(1,1,2)\})|=2$}} and  {\small{$|\pi_3^{-1}(\{(2,3,0)\})|=1$}}.\\ 
Similarly, if we consider $\pi_2:
\mathcal{V}(\mathcal{J}_{\mathcal{S},\mathcal{A}_3,\mathcal{A}_2})
\rightarrow \mathcal{V}(\mathcal{J}_{\mathcal{S},\mathcal{A}_3}),$ then 
{\small{$|\pi_2^{-1}(\{(1,1,2,0,1)\})|$}} is equal to $2$ and  for other  {\footnotesize{$ P\in \mathcal{V}(\mathcal{J}_{\mathcal{S},\mathcal{A}_3})$}} we have that
{\footnotesize{$|\pi_2^{-1}(\{P\})|=1$}}.\\ 
Finally, if we consider $\pi_1:
\mathcal{V}(\mathcal{J}_{\mathcal{S},\mathcal{A}_3,\mathcal{A}_2,\mathcal{A}_1})
\rightarrow \mathcal{V}(\mathcal{J}_{\mathcal{S},\mathcal{A}_3,\mathcal{A}_2}),$ then for any\\ {\small{$ P\in \mathcal{V}(\mathcal{J}_{\mathcal{S},\mathcal{A}_3,\mathcal{A}_2})$}} we have that
$|\pi_1^{-1}(\{P\})|=1$ and so  $J$ is   multi-stratified.\\ 
It is easy to see that $J$ is not weakly stratified. In fact, if we consider a projection
{\small{$\pi_{3,1}:
\mathcal{V}(\mathcal{J}_{\mathcal{S},{\sf a}_{3,1},{\sf a}_{3,2}})
\rightarrow \mathcal{V}(\mathcal{J}_{\mathcal{S},{\sf a}_{3,1}})$}}, then
{\small{$\pi_{3,2}^{-1}(\{(0,0,1,1)\})=\{(0,0,1,1,1),$ $(0,0,1,1,2),$ $(0,0,1,1,3)\}$, $\pi_{3,2}^{-1}(\{(1,1,2,2)\})=\{(1,1,2,2,1)\}$, $\pi_{3,2}^{-1}(\{(1,1,2,0)\})=\{(1,1,2,0,1)\}$, $\pi_{3,2}^{-1}(\{(2,3,0,3)\})=\{(2,3,0,3,3)\}$}}. So 
$\sum_3^{3,2}=\{(0,0,1,1)\}$, but $\sum_2^{3,2}=\emptyset$.
\end{example}
%
%
\begin{proposition} \label{smsw}
Let $J$ be a strongly multi-stratified ideal  then
 $J$ is a  weakly stratified ideal.
\begin{proof}
For any $1\le i\le m$ and for any $j=1,\dots, L-1$,  let us consider the natural projection 
{\footnotesize{$$
\pi_{j,i}:\,
{\mathcal{V}}(J_{\mathcal{S},\mathcal{A}_L,\dots,\mathcal{A}_{j+1},
\bar{{\sf{a}}}_{j,1},\dots,\bar{{\sf{a}}}_{j,i}})
\longrightarrow
{\mathcal{V}}(J_{\mathcal{S},\mathcal{A}_L,\dots,\mathcal{A}_{j+1},
\bar{{\sf{a}}}_{j,1},\dots,\bar{{\sf{a}}}_{j,i-1}})
$$}}
\noindent We will also use $\rho_j$ and $\pi_j$ as in Definition \ref{strongly}.\\
To avoid complications, we consider only the case $2\leq i\leq m-1$, being the modifications in the  $i=1$ and $i=m$  obvious.\\
\indent The first fact that we note is that $ \eta(j,i)\leq j$, because if the pre-images at block level contain at most $j$ elements, then at variable level they cannot contain more.
Let {\small{$2\le l\le \eta(j,i)$}} such that {\small{$\Sigma_l^{j,i}\not=\emptyset$}}.
It is enough to show that\\ {\small{$\Sigma_{l-1}^{j,i}\not=\emptyset$}}.

\noindent Let $\bar R, \bar P$ and $Q$ such that $Q\in\Sigma_l^{j,i}$, {\small{$Q=(\bar{\mathcal{S}},\bar{\mathcal{A}}_L,
\dots,\bar{\mathcal{A}}_{j+1},
\bar{{\sf{a}}}_{j,1},\dots,\bar{{\sf{a}}}_{j,i-1})$}},\\ {\small{$\bar P=(\bar{\mathcal{S}},\bar{\mathcal{A}}_L,
\dots,\bar{\mathcal{A}}_{j+1})$}}, {\small{$\bar R=(\bar{{\sf{a}}}_{j,1},\dots,\bar{{\sf{a}}}_{j,i-1})$}}, so  {\small{$Q= (\bar{P},\bar R)$}}. \\
Then {\small{$\pi_{j,i}^{-1}(\{Q\})=\{(Q,\lambda_1),\ldots, (Q,\lambda_l)\}$}} and all $\lambda_{\ell}$'s are distinct. \\
Let {\small{$\Gamma_1,\ldots, \Gamma_l \in \mathcal{V}(J_{{\sf{a}}_{j,i+1},\dots,{\sf{a}}_{j,m}})$}} such that {\small{$(Q,\lambda_{\ell},\Gamma_{\ell})\in \mathcal{V}(J_{\mathcal{S},\mathcal{A}_L,\dots,\mathcal{A}_{j}})$}}. The $\Gamma_{\ell}$'s do not have to be distinct. For any {\small{$1\leq \ell\leq l$}} at least one such {\small{$\Gamma_{\ell}$}} must exist. We choose  one {\small{$\Gamma_{\ell}$}} for any ${\ell}$.
So {\footnotesize{$\{(Q,\lambda_1,\Gamma_1),\ldots, (Q,\lambda_l,\Gamma_l)\}\subset \pi_j^{-1}(\bar P)$}} and \\{\footnotesize{$\{(\bar R,\lambda_{\ell},\Gamma_{\ell})\}_{1\leq \ell\leq l} \subset \rho_j(V(J_{\mathcal{S},\mathcal{A}_L,\ldots,\mathcal{A}_{j+1},\mathcal{A}_{j}}))$}}. 
Let 
{\small{
$$
T=\{(\bar R,\lambda_1,\Gamma_1),\ldots, (\bar R,\lambda_{l-1},\Gamma_{l-1})\}.
$$}}
\hspace{-0.3cm}
\noindent Then $T\subset \rho_j(V(J_{\mathcal{S},\mathcal{A}_L,\ldots,\mathcal{A}_{j+1},\mathcal{A}_{j}})) $ and $|T|=l-1\leq \eta(j,i)-1\leq j-1$.\\
Since $J$ is strongly multi-stratified, there is $\widetilde{P}\in{\mathcal{V}}(J_{\mathcal{S},\mathcal{A}_L,\dots,
\mathcal{A}_{j+1}})$ such that
{\footnotesize{$$T=\rho_j(\pi_j^{-1}(\{\widetilde{P}\}))\textrm{, so }\pi_j^{-1}(\{\widetilde{P}\})=  \{(\widetilde{P},\bar R,\lambda_1,
\Gamma_1),\ldots, (\widetilde{P},\bar R,\lambda_{l-1},\Gamma_{l-1})\}.
$$}} 

\hspace{-0.2cm}\noindent This implies that {\footnotesize{$\{(\widetilde{P},\bar R,\lambda_1),\ldots, (\widetilde{P},\bar R,\lambda_{l-1})\}= \pi_{j,i}^{-1}(\{(\widetilde{P},\bar R)\})$}}, and so  {\footnotesize{$\Sigma_{l-1}^{j,i}\not=\emptyset$}},
as all $\lambda_{\ell}$'s are distinct.
\end{proof}
\end{proposition}
Let $<_{\lex}$ be the lexicographic term order such that $\mathcal{A}_L <_{\lex} \ldots <_{\lex} \mathcal{A}_1$ and $a_{j,1} <_{\lex} \ldots <_{\lex} a_{j,m}$, for any $1 \leq j \leq L$. Let $<_S$ be a term order on $\mathcal{S}$ and $<_{\mathcal{T}}$ a term order on $\mathcal{T}$. Let $<$ be the block order $<=(<_S,<_{\lex},<_{\mathcal{T}})$. We are now assuming that $J$ is any zero-dimensional ideal is $\mathbb{K}[\mathcal{S},\mathcal{A}_L,\dots,\mathcal{A}_1,\mathcal{T}]$. 
Let $G=\mathrm{GB}(J)$. It is
well--known that the elements of $G \cap
(\mathbb{K}[\mathcal{S},\mathcal{A}_L,\dots,\mathcal{A}_1]\setminus\mathbb{K}[\mathcal{S}])$
can be collected into non-empty blocks $\{G^{\,j}\}_{1\leq j \leq L}$, where\footnote{In \cite{CGC-cd-art-gelp1}  we use the notation $G_i$.} 
$$G^{\,L}=G\cap(\mathbb{K}[\mathcal{S},\mathcal{A}_{L}]\sms\mathbb{K}[\mathcal{S}])$$
and, for $1\leq j\leq L-1$,
$$
G^{\,j}=G \cap (\mathbb{K}[\mathcal{S},
\mathcal{A}_L,\dots,\mathcal{A}_{j+1}, \mathcal{A}_{j}]\sms\mathbb{K}[\mathcal{S},
\mathcal{A}_L,\dots,\mathcal{A}_{j+1}]).
$$ 
Then we denote by $G^{\, L,1}$, $G^{\, L,i}$, $G^{\, j,1}$ and 
$G^{\, j,i}$, $1 < j \leq L,\, 1 < i \leq m$, respectively, the sets:
$$
\begin{array}{ll}
G^{\, L,1}=&G \cap (\mathbb{K}[\mathcal{S},
{\sf a}_{L,1}]\setminus\mathbb{K}[\mathcal{S}])\\
G^{\, L,i}=&G \cap (\mathbb{K}[\mathcal{S},{\sf a}_{L,1},\dots,{\sf a}_{L,i}]\setminus\mathbb{K}[\mathcal{S},{\sf a}_{L,1},\dots,{\sf a}_{L,i-1}])\\
G^{\, j,1}=&G \cap (\mathbb{K}[\mathcal{S},
\mathcal{A}_L,\dots,\mathcal{A}_{j+1},{\sf a}_{j,1}]\setminus\mathbb{K}[\mathcal{S},
\mathcal{A}_L,\dots,\mathcal{A}_{j+1}])\\
G^{\, j,i}=&G \cap (\mathbb{K}[\mathcal{S},
\mathcal{A}_L,\dots,\mathcal{A}_{j+1},{\sf a}_{j,1},\dots,{\sf a}_{j,i}]\setminus\mathbb{K}[\mathcal{S},
\mathcal{A}_L,\dots,\mathcal{A}_{j+1},{\sf a}_{j,1},\dots,{\sf a}_{j,i-1}]).  
\end{array}
$$
\noindent In other words, let $g$ be any polynomial in $G^{\, j,i}$.
Then: 
\begin{itemize}
\item $g$ contains the variable ${\sf a}_{j,i}$, 
\item $g$ does not contain any greater variable (i.e. no variables in blocks 
      ${\mathcal{A}}_{j+1}\ldots{\mathcal{A}}_{1}$ and none of the remaining 
      variables in the $j$-th block ${\sf a}_{j,i+1},\ldots,{\sf a}_{j,m}$),
\item $g$ \underline{may} contain lesser variables (the ${\mathcal{S}}$ variables,
      the $\sf{a}$ variables contained in blocks $L,\ldots,j-1$ and the
      lesser $\sf{a}$ variables in the same block: 
      ${\sf{a}}_{j,1},\ldots,{\sf{a}}_{j,i-1}$).
\end{itemize}
As the ideal under consideration is assumed zero-dimensional, the sets $G^{\, j,i}$ are non-empty.
The polynomials in any $G^{\, j,i}$ can be grouped according to
their degree $\delta$ with respect to  ${\sf a}_{j,i}$.\\
\noindent For us it is essential to know the \textit{ maximum value} of $\delta$ in $G^{\, j,i}$, that we call 
\begin{eqnarray}\label{zeta}
\qquad\qquad\qquad\qquad\qquad\qquad\quad\zeta(j,i)
\end{eqnarray}
\noindent So we can write:
{\small{$$
G^{\, j,i}=\sqcup_{\delta=1}^{{\zeta(j,i)}} G_{\delta}^{\, j,i},
\quad j=1,\dots,L, \,\, i=1,\dots,m, \,\,\mathrm{with}\,\,\,
G_{\zeta(j,i)}^{\, j,i} \neq \emptyset,
$$}}
but some
$G_{\delta}^{\, j,i}$ could be empty. 
In this way, if $g \in G_{\delta}^{\, j,i}$ we have:
\begin{itemize}
\item {\small{$g \in \mathbb{K}[\mathcal{S},
\mathcal{A}_L,\dots,\mathcal{A}_{j+1},{\sf a}_{j,1},\dots, {\sf
a}_{j,i-1}][{\sf a}_{j,i}]\setminus\mathbb{K}[\mathcal{S},
\mathcal{A}_L,\dots,\mathcal{A}_{j+1},{\sf a}_{j,1},\dots, {\sf
a}_{j,i-1}]$}}
\item $\deg_{{\sf{a}}_{j,i}}(g)=\delta$.
\end{itemize}
Note that we can view $\zeta$ as a function $\zeta: \{1\ldots L\}\times \{1\ldots m\} \rightarrow \NN$, that is, as a function
with exactly the same range of $\eta$.

If $g \in G_{\delta}^{\, j,i}$, then we can write uniquely $g$ as
$$ g= a_{\delta} {\sf a}_{j,i}^{\delta}+a_{\delta-1}{\sf a}_{j,i}^{\delta-1}+ \cdots +a_0,$$
  with $a_j \in \mathbb{K}[\mathcal{S},\mathcal{A}_L,\dots,\mathcal{A}_{j+1},{\sf a}_{j,1},\dots, {\sf a}_{j,i-1}]$.
and $a_{\delta}$ is the leading polynomial of $g$.

\noindent We name the elements of $G_{\delta}^{\, j,i}$ according to the term
order of their leading terms, $i.e.$
$G_{\delta}^{\, j,i}=\{g^{(i)}_{j,\delta,1},
\dots,g^{(i)}_{j,\delta,|G_{\delta}^{\, j,i}|}\}$, with ${\bf
T}(g_{j,\delta,h}^{(i)}) <  {\bf T}(g_{j,\delta,h+1}^{(i)})$ for
any $h$.

\noindent  We note the following lemma.
\begin{lemma}\label{lemgradmax}
For any $ j=1,\dots,L$ and  $i=1,\dots, m$,
$G_{\zeta(j,i)}^{\, j,i}=\{g^{(i)}_{j,\zeta(j,i),1}\}$, i.e. there
exists only one polynomial in $G_{\zeta(j,i)}^{\, j,i}$ such that
$\deg_{{\sf{a}}_{j,i}}=\zeta(j,i)$.

\begin{proof}
From elementary properties of \Gr\ bases of zero-dimensional ideals, for any variable ${\sf a}_{j,i}$, $G$ must contain a polynomial $g$ with leading term ${\sf a}_{j,i}^k$, for some $k\geq 1$. Note that $g\in G^{\, j,i}$, because variable ${\sf a}_{j,i}^k$ is present in $g$ and any greater variable cannot be present. If there is a $\overline{g}\in G^{\, j,i}$ with $\deg_{{\sf a}_{j,i}} \overline{g} \geq k$, then ${\sf a}_{j,i} | \textbf{T}(\overline{g})$ and so $\overline{g}$ can be removed (recall that $G$ is reduced).
As a consequence, $g$ has the highest possible degree in ${\sf a}_{j,i} $, i.e. $k=\zeta(j,i)$, and so $g=g^{(i)}_{j,\zeta(j,i),1}$.
\end{proof}
\end{lemma}
We are ready for the main result of this section (whose proof is given in Section \ref{dim}). Compare with Theorem 32 in \cite{CGC-cd-art-martamax}.
\begin{proposition}\label{GBstructure}
Let $G$ be a reduced \Gr \ basis of a radical weakly stratified
ideal $J$ with respect to $<$ as previously described. Let $\mathcal{V}(J)\subset \mathbb{A}$.  Then for any $ j=1,\dots,L$ and  $i=1,\dots, m$,
{\small{$$
G^{\, j,i}= \sqcup_{\delta=1}^{{\zeta(j,i)}} G_{\delta}^{\, j,i},\mbox{with}
$$}}
\vspace{-0.2cm}
\begin{itemize}
\item[1.]$\zeta(j,i)=\eta(j,i)$, i.e. $\zeta$ is the level function of
$J$;
\item[2.] 
$G_{\delta}^{\, j,i}\not=\emptyset$  for any $1\leq\delta\leq \zeta(j,i)$;
\item[3.] 
$G_{\zeta(j,i)}^{\, j,i}=\{g^{(i)}_{j,\zeta(j,i),1}\}$, i.e. there
exists only one polynomial in $G_{\zeta(j,i)}^{\, j,i}$ such that
$\deg_{{\sf{a}}_{j,i}}=\zeta(j,i)$;
\item[4.]  we have that
{\small{$${\bf T}(g_{j,\zeta(j,i),1}^{(i)})={\sf{a}}_{j,i}^{\zeta(j,i)}.$$}}
\end{itemize}
\end{proposition}

Note that it is  the radicality that ensures $1.$, but in later situations we will have $1.$ also without radicality.


\section{Proof of Proposition \ref{GBstructure}}
\label{dim}
\subsection{Preliminaries of proof}

To prove Proposition \ref{GBstructure} we  use this classical theorem:
\begin{theorem}[Buchberger-M\"{o}ller,
\cite{CGC-alg-art-buchmoeller,CGC-cd-inbook-D1morafglm}]\label{teoBM}
Let $H'\subset H$ be ideals in $\mathbb{K}[V_1,\ldots,V_{\mathcal{N}}]$ such that:
\begin{itemize}
    \item[(i)] there is a $\mathbb{K}$-linear map $\theta:H\longmapsto \mathbb{K}$ s.t. $\ker(\theta)=H'$,
    \item[(ii)] there are $\mathcal{N}$ field elements $\{\beta_k\}_{1\leq k\leq N}\subset\mathbb{K}$ s.t. $(V_k-\beta_k)H\subset H'$ for $1\leq k\leq \mathcal{N}$, that is, $\theta((V_k-\beta_k)f)=0$ for all $f\in H$.
\end{itemize}
Let $W$ be a strictly ordered \Gr\ basis of $H$ relative to a term order $<$, then a \Gr\ basis $W'$ of $H'$ w.r.t $<$ can be constructed as follows:
\begin{enumerate}
    \item compute $\alpha_g=\theta(g)$ for all $g\in W$.
    \item if $\alpha_g=0$ for all $g$, then $W=W'$, which happens if and only if  $H=H'$ and $\theta=0$ in $Hom_{\mathbb{K}}(H,\mathbb{K})$.
    \item otherwise, let $g^*$ be the least $g$ such that $\alpha_g\not=0$.
\end{enumerate}
We  have $W'=W_1\cup W_2\cup W_3$, with
\begin{itemize}
    \item $W_1=\{g\mid g<g^*\}$,
\item $W_2=\{(V_k-\beta_k)g^*\mid 1\leq k\leq \mathcal{N}\}$,
\item $W_3=\{g-\frac{\alpha_g}{\alpha_{g^*}}g^*\mid g>g^*\}$.
\end{itemize}
\end{theorem}

\begin{remark}\label{wow}
In the proof of Theorem \ref{teoBM}, the hypothesis ($ii$) is used only to prove that $W_2\subset H'$. Therefore, Theorem \ref{teoBM} still  holds if  we replace ($ii$) with a much weaker hypothesis, that is,
\begin{itemize}
    \item[$(iii)$] there are $\mathcal{N}$ field elements $\{\beta_k\}_{1\leq k\leq N}\subset\mathbb{K}$ s.t. $(V_k-\beta_k)g^*\in H'$, $1\leq k\leq \mathcal{N}$, where $g^*$ is as in (3). 
\end{itemize}
\end{remark}

\begin{remark}\label{remA}
Let $G$ be a \Gr\ basis of an ideal $I$ with respect to a term ordering $>$ and let $g_1,g_2\in G$ be  such that ${\bf T}(g_1)|{\bf T}(g_2)$. Then $G\backslash \{g_2\}$ is  again a \Gr\ basis of $I$. Therefore, any time there is a redundant basis element, we can remove it.
\end{remark}

From the remainder of this section, we fix $1\leq i\leq m$ and $1\leq j\leq L$, and we extend the projection
\begin{equation}\label{pi}
    \pi:\mathcal{ V}(J_{\mathcal{S},\mathcal{A}_{L},\dots,\mathcal{A}_{j+1},
{\sf{a}}_{j,1},\dots,{\sf{a}}_{j,i-1},
{\sf{a}}_{j,i}})\rightarrow\mathcal{ V}(J_{\mathcal{S},\mathcal{A}_{L},\dots,\mathcal{A}_{j+1},{\sf{a}}_{j,1},\dots,
{\sf{a}}_{j,i-1}})
\end{equation}
to 
$$\pi:\overline{\mathbb{K}}^{N+(L-j)m+i}\rightarrow\overline{\mathbb{K}}^{N+(L-j)m+i-1}$$
Coherently, we consider only the variable ${\sf{a}}_{j,i}$ in the block $A_j$.\\ 

\begin{remark}\label{remB}
To simplify the notation in the proof, we use $\tau$ as a symbol with a special meaning, as follows. We introduce $\tau$ to single out the contribution of variable ${\sf{a}}_{j,i}$. Any non-zero element of $\mathbb{K}[S,\mathcal{A}_{L},\dots,\mathcal{A}_{j+1},{\sf{a}}_{j,1},\dots,
{\sf{a}}_{j,i-1}]$ may be written as $\tau$ and we use $\approxeq$ to express this unconventional identification. For example, ${\sf{a}}_{L,1}\approxeq\tau$ and $1\approxeq\tau$ but also $\tau {\sf{a}}_{L,1}\approxeq\tau$ and   ${\sf{a}}_{L,1}{\sf{a}}_{j,i}\approxeq \tau {\sf{a}}_{j,i} \not\approxeq \tau$ and  $s_1{\sf{a}}_{j,i}^2\approxeq\tau {\sf{a}}_{j,i}^2 \approxeq {\sf{a}}_{L,2} {\sf{a}}_{j,i}^2$.
\end{remark}

$\,$\\
Let $H$ be a zero-dimensional ideal in $\mathbb{K}[S,\mathcal{A}_{L},\dots,\mathcal{A}_{j+1},{\sf{a}}_{j,1},\dots,
{\sf{a}}_{j,i}]$.\\ Let $W$ be its \Gr\ basis. Denote with\\ {\small$\overline{W}=W\cap(\mathbb{K}[S,\mathcal{A}_{L},\dots,\mathcal{A}_{j+1},
{\sf{a}}_{j,1},\dots,{\sf{a}}_{j,i}]\backslash\mathbb{K}[S,\mathcal{A}_{L},\dots,
\mathcal{A}_{j+1},{\sf{a}}_{j,1},\dots,{\sf{a}}_{j,i-1}])$} and  $\widehat{W}=W\cap(\mathbb{K}[S,\mathcal{A}_{L},\dots,\mathcal{A}_{j+1},
{\sf{a}}_{j,1},\dots,{\sf{a}}_{j,i-1}]),$ so that $W=\overline{W}\sqcup \widehat{W}$.
With the $\tau$ notation, we have $$\widehat{W}\approxeq\{\tau,\ldots,\tau\}\textrm{ and  }\overline{W}\subset\{\tau  {\sf{a}}_{j,i}+\tau,\ldots,\tau  {\sf{a}}_{j,i}+\tau,\tau  {\sf{a}}_{j,i}^2+\tau  {\sf{a}}_{j,i}+\tau,\ldots\}.$$ In the same way we can denote $$\overline{H}=H\cap(\mathbb{K}[S,\mathcal{A}_{L},\dots,\mathcal{A}_{j+1},{\sf{a}}_{j,1},\dots,
{\sf{a}}_{j,i}]\backslash\mathbb{K}[S,\mathcal{A}_{L},\dots,\mathcal{A}_{j+1},{\sf{a}}_{j,1},\dots,
{\sf{a}}_{j,i-1}])$$ and  $\widehat{H}=H\cap(\mathbb{K}[S,\mathcal{A}_{L},\dots,\mathcal{A}_{j+1},
{\sf{a}}_{j,1},\dots,
{\sf{a}}_{j,i-1}])$.\\

\begin{remark}\label{remC}
Suppose we want to compute the ideal $H'$ from $H$ by adding a point $Q=(P,\overline{{\sf{a}}}_{j,i})$, with $P=(\overline{s}_1,\ldots,\overline{s}_N,\overline{{\sf{a}}}_{L,1},\ldots,
\overline{{\sf{a}}}_{j,i-1})$. We apply Theorem~\ref{teoBM} to compute $W'$ from $W$ using the point evaluation $\theta(g)=g(Q)$. In this case it is easy to see that we can take as $\beta_i$ the $i$-th component of $Q$. There are two distinct cases:
\begin{enumerate}
    \item \textit{either} for all $g\in \widehat{W}$, $g(Q)=g(P)=0$,
    \item \textit{or} there exists $g\in \widehat{W}$ such that $g(Q)=g(P)\not=0$.
\end{enumerate}
The first case implies $g^*\in\overline{W}$, the second case implies $g^*\in\widehat{W}$. Since these are logically distinct, we can conclude that there are only two (distinct) cases:
\begin{enumerate}
    \item \textit{either} for all $g\in \widehat{W}$, $g(Q)=g(P)=0$, and this happens if and only if $g^*\in\overline{W}$,
    \item \textit{or} there exists $g\in \widehat{W}$ such that $g(Q)=g(P)\not=0$, this happens if and only if  $g^*\in\widehat{W}$. 
\end{enumerate}
\end{remark}

\subsection{Sketch of proof}
Let us consider $g=g_{j,\zeta(j,i),1}^{(i)}$ and
$\Delta=\eta(j,i)$.\\
Let 
$I=J\cap(\mathbb{K}[S,\mathcal{A}_{L},\dots,\mathcal{A}_{j+1},{\sf{a}}_{j,1},\dots,
{\sf{a}}_{j,i}])$. Since  $\mathcal{V}(I)\subset \mathbb{A}_{j,i}$ and $I$ is radical and zero-dimensional, $I=\mathcal{I}(\mathcal{V}(I))=\mathcal{I}
(\Sigma_1^{j,i}\sqcup\ldots\sqcup\Sigma_{\Delta}^{j,i})$. Since $J$ is  weakly-stratified, we will have $\Sigma_h^{j,i}\neq \emptyset$ for all $1\leq k\leq \Delta$.\\
Our proof needs several steps:
\begin{itemize}
    \item \textbf{Step \texttt{I}}.\\ 
We consider $P_1\in \Sigma_1^{j,i}$, $P_2\in \Sigma_2^{j,i}$ and $\pi$ as in (\ref{pi}). We are interested in the leading terms of the \Gr\ basis
    of $\mathcal{I}(\pi^{-1}(P_1))$ and of $\mathcal{I}(\pi^{-1}(P_1)\cup \pi^{-1}(P_2))$. However, the exact knowledge of these leading terms is unnecessary and it is sufficient for us to determine their expression in the $\tau$ notation. We perform this step in Subsection \ref{sub1}.
\item \textbf{Step \texttt{II}}.\\
Generalising the previous argument, in Subsection \ref{sub2} (Lemma \ref{lemma1}) we take any $2\leq t\leq \Delta$ and consider any point $P_h$ in $\Sigma_h^{j,i}$ for all $1\leq h\leq t$. We describe the leading terms of the \Gr\ basis of $\mathcal{I}(\pi^{-1}(P_1)\cup\ldots\cup \pi^{-1}(P_t))$. Since we need an induction on the number of points to prove Lemma \ref{lemma1}, we give an intermediate lemma: Lemma \ref{lemma2}.
\item \textbf{Step \texttt{III}}.\\
As the leading terms of the \Gr\ basis of $\mathcal{I}(\pi^{-1}(P_1)\cup\ldots\cup \pi^{-1}(P_{\Delta}))$ are already in the desired shape, in Lemma \ref{lemma3} we show that adding more points does not change the shape of the leading terms of the \Gr\ basis, as long as the points come from some  $\Sigma_h^{j,i}$ with $h\leq \Delta$.
\end{itemize}

\subsection{First part of the proof}

\label{sub1}
We use the approach of Remark \ref{remC}.
\begin{itemize}
    \item Let $P_1=(\overline{s}_1,\ldots,\overline{s}_N,\overline{{\sf{a}}}_{L,1},\ldots,
\overline{{\sf{a}}}_{j,i-1})\in \Sigma_1^{j,i}$ and $H=\mathcal{I}(\pi^{-1}(P_1))$ be the vanishing ideal  of  $\pi^{-1}(P_1)$.\\ Then $\pi^{-1}(P_1)=\{(\overline{s}_1,\ldots,\overline{s}_N,
\overline{{\sf{a}}}_{L,1},\ldots,\overline{{\sf{a}}}_{j,i-1},
\overline{{\sf{a}}}_{j,i})\}.$
The basis $W=\textrm{GB}(H)$ is $W=\{s_1-\overline{s}_1,\ldots,s_N-\overline{s}_N,{\sf{a}}_{L,1}-\overline{{\sf{a}}}_{L,1}
\ldots,{\sf{a}}_{j,i}-\overline{{\sf{a}}}_{j,i}\}$.
Using our notation we have
\begin{equation}\label{t1} 
{\bf T}(W)=\{\tau,\ldots,\tau,{\sf{a}}_{j,i}\}.
\end{equation} 

\item We consider a point  $P_2\in \Sigma_2^{j,i}$ that, with abuse of notation\footnote{  Where we do not imply that the components of $P_2$ are the same as those of $P_1$, although we use the same symbols.}, we write $P_2=(\overline{s}_1,\ldots,\overline{s}_N,\overline{{\sf{a}}}_{L,1},\ldots,
\overline{{\sf{a}}}_{j,i-1})$.  We  can write
$$\pi^{-1}(P_2)=\left\{\begin{array}{l}
Q_1=(\overline{s}_1,\ldots,\overline{s}_N,\overline{{\sf{a}}}_{L,1},\ldots,
\overline{{\sf{a}}}_{j,i-1},\overline{{\sf{a}}}^{(1)}_{j,i})\\

Q_2=(\overline{s}_1,\ldots,\overline{s}_N,\overline{{\sf{a}}}_{L,1},\ldots,
\overline{{\sf{a}}}_{j,i-1},\overline{{\sf{a}}}^{(2)}_{j,i})
\end{array}\right.$$

\begin{itemize}
    \item[$*$]\underline{ We add the point $Q_1$}.\\
Using Theorem \ref{teoBM} we can build $W'$ from $W$ in (\ref{t1}). If $\forall g\in \widehat{W}$, $g(Q_1)=0$, then $\pi(Q_1)\in \mathcal{V}(\widehat{H})$. But  $\pi(Q_1)=P_2$ and  $\mathcal{V}(\widehat{H})=\{P_{1}\}$, so $P_1=P_2$ and $|\pi^{-1}(P_2)|=3$, which is impossible because $P_2\in \Sigma_2^{j,i}$. Therefore, for Remark \ref{remC}, $g^*\in \widehat{W}$.\\
So the \Gr\ basis $W'=W_1\sqcup W_2\sqcup W_3$,  where
\begin{itemize}
    \item[-] {\scriptsize{$W_1=\{g\in \widehat{W}\mid g<g^*\}$}} because {\scriptsize{$g^*\in \widehat{W}$}}, so we have {\scriptsize{$W_1\approxeq\{\tau,\ldots,\tau\}$}} and {\scriptsize{$\textbf{T}(W_1)\approxeq\{\tau,\ldots,\tau\}$}}.
 \item[-] $W_2$ is composed by the following polynomials
\begin{itemize}
\item[] {\scriptsize{$g^*(s_1-\overline{s}_1),\ldots,g^*(s_N-\overline{s}_N)$}}
\item[] {\scriptsize{$g^*({\sf{a}}_{L,1}-\overline{{\sf{a}}}_{L,1}),\ldots,g^*({\sf{a}}_{j,i-1}-
\overline{{\sf{a}}}_{j,i-1})$}}
\item[] {\scriptsize{$g^*({\sf{a}}_{j,i}-\overline{{\sf{a}}}^{(1)}_{j,i})$}}
\end{itemize}
\item[-] {\scriptsize{$W_3=\{g-\frac{g(Q_1)}{g^*(Q_1)}g^*\mid g>g^*\}$}}.
\end{itemize}
We have {\scriptsize{${\bf T}(W_2)\approxeq\{\tau,\ldots,\tau,\tau {\sf{a}}_{j,i}\}$}} and {\scriptsize{${\bf T}(W_3)\subseteq\{\tau,\ldots,\tau, {\sf{a}}_{j,i}\}$}} and ${\sf{a}}_{j,i}\in {\bf T}(W_3)$. With {\scriptsize{${\bf T}(W_3)\subseteq\{\tau,\ldots,\tau, {\sf{a}}_{j,i}\}$}} we actually mean that ${\bf T}(W_3)$ is a subset of a set $S$ such that {\scriptsize{$S\approx\{\tau,\ldots,\tau, {\sf{a}}_{j,i}\}$}}. We will write similarly from now on without any further comment.
Observe that {\scriptsize{${\bf T}(W')\approxeq\{\tau,\ldots,\tau,\tau {\sf{a}}_{j,i},{\sf{a}}_{j,i}\}$}}. By Remark \ref{remA}, we have
{\scriptsize{${\bf T}(W')\approxeq\{\tau, \ldots, \tau, {\sf{a}}_{j,i}\}$}}.

\item[$*$] \underline{We add the point $Q_2$}.\\ Let\footnote{With ''let $W:=W'$'' we mean that in this proof step we remove all elements in set $W$ and instead we insert into $W$ all elements from $W'$. After that, we remove all elements from $W'$. We also forget the values of $g^*$ and $W_1,W_2,W_3$.} $W:=W'$ and let us use again  Theorem \ref{teoBM}.\\
We have to find a polynomial $g^*\in W$ such that {\small{$g^*(Q_2)\not=0$}}. Of course {\small{$g^*\not\in \widehat{W}$}}, because {\small{$\pi(Q_1)=\pi(Q_2)=P_2$}}. Thus {\small{$g^* \in \overline{W}$}} and {\small{$g^*={\sf{a}}_{j,i}+\tau$}}.\\
$W'$ is formed by $W'=W_1\sqcup W_2\sqcup W_3$,  where

\begin{itemize}
    \item[-]{\scriptsize{$W_1\approxeq\{\tau,\ldots,\tau\}$}},
 \item[-] {\scriptsize{$W_2=W_{2,1}\cup W_{2,2}$}} where\\
{\scriptsize{$W_{2,1}=\{g^*(s_1-\overline{s}_1),\ldots,g^*(s_N-\overline{s}_N),
g^*({\sf{a}}_{L,1}-\overline{{\sf{a}}}_{L,1}),\ldots,g^*({\sf{a}}_{j,i-1}-
\overline{{\sf{a}}}_{j,i-1})$}}, so\\ {\scriptsize{$ \textbf{ T}(W_{2,1})\approxeq\{\tau {\sf{a}}_{j,i},\ldots,\tau {\sf{a}}_{j,i}\}$}}.\\ 
{\scriptsize{$W_{2,2}=\{g^*({\sf{a}}_{j,i}-\overline{{\sf{a}}}^{(2)}_{j,i})\}\implies \textbf{T}(W_{2,2})=\{{\sf{a}}_{j,i}^2\}$}}.
\item[-] {\scriptsize{$W_3=\emptyset$}}.
\end{itemize}
\end{itemize}
So \begin{equation}\label{t2} 
{\bf T}(W')=\{\tau,\ldots,\tau,\tau {\sf{a}}_{j,i},\ldots,\tau {\sf{a}}_{j,i},{\sf{a}}_{j,i}^2\}\end{equation}
\end{itemize}

\subsection{Second part of proof}\label{sub2}
If $\Delta\leq 2$, we have finished our proof. Otherwise, i.e. $\Delta\geq 3$, we want to prove, using induction on $t$ with $1\leq t\leq \Delta$, the following
\begin{lemma}\label{lemma1}
The \Gr\  basis $W$ of $H=\mathcal{I}(\pi^{-1}(P_1)\cup \ldots \cup \pi^{-1}(P_t))$, where $1\leq t\leq \Delta$ and $P_h$ is any point in $\Sigma^{j,i}_h$ for $1\leq h\leq t$, is such that
\begin{equation}\label{W} 
{\bf T}(W)\approxeq\{\tau,\ldots,\tau,\tau {\sf{a}}_{j,i},\ldots,\tau {\sf{a}}_{j,i},\ldots,\tau {\sf{a}}_{j,i}^{t-1},\ldots,\tau {\sf{a}}_{j,i}^{t-1},{\sf{a}}_{j,i}^{t}\}
\end{equation}

\begin{proof}
The \Gr\ basis with $t=1$  and $t=2$ were just shown in (\ref{t1}) and (\ref{t2}) respectively.\\
By induction  we suppose to have $t-1$ points $\{P_1,\ldots,P_{t-1}\}$ and to have a \Gr\ basis $W$ such that:
\begin{equation}\label{G2}
{\bf T}(W)\approxeq\{\tau,\ldots,\tau,\tau {\sf{a}}_{j,i},\ldots,\tau {\sf{a}}_{j,i},\ldots,\tau {\sf{a}}_{j,i}^{t-2},\ldots,\tau {\sf{a}}_{j,i}^{t-2},{\sf{a}}_{j,i}^{t-1}\}
\end{equation}

Now we can prove the $t$-th step. In order to do it, we prove the following lemma, with its long proof between horizontal lines.\\

\hrule
\begin{lemma}\label{lemma2}
Let $3\leq t\leq \Delta$ and  $P_t\in \Sigma_t^{j,i}$ with $\pi^{-1}(P_t)=\{Q_1,\ldots,Q_t\}$.  For any  $1\leq u\leq t-1$, let $H^u$ be the vanishing ideal
$$H^u= \mathcal{I}(\pi^{-1}(P_1)\cup \ldots \cup \pi^{-1}(P_{t-1})\cup \{Q_1,\ldots,Q_{u}\})\textrm{ and}$$
$$H^0= \mathcal{I}(\pi^{-1}(P_1)\cup \ldots \cup \pi^{-1}(P_{t-1})).$$ 
Let $W^u$ be its reduced \Gr\ basis, then
\begin{center}
${\bf T}(W^u)$ has the same structure as ${\bf T}(W)$ in $(\ref{G2})$.
\end{center} 

\begin{proof} Let $P_t=(\overline{s}_1,\ldots,\overline{s}_N,\overline{{\sf{a}}}_{L,1},\ldots,
\overline{{\sf{a}}}_{j,i-1})$. Since $P_t\in \Sigma_t^{j,i}$, then
{\small$$\pi^{-1}(P_t)=\left\{\begin{array}{l}
Q_1=(\overline{s}_1,\ldots,\overline{s}_N,\overline{{\sf{a}}}_{L,1},\ldots,
\overline{{\sf{a}}}_{j,i-1},\overline{{\sf{a}}}^{(1)}_{j,i})\\
Q_2=(\overline{s}_1,\ldots,\overline{s}_N,\overline{{\sf{a}}}_{L,1},\ldots,
\overline{{\sf{a}}}_{j,i-1},\overline{{\sf{a}}}^{(2)}_{j,i})\\
\quad\,\,\,\,\,\vdots \\
Q_t=(\overline{s}_1,\ldots,\overline{s}_N,\overline{{\sf{a}}}_{L,1},\ldots,
\overline{{\sf{a}}}_{j,i-1},\overline{{\sf{a}}}^{(t)}_{j,i})
\end{array}\right.$$}
We prove the lemma by induction on $u$.
\begin{itemize}
    \item[$(a)$] We know that $W^0$ is as in (\ref{G2}). We add  point $Q_1$ to $H^0$.\\
Using  Theorem \ref{teoBM} we can build $W^1$ from $W^0$ as usual. We adopt the ''$W,W'$'' notation. If $\forall g\in \widehat{W}$, $g(Q_1)=0$, then $\pi(Q_1)\in \mathcal{V}(\widehat{H})$. But  $\pi(Q_1)=P_t$ and  $\mathcal{V}(\widehat{H})=\{P_1,\ldots,P_{t-1}\}$, so $P_t=P_k$ for some $1\leq k\leq t-1$, and $ |\pi^{-1}(P_{k})|=k+1$ which is impossible because $P_k\in \Sigma_k^{j,i}$. Therefore, for Remark \ref{remC}, $g^*\in \widehat{W}$.\\
So the \Gr\ basis $W'$ is
formed by the union of these sets:

\begin{itemize}
    \item[-]{\scriptsize{$W_1=\{g\in \widehat{W}\mid g<g^*\}$}}. Since {\scriptsize{$g^*\in \widehat{W}$}} then {\scriptsize{$ {\bf T}(W_1)=\{\tau,\ldots,\tau\}$}},
 \item[-] {\scriptsize{$W_2=W_{2,1}\cup W_{2,2}$}} where\\
{\scriptsize{$W_{2,1}=\{g^*(s_1-\overline{s}_1),\ldots,g^*(s_N-\overline{s}_N),
g^*({\sf{a}}_{L,1}-\overline{{\sf{a}}}_{L,1}),\ldots,g^*({\sf{a}}_{j,i-1}-
\overline{{\sf{a}}}_{j,i-1})$}},\\ so {\scriptsize{$  {\bf T}(W_{2,1})=\{\tau ,\ldots,\tau\}$}}. \\
{\scriptsize{$W_{2,2}=\{g^*({\sf{a}}_{j,i}-\overline{{\sf{a}}}^{(1)}_{j,i})\}\implies {\bf T}(W_{2,2})=\{\tau {\sf{a}}_{j,i}\}$}}.
\item[-] {\scriptsize{$W_3=\{g-\frac{g(Q_1)}{g^*(Q_1)}g^*\mid g>g^*\}$}} and so the leading terms of $W_3$ are those in {\scriptsize{${\bf T}(W)$}}, except possibly for $\tau$.
\end{itemize}

\noindent Therefore $W^1=W'$ has the same structure of $W^0=W$ in (\ref{G2}) (because $\tau {\sf{a}}_{j,i}$ is already present in (\ref{G2})).

 \item[$(b)$] We add the point $Q_2$ to $H^1$ and we compute $W^2$.\\ Let $W:=W'$ and we use again Theorem \ref{teoBM}.\\
We find $g^*\in W$ such that $g^*(Q_2)\not=0$. We are sure that $g^*\not\in \widehat{G}$, because $\pi(Q_1)=\pi(Q_2)=P_t$, and so $g^*\in \overline{G}$. We can claim:\\ 
\textbf{Claim}\label{claim} ${\bf T}(g^*)\approxeq \tau {\sf{a}}_{j,i}$. 
\begin{proof}
The Gianni-Kalkbrener theorem (\cite{CGC-alg-art-gianni}, \cite{CGC-alg-art-kalkbrener}) says that there exists a polynomial $g\in \overline{W}$ such that $$g(P_t,{\sf{a}}_{j,i})=g(\overline{s}_1,\ldots,\overline{s}_N,\overline{{\sf{a}}}_{L,1},
\ldots,\overline{{\sf{a}}}_{j,i-1}
,{\sf{a}}_{j,i})\not= 0\textrm{ in }\mathbb{K}[{\sf{a}}_{j,i}]$$ and the solutions of $g(P_t,{\sf{a}}_{j,i})$ are exactly the extensions of $P_t$. In $\mathcal{V}(H)$ we have only one extension of $P_t$ (which is $Q_1$), so the degree of  $g$ w.r.t. ${\sf{a}}_{j,i}$ must be $1$  and so $g\approxeq\tau {\sf{a}}_{j,i}+\tau.$\\
Let $g$ be the smallest polynomial of this kind. We have that {\small{$g^*=g$}}, because {\small{$g(Q_1)=0$, $g(Q_2)\not=0$}} and all smaller polynomials  vanish at $Q_2$. 
\end{proof}
So $W'$ is the union of
\begin{itemize}
    \item[-]{\scriptsize{$W_1=\{g\in \widehat{W}\mid g<g^*\}$}}.\\ Since {\scriptsize{$g^*=\tau {\sf{a}}_{j,i}+\tau$}} then {\scriptsize{$ {\bf T}(W_1)=\{\tau,\ldots,\tau\}$}} or {\scriptsize{$ {\bf T}(W_1)=\{\tau,\ldots,\tau,\tau {\sf{a}}_{j,i},\ldots,\tau {\sf{a}}_{j,i}\}$}},
 \item[-] {\scriptsize{$W_2=W_{2,1}\cup W_{2,2}$}} where\\
 {\scriptsize{$W_{2,1}=\{g^*(s_1-\overline{s}_1),\ldots,g^*(s_N-\overline{s}_N),
g^*({\sf{a}}_{L,1}-\overline{{\sf{a}}}_{L,1}),\ldots,g^*({\sf{a}}_{j,i-1}-
\overline{{\sf{a}}}_{j,i-1})\}$}},\\ so {\scriptsize{$  {\bf T}(W_{2,1})=\{\tau {\sf{a}}_{j,i},\ldots,\tau {\sf{a}}_{j,i}\}$}}. 
{\scriptsize{$W_{2,2}=\{g^*({\sf{a}}_{j,i}-\overline{{\sf{a}}}^{(2)}_{j,i})\}\implies {\bf T}(W_{2,2})=\{\tau {\sf{a}}_{j,i}^2\}$}}.
\item[-] {\scriptsize{$W_3=\{g-\frac{g(Q_2)}{g^*(Q_2)}g^*\mid g>g^*\}$}} so the leading terms of {\small{$W_3$}} are those in {\small{${\bf T}(W)$}}, except possibly for $\tau$ and $\tau {\sf{a}} _{j,i}$.
\end{itemize}
If $t=3$, we have that  {\small{${\sf{a}}_{j,i}^2\in{\bf T}(W)$}}, and so any leading term $\tau {\sf{a}}_{j,i}^2$ can be removed (by Remark \ref{remA}) and we obtain again that the structure of $W'=W^2$ is as in (\ref{G2}). Otherwise ($t\geq 4$), the leading term $\tau {\sf{a}}_{j,i}^2$ remains and we still have  the structure of (\ref{G2}). 

\item[$(c)$]We proceed inductively on $u$ until we are left to add the point $Q_{t-1}$.

\item[$(d)$] We add $Q_{t-1}$.\\ 
In this case {\small{$H=\mathcal{I}(\pi^{-1}(P_1)\cup \ldots \cup \pi^{-1}(P_{t-1})\cup\{Q_1,\ldots,Q_{t-2 }\})$}} and $W^{t-2 }$ has (by induction on $u$) the structure of (\ref{G2}). Let {\small{$W=W^{t-2 }$}} and {\small{$W'=W^{t-1}$}}. We apply Theorem \ref{teoBM}.\\
We have to find $g^*\in W$ such that {\small{$g^*(Q_{t-1})\not=0$}}. Exactly as before, {\small{$g^*\not\in \widehat{W}$}}. We know that {\small{${\bf T}(g^*)\approxeq\tau {\sf{a}}_{j,i}^{t-2}.$}} To prove it we might use the Gianni-Kalkbrener theorem (\cite{CGC-alg-art-gianni}, \cite{CGC-alg-art-kalkbrener}) repeating the reasoning of our Claim on page \pageref{claim}.

So $W'$ is the union of the following sets:

\begin{itemize}
    \item[-] {\scriptsize{$W_1=\{g\in W\mid g<g^*\}$}}. Since {\scriptsize{${\bf T}(g^*)\approxeq\tau {\sf{a}}_{j,i}^{t-2 }$,\\ ${\bf T}(W_1)=\{\tau,\ldots,\tau,\tau {\sf{a}}_{j,i},
\ldots,
\tau {\sf{a}}_{j,i},
\ldots,
\tau {\sf{a}}_{j,i}^{t-3},
\ldots,
\tau {\sf{a}}_{j,i}^{t-3}\}$}} or possibly also {\scriptsize{$\tau {\sf{a}}_{j,i}^{t-2}\in{\bf T}(W_1)$}},
 \item[-] {\scriptsize{$W_2=W_{2,1}\cup W_{2,2}$}} where\\
{\scriptsize{$W_{2,1}=\{g^*(s_1-\overline{s}_1),\ldots,g^*(s_N-\overline{s}_N), g^*({\sf{a}}_{L,1}-\overline{{\sf{a}}}_{L,1}),\ldots,g^*({\sf{a}}_{j,i-1}-
\overline{{\sf{a}}}_{j,i-1})\}$}}\\
{\scriptsize{$W_{2,2}=\{g^*({\sf{a}}_{j,i}-\overline{{\sf{a}}}^{(r+1)}_{j,i})\}$}}
\item[-]{\scriptsize{$W_3=\{g-\frac{g(Q_{r+1})}{g^*(Q_{r+1})}g^*\mid g>g^*\}$}}.
\end{itemize}
Since  {\small{$g^*=\tau {\sf{a}}_{j,i}^{t-2}+\ldots$}}, we have {\scriptsize{
${\bf T}(W_{2,1})=\{\tau {\sf{a}}_{j,i}^{t-2},\ldots,\tau {\sf{a}}_{j,i}^{t-2}\}$}} and {\scriptsize{${\bf T}(W_{2,2})=\{\tau {\sf{a}}_{j,i}^{t-1}\}$}}.\\
But in the \Gr\ basis $W$ in  (\ref{G2}) there exists a polynomial $\overline{g}$ such that {\small{${\bf T}(\overline{g})={\sf{a}}_{j,i}^{t-1}$}}.
So {\small{${\bf T}(\overline{g})|\tau {\sf{a}}_{j,i}^{t-1}$}} and we can remove the new term. Hence {\small{${\bf T}(W')$}} does not change and it remains as in (\ref{G2}).
\end{itemize}

Lemma \ref{lemma2} is proved. $\quad\Box$
\end{proof}
\end{lemma}

\hrule
$\,$\\
Now we know  {\small{${\bf T}(W^{t-1})$}}, which are the leading term  for the basis of {\small{$H=\mathcal{I}(\pi^{-1}(P_1)\cup \ldots \cup \pi^{-1}(P_{t-1})\cup \{Q_1,\ldots,Q_{t-1}\})$}}. We can add the point $Q_t$ and we use our ''$W,W'$'' notation. Using  Gianni-Kalkbrener's theorem we may prove as usual that {\small{ ${\bf T}(g^*)={\sf{a}}_{j,i}^{t-1}.$}}
So the leading terms of 
$$
g^*(s_1-\overline{s}_1),\ldots,g^*(s_N-\overline{s}_N),
g^*({\sf{a}}_{L,1}-\overline{{\sf{a}}}_{L,1}),\ldots,g^*({\sf{a}}_{j,i-1}-
\overline{{\sf{a}}}_{j,i-1})
$$
are all of the type $\tau {\sf{a}}_{j,i}^{t-1}$, while {\small{$g^*({\sf{a}}_{j,i}-\overline{{\sf{a}}}^{(t)}_{j,i})=
{\sf{a}}_{j,i}^{t}+\ldots$}}, so its leading term is {\small{${\sf{a}}_{j,i}^{t}$}}. The new leading terms are  {\small{$\{\tau {\sf{a}}_{j,i}^{t-1},\ldots,\tau {\sf{a}}_{j,i}^{t-1},{\sf{a}}_{j,i}^{t}\}$}}.
Therefore (Remark \ref{remA}),  the structure of $W'$ becomes the same as in (\ref{W}), because there are no other new terms, since {\small{$\{g>g^*\}=\emptyset$.}}\\
This concludes the proof of Lemma \ref{lemma1}.$\quad\Box$ \end{proof}

\end{lemma}
\begin{corollary}
With the above notation, if {\small{$H=\mathcal{I}(\pi^{-1}(P_1)\cup \ldots \cup \pi^{-1}(P_{\Delta}))$}}, then
\begin{equation}\label{4a} 
{\bf T}(W)\approxeq\{\tau,\ldots,\tau,\tau {\sf{a}}_{j,i},\ldots,\tau {\sf{a}}_{j,i},\ldots,\tau {\sf{a}}_{j,i}^{\Delta-1},\ldots,\tau {\sf{a}}_{j,i}^{\Delta-1},{\sf{a}}_{j,i}^{\Delta}\}
\end{equation}
\begin{proof}
Apply Lemma \ref{lemma1} with $t=\Delta$.
\end{proof}
\end{corollary}
\subsection{Third part of the proof}\label{sub3}

\begin{lemma}\label{lemma3}
Let {\small{$\mathcal{I}(\pi^{-1}(P_1)\cup\ldots\cup \pi^{-1}(P_{\Delta}))\supset H \supset J$}} be a radical zero-dimensional ideal. Suppose that the leading terms of its reduced \Gr\ basis satisfy (\ref{4a}). Let  {\small{$\dot{P_h}\in \Sigma_h^{j,i}$, $1\leq h\le \Delta$}} and let $H'=\mathcal{I}(\mathcal{V}(H)\cup \pi^{-1}(\dot{P_h}))$.\\ Then  {\small{$\mathcal{I}(\pi^{-1}(P_1)\cup\ldots\cup \pi^{-1}(P_{\Delta}))\supset H\supset H' \supset J$}} and the leading terms of its reduced \Gr\ basis satisfy (\ref{4a}).
\end{lemma}

\begin{proof}
We use our ''{\small{$W,W'$}}'' notation, so that {\small{$W=\textrm{GB}(H)$}} and  {\small{$W'=\textrm{GB}(H')$}}. 
Let us take a point\footnote{ With our usual abuse of notation.} $\dot{P_k}=(\overline{s}_1,\ldots,\overline{s}_N,\overline{{\sf{a}}}_{L,1},\ldots,
\overline{{\sf{a}}}_{j,i-1})\in\Sigma_k^{j,i}$ with $1\leq k\leq \Delta$.
{\small$$\pi^{-1}(\dot{P_k})=\left\{\begin{array}{l}
Q_1=(\overline{s}_1,\ldots,\overline{s}_N,\overline{{\sf{a}}}_{L,1},\ldots,
\overline{{\sf{a}}}_{j,i-1},\overline{{\sf{a}}}^{(1)}_{j,i})\\
\quad\,\,\,\,\,\vdots \\
Q_k=(\overline{s}_1,\ldots,\overline{s}_N,\overline{{\sf{a}}}_{L,1},\ldots,
\overline{{\sf{a}}}_{j,i-1},\overline{{\sf{a}}}^{(k)}_{j,i})
\end{array}\right.$$}
\begin{itemize}
    \item[$*$]We add the point  $Q_1$.\\
We build  $W'$ using Theorem \ref{teoBM}. We know that $g^*\in \widehat{W}$ (as in $(a)$ of Lemma \ref{lemma2}). So $W'=W_1\sqcup W_2\sqcup W_3$ where
\begin{itemize}
    \item[-] {\scriptsize{$W_1\approxeq\{\tau,\ldots,\tau\}$}}, because {\scriptsize{$g^*\in \widehat{W}$}}. So {\scriptsize{${\bf T}(W_1)\approxeq\{\tau,\ldots,\tau\}$}},
 \item[-] {\scriptsize{$W_2=W_{2,1}\cup W_{2,2}$}} where\\
{\scriptsize{ $W_{2,1}=\{g^*(s_1-\overline{s}_1),\ldots,g^*(s_N-\overline{s}_N)
,g^*({\sf{a}}_{L,1}-\overline{{\sf{a}}}_{L,1}),\ldots,g^*({\sf{a}}_{j,i-1}-
\overline{{\sf{a}}}_{j,i-1})$}},\\ so {\scriptsize{$  {\bf T}(W_{2,1})=\{\tau ,\ldots,\tau\}$}}. \\
{\scriptsize{$W_{2,2}=\{g^*({\sf{a}}_{j,i}-\overline{{\sf{a}}}^{(1)}_{j,i})\}$}} so {\scriptsize{$ {\bf T}(W_{2,2})=\{\tau {\sf{a}}_{j,i}\}$}}.
\item[-] {\scriptsize{$W_3=\{g-\frac{g(Q_1)}{g^*(Q_1)}g^*\mid g>g^*\}$}} and so the leading terms of $W_3$ are those in {\small{${\bf T}(W)$}}, except possibly for new $\tau$'s.
\end{itemize}
Therefore the structure of $W'$ is the same as that of $W$.
\item[$*$] We add $Q_{r+1}$ with $2\leq r+1\leq k$. We assume, using  induction on $r$, that $W$ verifies (\ref{4a}).\\
Let $W:=W'$ and we use again  Theorem \ref{teoBM}.\\
To construct $W'$ we have to find $g^*\in W$ such that $g^*(Q_{r+1})\not=0$. Exactly as in case $(d)$ of Lemma \ref{lemma2}, ${\bf T}(g^*)=\tau {\sf{a}}_{j,i}^r.$
So $W'=W_{1}\sqcup W_{2}\sqcup W_3$, where

\begin{itemize}
    \item[-] {\scriptsize{$W_1=\{g\in \widehat{W}\mid g<g^*\}$}} where {\scriptsize{${\bf T}(g^*)=\tau {\sf{a}}_{j,i}^r$}}.\\ So {\scriptsize{${\bf T}(W_1)=\{\tau,\ldots,\tau,\tau {\sf{a}}_{j,i},
\ldots,
\tau {\sf{a}}_{j,i},
\ldots,
\tau {\sf{a}}_{j,i}^{r-1},
\ldots,
\tau {\sf{a}}_{j,i}^{r-1}\}$}} or possibly also\\ {\scriptsize{$
\tau {\sf{a}}_{j,i}^{r}\in {\bf T}(W_1)$}},
 \item[-] {\scriptsize{$W_2=W_{2,1}\cup W_{2,2}$}} where\\
{\scriptsize{$W_{2,1}=\{g^*(s_1-\overline{s}_1),\ldots,g^*(s_N-\overline{s}_N), g^*({\sf{a}}_{L,1}-\overline{{\sf{a}}}_{L,1}),\ldots,g^*({\sf{a}}_{j,i-1}-
\overline{{\sf{a}}}_{j,i-1})\}$}},\\
{\scriptsize{$W_{2,2}=\{g^*({\sf{a}}_{j,i}-\overline{{\sf{a}}}^{(r+1)}_{j,i})\}$}}

\item[-]{\scriptsize{$W_3=\{g-\frac{g(Q_{r+1})}{g^*(Q_{r+1})}g^*\mid g>g^*\}$}}.
\end{itemize}
Now
\begin{itemize}
    \item[-] If {\small{$r+1\leq k\leq \Delta-1$}}, then the structure of {\small{$\textbf{T}(W')$}} does not change. In fact
{\scriptsize{$\textbf{T}(W')=\mathbf{T}(W_{1})\cup \mathbf{T}(W_{2})\cup \mathbf{T}(W_3)$}}, where

\begin{itemize}
    \item[-]{\scriptsize{${\bf T}(W_1)=\{\tau,\ldots,\tau,\tau {\sf{a}}_{j,i},
\ldots,
\tau {\sf{a}}_{j,i},
\ldots,
\tau {\sf{a}}_{j,i}^{r-1},
\ldots,
\tau {\sf{a}}_{j,i}^{r-1}\}$}} or possibly also\\ {\scriptsize{$\tau {\sf{a}}_{j,i}^{r}\in {\bf T}(W_1)$}},
 \item[-] {\scriptsize{$\mathbf{T}(W_2)=\mathbf{T}(W_{2,1})\cup \mathbf{T}(W_{2,2})$}} where
{\scriptsize{${\bf T}(W_{2,1})=\{\tau {\sf{a}}_{j,i}^{r},
\ldots,
\tau {\sf{a}}_{j,i}^{r}\}$}} and\\
{\scriptsize{${\bf T}(W_{2,2})=\{\tau {\sf{a}}_{j,i}^{r+1}\}$}}.
\item[-] The leading terms of $W_3$ are those in {\small{${\bf T}(W)$}} with degree (in {\small{${\sf{a}}_{j,i}$}}) at least {\small{$r+1$}}, 
plus possibly some  terms in {\small{${\bf T}(W)$}} of degree {\small{$r$}}
, that is, those greater than ${\bf T}(g^*)$.
\end{itemize}

\item[-] If $r+1=\Delta$ then ${\bf T}(g^*)\approxeq\tau {\sf{a}}_{j,i}^{\Delta-1}$, so the leading terms of
{\small{$$
g^*(s_1-\overline{s}_1),\ldots,g^*(s_N-\overline{s}_N),
g^*({\sf{a}}_{L,1}-\overline{{\sf{a}}}_{L,1}),\ldots,g^*({\sf{a}}_{j,i-1}-
\overline{{\sf{a}}}_{j,i-1})$$}} remain $\tau {\sf{a}}_{j,i}^{\Delta-1}$, but
$$g^*({\sf{a}}_{j,i}-\overline{{\sf{a}}}^{(\Delta)}_{j,i})\approxeq\tau {\sf{a}}_{j,i}^{\Delta}+\ldots \,. $$ 
Since in  $W$ there is a $\overline{g}$ such that ${\bf T}(\overline{g})={\sf{a}}_{j,i}^{\Delta}$ and ${\bf T}(\overline{g})|\tau {\sf{a}}_{j,i}^{\Delta}$, then (Remark \ref{remA}) the structure of $W'$ does not change and verifies (\ref{4a}).
\end{itemize}
\end{itemize}\end{proof}
We reiterate Lemma \ref{lemma3} starting from  $H=\mathcal{I}(\pi^{-1}(P_1)\cup\ldots\cup \pi^{-1}(P_{\Delta}))$ and adding all the sets $ \pi^{-1}(\dot{P_h})$ until all points in $\mathcal{V}(J)$ have been added. When we obtain $J$, we will have that its leading terms satisfy  (\ref{4a}), so point $(1)$ and $(2)$ of Proposition \ref{GBstructure} are proved. In particular, (\ref{4a}) proves also   $(3)$ and $(4)$.\\
The proof of Proposition \ref{GBstructure} is complete.


\section{Multi-dimensional general error locator polynomials}
\label{Loc}
The following theorem  ensures that our  weak
multi-dimensional general error locator polynomials (see Definition \ref{locDebole}) exist for any code.
\begin{theorem} \label{bomba}
Let $C=C^{\perp}(I,L)$ be an  affine--variety code with $d \geq
3$.  Then
\begin{itemize}
\item[i)]$J_{*}^{C,t}$ is a radical strongly multi-stratified
ideal w.r.t. the $X$ variables.
\item[ii)] A \Gr\ basis of $J_{*}^{C,t}$ contains a set of weak
multi-dimensional general error locator polynomials for $C$.
\end{itemize} 
\end{theorem}

\begin{proof}
\begin{itemize}
    \item[$i)$] We recall that $J_{*}^{C,t}$ is the ideal in {\small{$\FF_q[s_1,\dots,s_r,X_t, \dots, X_1,e_1,\dots,e_t]$}}
as defined in \eqref{fineideal}. We set  $H=J_{*}^{C,t}$.
We want to show that $H$  is a radical strongly  multi-stratified ideal with respect to the $X$ variables.
The radicality of $H$ is obvious since it contains the field equations for all variables.\\
Let us consider $\pi_j$ and $\rho_i$, $1 \leq j \leq t$ as in Definition \ref{strongly}
$$\begin{array}{c}
\pi_t : \mathcal{V}(H_{S,X_t}) \rightarrow  \mathcal{V}(H_{S}),
\quad 
\pi_j : \mathcal{V}(H_{S,X_t,\dots,X_{j}}) \rightarrow  \mathcal{V}(H_{S,X_t,\dots,X_{j+1}})\\
\rho_j:\,
{\mathcal{V}}(H_{\mathcal{S},X_t,\dots,X_{j+1},X_{j}})
\longrightarrow
{\mathcal{V}}(H_{X_{j}}),
\,\, j=1,\dots, L.
\end{array}
$$
By Definition \ref{strongly}, $H$ is a strongly multi-stratified ideal with respect to the $X$ variables if:
\begin{itemize}
\item[$a0)$] Let $Z_j=\rho_j(\mathcal{V}(H_{\mathcal{S},X_t,\dots,X_{j+1},X_{j}}))$, then $Z_j=Z_{\bar j}$ for any $1\le j\ne \bar j\le t$. In this case we use $Z=Z_j$.
Since the locations are only $\mathcal{V}(I)\cup \{P_0\}$, then $Z=\mathcal{V}(I)\cup \{P_0\}$.
\item[$a1)$] Let $1 \leq j \leq t-1$. For any $T\subset Z$  with $1\leq |T|\leq j$,  there is $\widetilde{v} \in 
\mathcal{V}(H_{S,X_t,\dots,X_{j+1}})$ such that $|\rho_j(\pi_j^{-1}\{\widetilde{v}\})| =T$. 
\item[$a2)$] Moreover,  for any $T\subset Z$,  $1\leq |T|\leq t$ there is $\bar s \in \mathcal{V}(H_S)$ such that $\rho_t(\pi_t^{-1}\{\bar s\}) = T$.
\item[$b1)$] For any $1\leq j\leq t-1$  and for any $u \in
{\mathcal{V}}(H_{\mathcal{S},X_t,\dots,X_{j+1},X_{j}})$
we have that $|\pi_{j}^{-1}(\{ u \})| \leq j$. 
\item[$b2)$] Moreover, for any $\bar s   \in$
${\mathcal{V}}(H_{\mathcal{S}})$ we have that $|\pi_{t}^{-1}(\{ \bar s\})|$
$\leq t$.
\end{itemize}

Let ${\bf s}=({\bar s}_1,\dots,{\bar s}_r)$ be a correctable syndrome corresponding
to an error $e$ of weight $\mu \leq t$. Let $Q$ be a point in $\mathcal{V}(H)$ corresponding to ${\bf s}$. 
We have 
$$
Q=(\bar{s}_1,\ldots,\bar{s}_r,\bar{A}_t,\dots, \bar{A}_1,  \bar{e}_1,\ldots,\bar{e}_t).
$$
We note that for any permutation $\sigma\in S_t$, there is $\widetilde Q\in \mathcal{V}(H)$,
\begin{eqnarray}\label{pperm}
  \widetilde Q=(\tilde{s}_1,\ldots,\tilde{s}_r,\bar{A}_{\sigma(t)},\dots, \bar{A}_{\sigma(1)},  \bar{e}_{\sigma(1)},\ldots,\bar{e}_{\sigma(t)}) .  
\end{eqnarray}
So (\ref{pperm}) gives immediately $a0)$.

We want to prove $a1)$ and $a2)$. Let $1\leq j\leq t-1$ and let $T \subset Z$,  $1\leq |T|\leq j$. Let $k=|T|$. There are two cases to consider: either $P_0\in T$ or $P_0\not\in T$.
\begin{itemize}
    \item[-] $P_0\in T$. Let $Q\in \mathcal{V}(H)$ corresponding to an error with weight\\ $\mu=t-j+k-1$. Thanks to (\ref{pperm}) we can assume that 
$$
Q=(\bar{s}_1,\ldots,\bar{s}_r,\bar{A}_t,\dots,\bar{A}_{j+1}, \bar{A}_{j},\dots, \bar{A}_1,  \bar{e}_1,\ldots,\bar{e}_t)
$$ 
where $\{\bar{A}_t,\dots,\bar{A}_{j+1}\}$ are $t-j$ elements in $Z$ that are different from $P_0$, $\{\bar{A}_j,\dots,\bar{A}_{1}\}$ are $(j-k+1)$  $P_0$'s and $(k-1)$ is the number of the elements of $T$ different from $P_0$. Let $u=(\bar{s}_1,\ldots,\bar{s}_r,\bar{A}_t,\dots,\bar{A}_{j+1})$.
At this point, we will obviously have
$\rho_j(\pi_j^{-1}(u))=T.$
 \item[-] $P_0\not\in T$. Let $Q\in \mathcal{V}(H)$ corresponding to an error with weight $\mu=t-j+k-1$. Similarly to the previous case, thanks to (\ref{pperm}), we can assume that $Q=(\bar{s}_1,\ldots,\bar{s}_r,\bar{A}_t,\dots,\bar{A}_{j+1}, \bar{A}_{j},\dots, \bar{A}_1,  \bar{e}_1,\ldots,\bar{e}_t)$, where $\{\bar{A}_t,\dots,\bar{A}_{j+1}\}$ are  $(t-j)$ elements of $\mathcal{V}(I)=Z\setminus\{P_0\}$, $\{\bar{A}_j,\dots,\bar{A}_{1}\}$ contains $(j-k)$ points equal to $P_0$ and $k$ points forming $T$. Let $u=(\bar{s}_1,\ldots,\bar{s}_r,\bar{A}_t,\dots,\bar{A}_{j+1})$, then we have
$\rho_j(\pi_j^{-1}(u))=T.$
\end{itemize}
The proof of $a2)$ is similar and is omitted.

\noindent To prove $b1)$ and $b2)$ it is enough to observe that if $t-j$ locations (including possibly the ghost point) are fixed, then at most $j$ distinct locations can exist for that error.
\item[$ii)$] Since $H$ is strongly multi-stratified, $H$ is weakly stratified (for Proposition~\ref{smsw}), and so we can apply Proposition~\ref{GBstructure}. As weak locators, we take $\mathcal{P}_{i}=g_{t,\zeta(t,i),1}^{(i)}$, where $\zeta(t,i)=\eta(t,i)\leq t$ and
${\bf T}(\mathcal{P}_{i})={x}_{i}^{t_i}$. 
In fact, the number of possible extensions is bounded by both $t_i$ and $|\{\hat{\pi}_i(P)\mid P\in \mathcal{V}(I)\cup P_0\}|$.
The first condition of Def. \ref{locDebole} is satisfied.\\
In order to prove the second condition we note that 
$\mathcal{P}_{i}({\bf s},{\bar x}_1,\dots,{\bar x}_{i-1},x_i)$ has among its solutions the ${\bar x}_i$'s such that $({\bar x}_1,\dots,{\bar x}_{i})$ are
the first $i$ components of an error location corresponding to ${\bf s}$ (or  $P_{0,i}$ value).

\end{itemize}
\end{proof}

We can summarize our findings so far.

Using weak locators does not work because ${\mathcal{P}}_{i}(S,x_1,\ldots,x_i)$ depends also on $i-1$ $x$-variables. Thus, the point $(S,x_1,\ldots,x_{i-1})\in \mathcal{V}(I)$ has the right multiplicity if and only if $t_i=1$. If this fail, it is very likely to have parasite solutions.

On the other hand, if we use the general error evaluator polynomial $\mathcal{E}$, we can proceed in two ways (see Example \ref{EXHerm1}), but both require an additional choice to discover parasite solutions. With non-trivial codes, this choice is very computationally expensive.

The strategy  we propose here is to force point $(S,x_1,\ldots,x_{i-1})\in \mathcal{V}(I)$ to have the right multiplicity.
See Definition \ref{locDebole} for the $t_i$'s.

\begin{definition}\label{zeroaf}
Let $C=C^{\perp}(I,L)$ be an affine-variety code.\\ 
Let $P_0=(\bar x_{0,1},\ldots, \bar x_{0,m})\in (\mathbb{F}_q)^m\setminus \mathcal{V}(I)$ be a ghost point. For any $1\leq
i\leq m$, let ${\mathcal{L}}_{i}$ be a polynomial in
$\FF_q[S,x_{1},\ldots, x_i]$, where $S=\{s_1,\dots,s_r\}$. 
Then
$\{{\mathcal{ L}}_{i}\}_{1 \leq i \leq m}$ is a set of  {\bf
multi-dimensional general error locator polynomials} for $C$ if for any $i$ 
\begin{itemize}
\item
${\mathcal{L}}_{i}(S,x_1,\ldots,x_i)= x_i^{t_i}+a_{t_i-1}x_i^{t_i-1}+\ldots +a_0$, $a_j \in \FF_q[S,x_1,\ldots,x_{i-1}]$ for  $0 \leq j \leq t_i-1$. In other words,
  ${\mathcal{L}}_{i}$ is a monic polynomial with degree $t_i$ with respect
  to the variable $x_i$ and its coefficients are in $\FF_q[S,x_1,\ldots,x_{i-1}]$.
\item Given a syndrome
  {\small{ ${\bf \bar s}=(\bar{s}_1,\dots \bar{s}_r)\in (\FF_{q})^{r}$}},
  corresponding to an error vector of weight
  $\mu\leq t$ and $\mu$ error locations
  $(\bar{x}_{1,1},\ldots,\bar{x}_{1,m})$ $,\ldots,
   (\bar{x}_{\mu,1},\ldots,\bar{x}_{\mu,m})$.
  If we evaluate the $S$ variables at ${\bf \bar s}$ and the variables $(x_1,\ldots,x_{i-1})$ at the truncated locations $\mathbf{\bar x}^j=(\bar{x}_{j,1},\ldots,\bar{x}_{j,i-1})$ for any $0\leq j\leq \mu$,  then the roots of
  ${\mathcal{ L}}_{i}({\bf \bar s},{\bf \bar x}^j,x_i)$ are
 $\{ \bar{x}_{h,i}\mid \overline{\mathbf{x}}^h=\overline{\mathbf{x}}^j, \, 1\le h\le \mu\}$ when $\mu=t$, and  $\{ \bar{x}_{h,i}\mid \overline{\mathbf{x}}^h=\overline{\mathbf{x}}^j, \, 0\le h\le \mu\}$ when $\mu\leq t-1$. That is, the polynomial  ${\mathcal{ L}}_{i}({\bf \bar s},\mathbf{\bar x}^j,x_i)$ does not have parasite solutions. 
\label{loc2af}
\end {itemize}
\end{definition}
\noindent Note that  the number of distinct first
components of error locations could be lower than $\mu$ and $t_i$.

To show how multi-dimensional general error locator polynomials can be applied, we redo the example on p. 155 of \cite{CGC-cd-art-lax}.
We postpone for the moment the problem of the existence of these polynomials and  of the method to compute them.
\begin{example}\label{EXHerm2}
 Let us consider  the Hermitian code
$C=C^{\perp}(I,L)$ from the curve $y^2+y=x^3$ over $\FF_4$ and
with defining monomials $\{1,x,y,x^2,x y\}$, as in the Example
\ref{exhq2}. Let us consider the lex term-ordering with
$s_1<s_2<s_3<s_4<s_5<x_2<y_2<x_1<y_1<e_2<e_1$ in
$\FF_4[s_1,s_2,s_3,s_4,s_5,x_2,y_2,x_1,y_1,e_1,e_2]$.\\ 

\noindent We consider the ideal $J_{*}^{C,t}$. In this ideal we are lucky enough to find the two
multi-dimensional general error locator polynomials that are {\small{$
\mathcal{L}_{2,1}(s_1,\ldots,s_5,x_2)$}}  and {\small{$\mathcal{L}_{2,2}(s_1,\ldots,s_5,x_2,y_2)$}},
which are respectively the polynomials $\mathcal{L}_{x}$ and $\mathcal{L}_{xy}$ of degree two in $x_2$ and $y_2$. In this case $t_1=t_2=t=2$ ($a_x,b_x,a_y,b_y,c_y$ are in the Appendix).
$$
\mathcal{L}_{x}= \mathbf{x}^2+\mathbf{x}\,a_x+b_x\textrm{ and }
\mathcal{L}_{xy}=\mathbf{y}^2+\mathbf{y}\,a_y+\mathbf{x}\,b_y+c_y
$$
Also in this example, we consider the three cases of Example \ref{exhq2}.
\begin{itemize}
    \item[-] We suppose that two errors occurred in the points $P_6=(\alpha,\alpha+1)$ and $P_7=(\alpha+1,\alpha)$, so the syndrome vector
corresponding to $(0,0,0,0,0, 1, 1, 0)$ is  ${\mathbf s}=(0,1,1,1,0)$.
In order to find the error positions we evaluate 
$\mathcal{L}_{x}$ in $\overline{{\mathbf s}}$ and we obtain the correct values of $x$, in fact $\mathcal{L}_{x}(\overline{{\mathbf s}}, x)= x^2+x+1=(x-\alpha)(x-(\alpha+1))$.
Now we have to  evaluate  $\mathcal{L}_{y}$ in $(\overline{{\mathbf s}},\alpha)$ and in $(\overline{{\mathbf s}},\alpha+1)$. Also in this case we obtain the correct solutions (with the highest possible multiplicity)
$$\begin{array}{l}
\mathcal{L}_{xy}(\overline{{\mathbf s}},\alpha,y)= y^2+\alpha=(y-(\alpha+1))^2\\
\mathcal{L}_{xy}(\overline{{\mathbf s}},\alpha+1,y)= y^2+\alpha+1=(y-\alpha)^2.
\end{array}$$

\item[-] We consider the syndrome $(\alpha+1,0,\alpha,0,0)$,
corresponding to the error vector 
$(1,\alpha,0,0,0,0, 0, 0)$, we obtain
$$
\mathcal{L}_{x}(\overline{{\mathbf s}},x)=x^2\textrm{ and }
\mathcal{L}_{xy}(\overline{{\mathbf s}},0,y)=y^2+y=y(y-1).
$$
The solutions of the above system are $(0,0), (0,1)$. Also in this case  the solutions of the equation $\mathcal{L}_{x}(\overline{{\mathbf s}})=0$ are correct.
\item[-] Again, when there is only one error of value $\alpha+1$ in the third point, we have the correct answers, in fact
$$
\begin{array}{l}
 \mathcal{L}_{x}(\alpha+1,\alpha+1,1,\alpha+1,1,x)=x^2+1=(x+1)^2\\
\mathcal{L}_{xy}(\alpha+1,\alpha+1,1,\alpha+1,1,1,y)=y^2+(\alpha+1)y+\alpha=(y-1)(y-\alpha).   
\end{array}
$$
so the solutions are $(1,1)$, which is the ghost point, and $(1,\alpha)$ i.e. the coordinates of the right location.
\end{itemize}
\noindent 
The main difference between $\mathcal{L}_{x}$,$\mathcal{L}_{xy}$ of Example \ref{exhq2} and $\mathcal{L}_{x}$, $\mathcal{L}_{xy}$ of
this example is that now we do not have spurious solutions, that is,
now the roots of our locators are exactly the error locations and no
more ambiguity exists.
\end{example}

\vspace{0.5cm}

As evident from the previous example, multidimensional general error locator polynomials are very convenient for decoding.
However, to prove their existence we cannot use the theoretical methods developed so far, because these methods do not deal
with multiplicities. In the next section we will develop more advanced theoretical methods, that will permit to construct ideals where these
polynomials lie and can be easily spotted.


\section{Stuffed ideals}
\label{hasse}
Let $G$ be a reduced \Gr \ basis of a radical weakly stratified
ideal $J$ as in Proposition \ref{GBstructure}. From now on we consider the ordering as in Proposition~\ref{GBstructure}.
In this section we fix $1\leq i\leq m$ and $1\leq j\leq L$ and we consider the projection
$$\pi:\mathcal{ V}(J_{\mathcal{S},\mathcal{A}_{L},\dots,\mathcal{A}_{j+1},
{\sf{a}}_{j,1},\dots,{\sf{a}}_{j,i-1},{\sf{a}}_{j,i}})\rightarrow\mathcal{ V}(J_{\mathcal{S},\mathcal{A}_{L},\dots,\mathcal{A}_{j+1},{\sf{a}}_{j,1},\dots,
{\sf{a}}_{j,i-1}})$$

We consider the variable ${\sf{a}}_{j,i}$ in block $A_j$.

Let $\mathcal{R}=\mathbb{K}[\mathcal{S},\mathcal{A}_{L},\dots,\mathcal{A}_{j+1},
{\sf{a}}_{j,1},\dots,{\sf{a}}_{j,i-1}]$. Let $g$ be a polynomial in $G^{\, j,i}$ such that the degree in ${\sf{a}}_{j,i}$ of $g$ is $\Delta=\zeta(j,i)=\eta(j,i)$. By Proposition \ref{GBstructure}, we know that this polynomial exists and it can be assumed to be monic in $\mathcal{R}[{\sf{a}}_{j,i}]$.
Let $P_h\in \Sigma_h^{j,i}$ where $1\leq h\le\Delta-1$, then
$$g(P_h,{\sf{a}}_{j,i})={\sf{a}}_{j,i}^{\Delta}+\alpha_{\Delta-1} {\sf{a}}_{j,i}^{\Delta-1}+\ldots+\alpha_0 \in \mathbb{K}[{\sf{a}}_{j,i}]\textrm{ where }\alpha_i\in \mathbb{K}.$$
We are interested in solutions of the equation
\begin{equation}\label{g0}
    g(P_h,{\sf{a}}_{j,i})=0.
\end{equation}
Since  {\small{$P_h\in \Sigma_h^{j,i}$}}, there exist distinct {\small{$Q_1,\ldots,Q_h$}}  such that {\small{$\pi^{-1}(P_h)=\{Q_1,\ldots,Q_h\}$}}, with $Q_l=(P_h,\lambda_l)$ for any $1\le l  \le h$. So $\lambda_1,\ldots,\lambda_h$ are some solutions of (\ref{g0}). But there exist other $\Delta-h$ solutions (counting multiplicities) of (\ref{g0}), say $\lambda_{h+1},\ldots,\lambda_{\Delta}$. There are two cases:
\begin{enumerate}
    \item[(a)] It may be that $\lambda_{h+l}\not\in\{\lambda_1,\ldots,\lambda_h\}$ for some $l$. In this case, point $(P_h,\lambda_{h+l})$ is not an extension of $P_h$, because  {\small{$(P_h,\lambda_{h+l})\not\in \mathcal{ V}(J_{\mathcal{S},\mathcal{A}_{L},\dots,\mathcal{A}_{j+1},
{\sf{a}}_{j,1},\dots,{\sf{a}}_{j,i-1},{\sf{a}}_{j,i}})$}}, and so $\lambda_{h+l}$ is a parasite solution.
\item[(b)] But it may  also be that  $\{\lambda_{h+1},\ldots,\lambda_{\Delta}\}\subset\{\lambda_1,\ldots,\lambda_h\}$, depending on the multiplicities of the $\{\lambda_1,\ldots,\lambda_h\}$. In this case, if we solve (\ref{g0}), we have exactly the extensions and we are not confused by parasite solutions.
\end{enumerate} 
We want to change slightly our variety in order to force case (b).
To do that, we need that the sum of multiplicities of $\{\lambda_{l}\}_{1\leq l\leq h}$ is equal to $\Delta$. To increase the multiplicity of any $\lambda_l$, we can use the Hasse derivative (see Sec. \ref{hassesub} and in particular Theorem \ref{teomult}). 

\begin{definition}\label{stuffed}
Let $K\subset\mathcal{R}[{\sf{a}}_{j,i}]$ be a zero-dimensional ideal such that $\mathcal{V}(K)\subset \mathbb{A}_{j,i}$. Let $\Delta=\eta(j,i)$. Let $G=\mathrm{GB}(K)$ and $g=g_{\Delta}^{(i)}$.  We say that $K$ is \textbf{stuffed} if for any $1\leq h\leq \Delta-1$ and  for any  $P_h\in\Sigma_h^{j,i}$, the equation (\ref{g0}) has $h$ distinct solutions in $\mathbb{K}$.
\end{definition}


\begin{definition}\label{deftheta}
Let {\footnotesize{$H\subset\mathbb{K}[V_1,\ldots,V_{\mathcal{N}}]$}} be a zero-dimensional ideal. Let {\footnotesize{$n\geq 1$}}. Let $f\in H$ and $Q\in \mathcal{V}(H)$ where $Q=(P,\overline{V}_{\mathcal{N}})$. Let
$\vartheta_1: H\longrightarrow \mathbb{K}$ such that 
$$\vartheta_1(f)=\varphi^{(1)}(f(P,V_{\mathcal{N}}))_{\big{|_{V_{\mathcal{N}}=
\overline{V}_{\mathcal{N}}}}}$$
and let $H^{[Q,1]}=\ker\vartheta_1$. We define inductively
$\vartheta_n: H^{[Q,n-1]}\longrightarrow \mathbb{K}$ such that 
$$\vartheta_n(f)=\varphi^{(n)}(f(P,V_{\mathcal{N}}))_{\big{|_{V_{\mathcal{N}}=
\overline{V}_{\mathcal{N}}}}}$$
and we write $H^{[Q,n]}=\ker\vartheta_n$.
\end{definition}


\noindent We note that $H^{[Q,1]}$ is an ideal. In fact, if $f\in H^{[Q,1]}$, $g\in \mathbb{K}[V_{\mathcal{N}}]$ and $\overline{V}_{\mathcal{N}}\in \mathcal{V}(f)$ then we claim that
$$
\varphi^{(1)}(fg)(\overline{V}_{\mathcal{N}})=
\varphi^{(1)}(f)(\overline{V}_{\mathcal{N}})
g(\overline{V}_{\mathcal{N}})+
\varphi^{(1)}(g)(\overline{V}_{\mathcal{N}})f(\overline{V}_{\mathcal{N}})=0.
$$

\noindent In fact, $\varphi^{(1)}(f)(\overline{V}_{\mathcal{N}})=0$, since $f\in \ker\vartheta_1$ and  $f(\overline{V}_{\mathcal{N}})=0$, since $\overline{V}_{\mathcal{N}}\in \mathcal{V}(f)$.
Inductively, we can similarly prove that $H^{[Q,n]}$ is an ideal.

Let us consider a zero-dimensional ideal $K\subset\mathcal{R}[{\sf{a}}_{j,i}]$. It is convenient to call our variables also as $\{V_1,\ldots,V_{\mathcal{N}}\}=\mathcal{\mathcal{S}}\cup\mathcal{A}_L,
\cup\dots,\mathcal{A}_{j+1}\cup\{{\sf a}_{j,1},\dots,{\sf a}_{j,i}\}$, in such a way that $V_1<\ldots<V_{\mathcal{N}}$ and $V_{\mathcal{N}}={\sf a}_{j,i}$.\\
We suppose that {\small{$G=\mathrm{GB}(K)$}} satisfies (\ref{4a}), that is 
$$
{\bf T}(G)\approxeq\{\tau,\ldots,\tau,\tau V_{\mathcal{N}},\ldots,\tau V_{\mathcal{N}},\ldots,\tau V_{\mathcal{N}}^{\Delta-1},\ldots,\tau V_{\mathcal{N}}^{\Delta-1},V_{\mathcal{N}}^{\Delta}\}
$$
and $\tau$ is any elements in $V_1<\ldots<V_{\mathcal{N}-1}$. In particular there is a polynomial {\small{$g\in G^{\, j,i}$}} s.t. {\small{$\textbf{T}(g)={{\sf a}_{j,i}}^{\Delta}={V_{\mathcal{N}}}^{\Delta}$}}.\\
For each $1\leq h\leq \Delta-1$ we perform the following operations:
\begin{itemize}
    \item[(a)] If for any $P_h\in \Sigma_h^{j,i}$  equation (\ref{g0}) has $h$ distinct solutions in $\mathbb{K}$, we do nothing. Otherwise, we take any $P_h\in \Sigma_h^{j,i}$ such that (\ref{g0}) has more than $h$ solutions.
    \item[(b)] We consider $Q=(P_h,\overline{V}_{\mathcal{N}})$ which is any extension of $P_h$. We want to compute  $H^{[Q,\Delta-h]}$. In order to do that, we iteratively compute $\ker\vartheta_{n}$ (see Definition \ref{deftheta}) from $n=1$ to $n=\Delta-h$.
    \item[(c)] For any such $n$, we apply  Theorem \ref{teoBM} to $H=H^{[Q,n-1]}$ and $H'=H^{[Q,n]}$, so that $H'=\ker \vartheta_n$. The hypotheses of Theorem \ref{teoBM}  are trivially satisfied, because $\vartheta_{n}$ is $\mathbb{K}-$linear and $\ker \vartheta_n$ is an ideal. In the subsequent step (d), we get ready to apply  Theorem \ref{teoBM}.
    \item[(d)] We consider the point {\small{$Q=(\overline{V}_1,\ldots,\overline{V}_{\mathcal{N}})=(P,\overline{{\sf{a}}}_{j,i})$}} with {\small{$P=(\overline{V}_1,\ldots,\overline{V}_{\mathcal{N}-1})$}} $
=(\overline{s}_1,\ldots,\overline{s}_N,\overline{{\sf{a}}}_{L,1},\ldots,
\overline{{\sf{a}}}_{j,i-1})$, $Q$ is a solution of ideal $H$.
To apply  Theorem \ref{teoBM}, we consider the smallest  polynomial $g^*\in W$, with $W=\mathrm{GB}(H)$, such that
$$
\vartheta_{n}(g^*)\ne 0\textrm{, that is, }\varphi^{(n)}(g^*((P,V_{\mathcal{N}}))_{\big{|_{V_{\mathcal{N}}=
\overline{V}_{\mathcal{N}}}}}\neq 0.
$$
\noindent We compute $W'$ from $W$.
To apply Theorem \ref{teoBM}, we need to identify $\beta_k$'s such that
$$
\vartheta_{n}((V_k-\beta_k)g^*)=0\textrm{ where }1\leq k\leq \mathcal{N},
$$
where we consider the weaker form $(iii)$ in Remark \ref{wow}.\\
We solve previous equation as follows 
$$
\begin{array}{l}
\vartheta_{n}((V_k-\beta_k)g^*)=(\overline{V}_k-\beta_k)\vartheta_{n}(g^*)=0\quad(1\leq k\leq \mathcal{N}-1)\\
\implies \beta_k\textrm{ is the }k\textrm{-th component of }Q.\\
\end{array}
$$
$$
\begin{array}{l}
\vartheta_{n}((V_{\mathcal{N}}-\beta_{\mathcal{N}})g^*)=
g^*(Q)+\overline{V}_{\mathcal{N}}\vartheta_{n}(g^*)-\beta_{\mathcal{N}}
\vartheta_{n}(g^*)=0\\
\implies \beta_{\mathcal{N}}=\frac{g^*(Q)+\overline{V}_{\mathcal{N}}\vartheta_{n}(g^*)}
{\vartheta_{n}(g^*)}=\frac{g^*(Q)}{\vartheta_{n}(g^*)}+
\overline{V}_{\mathcal{N}}.\\  
\end{array}
$$

\begin{lemma}\label{rem2} We claim that
$$
g^*\in G_{r}^{\, j,i}\textrm{, i.e. }\textbf{T}(g^*)\approxeq\tau {\sf{a}}_{j,i}^{r}\textrm{ where }n-1\leq r \leq \Delta-1.
$$

\begin{proof}
Recall that we use $\approxeq$ to express a unconventional identification (see Remark \ref{remB}).
If $\textbf{T}(g^*)\approxeq\tau {\sf{a}}_{j,i}^{r}$ with $r<n-1$, then $\vartheta_{n}(g^*)=0$.
So  $\textbf{T}(g^*)\approxeq\tau {\sf{a}}_{j,i}^{r}$ with $r\geq n-1$. However, $r\ne\Delta$, otherwise we have already finished. So $n-1\leq r \leq \Delta-1$.
\end{proof}
\end{lemma}

    \item[(e)]
We build  $W'$ using Theorem \ref{teoBM}. By Lemma \ref{rem2} we have that $$\textbf{T}(g^*)\approxeq\tau {\sf{a}}^{r}_{j,i} \quad n-1\leq r \leq \Delta-1.$$ 
So {\scriptsize{$W'=W_1\cup W_2\cup W_3$}} where
\begin{itemize}
    \item[-]{\scriptsize{$W_1=\{g\mid g<g^*\}$}}. \\
 So {\scriptsize{${\bf T}(W_1)\approxeq\{\tau,\ldots,\tau{\sf{a}}_{j,i},\ldots,
\tau {\sf{a}}^{r-1}_{j,i}\}$}}, or possibly also {\scriptsize{$\tau {\sf{a}}_{j,i}^{r}\in{\bf T}(W_1)$}}.
 \item[-] {\scriptsize{$W_2=W_{2,1}\cup W_{2,2}$}} where
{\scriptsize{$W_{2,1}=\{g^*(s_1-\overline{s}_1),\ldots,g^*({\sf{a}}_{j,i-1}-
\overline{{\sf{a}}}_{j,i-1})$}},\\ so {\scriptsize{$ {\bf T}(W_{2,1})=\{\tau  {\sf{a}}^{r}_{j,i},\ldots,\tau  {\sf{a}}^{r}_{j,i}\}$}}.\\ 
{\scriptsize{$W_{2,2}=\{g^*({\sf{a}}_{j,i}-\frac{g^*}{\vartheta_{r}(g^*)}-
\overline{{\sf{a}}}_{j,i})\}$}}. Then {\scriptsize{$ {\bf T}(W_{2,2})=\{\tau {\sf{a}}_{j,i}^{r+1}\}$}}.
\item[-]{\scriptsize{$W_3=\{g-\frac{\vartheta_{r}(g)}{\vartheta_{r}(g^*)}g^*\mid g>g^*\}$}} and hence the leading terms of $W_3$ are those in ${\bf T}(W)$, except  for {\scriptsize{$\tau,\ldots,\tau {\sf{a}}_{j,i}^{r-1}$}} and possibly {\scriptsize{$\tau {\sf{a}}_{j,i}^{r}$}}.
\end{itemize}
Therefore, the structure of $W'$ is the same as that of $W$, except possibly if $r+1=\Delta$. In that case
\begin{itemize}
    \item[-] {\scriptsize{${\bf T}(W_1)\approxeq\{\tau,\ldots,\tau{\sf{a}}_{j,i},\ldots,\tau {\sf{a}}_{j,i},\tau{\sf{a}}_{j,i}^{\Delta-2},\ldots,\tau {\sf{a}}^{\Delta-2}_{j,i}\}$}}, or possibly also {\scriptsize{$\tau {\sf{a}}_{j,i}^{\Delta-1}\in{\bf T}(W_1)$}},
 \item[-] {\scriptsize{$\textbf{T}(W_2)=\textbf{T}(W_{2,1})\cup \textbf{T}(W_{2,2})$}} where
{\scriptsize{$  {\bf T}(W_{2,1})=\{\tau {\sf{a}}^{\Delta-1} ,\ldots,\tau {\sf{a}}^{\Delta-1}\}$}}\\ and 
{\scriptsize{ $ {\bf T}(W_{2,2})=\{\tau {\sf{a}}_{j,i}^{\Delta}\}$}}.
\item[-] {\scriptsize{$\textbf{T}(W_3)=\emptyset$}}.
\end{itemize}
In the \Gr\ basis $W$ in  (\ref{4a}) there exists a polynomial $\overline{g}$ such that ${\bf T}(\overline{g})={\sf{a}}_{j,i}^{\Delta}$.
So ${\bf T}(\overline{g})|\tau {\sf{a}}_{j,i}^{\Delta}$ and we can remove the new term. Thus ${\bf T}(W')$ does not change and it remains as in (\ref{4a}).

  \item[(f)] Once all the  above operations have been concluded,
 for any $1\leq h\leq \Delta-1$ and for any $P_h\in \Sigma_h^{j,i}$, (\ref{g0}) will have exactly $h$ distinct solutions and  the resulting ideal will be stuffed.

\end{itemize}

\noindent We have thus proved the following theorem:

\begin{theorem} \label{TeoHasse}
Let $K\subset\mathcal{R}[{\sf{a}}_{j,i}]$ be a zero-dimensional ideal such that $G=\mathrm{GB}(K)$ verifies (\ref{4a}). Let $g$ be the polynomial in $G$ such that $\textbf{T}(g)={\sf{a}}_{j,i}^{\Delta}$, with $\Delta=\eta(j,i)$. We can obtain an ideal $\widetilde{K}\subset\mathcal{R}[{\sf{a}}_{j,i}]$ such that
\begin{itemize}
    \item[1.]  $\widetilde{K}$ is stuffed.
\item[2.]  $\mathrm{GB}(\widetilde{K})$ verifies (\ref{4a}).
\item[3.] $\mathcal{V}(\widetilde{K})=\mathcal{V}(K)$.
\end{itemize} 
\end{theorem}
Although, in Theorem \ref{TeoHasse} we obtain $\widetilde{K}$ as in the procedure above, there are other ways to obtain $\widetilde{K}$, for example by simultaneously increasing the multiplicity of more $\lambda_h$'s.

Note that, generally speaking, $\widetilde{K}$ will lose the radicality, but its \Gr\ basis will retain (\ref{4a}), which is what we need.

\begin{theorem} \label{TeoHassesms}
If $K$ and $\widetilde{K}$ are as in Theorem \ref{TeoHasse}, then if $K$ is, respectively, strongly multi-stratified, multi-stratified and weakly stratified, then $\widetilde{K}$ is so.

\begin{proof}
The stuffing procedure does not change the number of pre-images at any level.
\end{proof}

\end{theorem}

Now, we are finally able to prove the existence of our multi-dimensional general error locator polynomials for any code.
Note that this is another constructive proof, since it tells us how to compute our polynomials, that is, simply by computing a suitable \Gr\ bases
of the corresponding stuffed ideal.

\begin{theorem} \label{bombastaffato}
Let $C=C^{\perp}(I,L)$ be an  affine--variety code with $d \geq
3$.  Let  $\widetilde{J}_{*}^{C,t}$ be a stuffed ideal of  $J_{*}^{C,t}$. Then
\begin{itemize}
\item[i)] $\widetilde{J}_{*}^{C,t}$ is a  strongly multi-stratified
ideal with respect to the $X$ variables.
\item[ii)] A \Gr\ basis of $\widetilde{J}_{*}^{C,t}$ contains a set of 
multi-dimensional general error locator polynomials for $C$.
\end{itemize} 
\end{theorem}

\begin{proof}
\begin{itemize}
\item[$i)$] We can use Th.\ref{TeoHassesms} and so {\small$\widetilde{J}_{*}^{C,t}$} is a  strongly multi-stratified ideal.
\item[$ii)$] As locators we can take for any $i$
$$
\mathcal{L}_{i}=g_{t,\zeta(t,i),1}^{(i)},
$$
where $\zeta(t,i)=\eta(t,i)$ and
${\bf T}(\mathcal{L}_{i})={x}_{i}^{t_i}, \quad t_i =\zeta(t,i)=\eta(t,i)$, thanks to Theorem \ref{TeoHasse}.
So the first condition of Definition \ref{zeroaf}
is satisfied.\\
 Let  $H=J_{*}^{C,t}$ and $\widetilde{H}=\widetilde{J}_{*}^{C,t}$. In order to prove the second condition we note that, 
since $\mathcal{L}_{i}$ is a polynomial of $\widetilde{H}_{S,x_{t,1},\dots,x_{t,i}}$ it will vanish at $({\bf s},{\bar x}_{1},\dots,{\bar x}_i)$, where
${\bf s}=({\bar s}_1,\dots,{\bar s}_r)$ and $({\bf s},{\bar x}_1,\dots,{\bar x}_i)$ can
be extended to a point in ${\mathcal{V}}(H)={\mathcal{V}}(\widetilde{H})$. 
Since $\widetilde{H}$ is stuffed,  $\mathcal{L}_{i}({\bf s},{\bar x}_1,\dots,{\bar x}_{i-1},x_i)$ has as solutions only the
${\bar x}_i$'s such that $({\bar x}_1,\dots,{\bar x}_{i})$ are
the first $i$ components of an error location corresponding to ${\bf s}$ (or $P_{0,i}$).

\end{itemize}
\end{proof}


\section{Families of affine--variety codes}
\label{families}

In this section we consider some families of affine-variety codes.

\subsection{SDG curves}
We discuss codes from some curves introduced in \cite{CGC-cd-art-salazar}.
\begin{definition}[\cite{CGC-cd-art-salazar}]
Let $\FF_s$ be a subfield of $\FF_q$. A polynomial $f$ in
$\FF_q[x]$ is called an $(\FF_q,\FF_s)$--polynomial if for each
$\gamma \in \FF_q$ we have $f(\gamma) \in \FF_s$.
\end{definition}
%
\begin{proposition}[\cite{CGC-cd-art-salazar}] \
\begin{itemize}
\item[1.] The polynomial $f(x)=b_3x^3+b_2x^2+b_1x+b_0 \in
\FF_4[x]$ is an $(\FF_4,\FF_2)$--polynomial if and only if $b_0,
b_3 \in \FF_2$ and $b_2=b_1^2$.
\item[2.] The polynomial $g(x)=b_7x^7+\dots+b_1x+b_0 \in
\FF_8[x]$ is an $(\FF_8,\FF_2)$--polynomial if and only if $b_0,
b_7 \in \FF_2$, $b_2=b_1^2$,  $b_4=b_2^2$, $b_6=b_3^2$ and
$b_3=b_5^2$.
\end{itemize}
\end{proposition}

Let {\small$\mathcal{F}=\{f(x)+g(y) \mid f,g {\mbox{ are }} (\FF_8,\FF_2){\mbox{-polynomials }}, \deg(f)=4, \deg(g)=6\}$}.\\ 
In \cite{CGC-cd-art-salazar} it is shown that the
family $\mathcal{F}$ has $784$ members and that each member of
this family has $32$ roots in $(\FF_8)^2$. Let us consider the
polynomial ${\mathcal{G}}=f(x)+g(y)$, with $f(x)=x^4+x^2+x$ and
$g(y)=y^6+y^5+y^3+1$, so that $\mathcal{G} \in \mathcal{F}$. Let
$I=\langle \mathcal{G}\rangle$ and  $J_{*}^{C,t}$ be the ideal associated
to the $C=C^{\perp}(I,L)$ code over $\FF_8$ that can correct up
to $t=1$ errors and with defining monomials $L=\{1,y,x,y^2\}$.\\
Ideal $J_{*}^{C,t}$ is generated by:
{\small{
$$
\{x_1^8-x_1,y_1^8-y_1,e_1^7-1,x_1^4+x_1^2+x_1+y_1^6+y_1^5+y_1^3+1,e_1-s_1,e_1y_1-s_2,e_1x_1-s_3,e_1y_1^2-s_4
\}
$$}}
and  the reduced \GR \ basis  $G$ with respect to the lex ordering with\\ $s_1<s_2<s_3<s_4<x_1<y_1<e_1$ is
{\small
\begin{align*}
&\{s_1^7+1, s_2^8+s_2,s_3^4+s_3^2s_1^2+s_3s_1^3+s_2^6s_1^5+s_2^5s_1^6+s_2^3s_1+s_1^4,
s_4+s_2^2s_1^6,\\
&{\bf{x_1}}+s_3s_1^6,{\bf{y_1}}+s_2s_1^6, e_1+s_1\}
\end{align*}
}
and then
$$
  \mathcal{L}_{2}={\bf y_1}+s_3s_1^6,   \qquad
  \mathcal{L}_{1}={\bf x_1}+s_2s_1^6.
$$

\subsection{SDG surfaces I}
We discuss codes from some surfaces introduced in \cite{CGC-cd-art-salazar}.

Let $\mathcal{F}=\{f(x)+g(y)+h(z) \mid f,g,h$ are
$(\FF_4,\FF_2){\mbox{-polynomials }}, \deg(f)=\deg(h)=3,
\deg(g)=2\}$. In \cite{CGC-cd-art-salazar} it is shown that the
family $\mathcal{F}$ has $96$ members and that each member of this
family has $32$ roots in $(\FF_4)^3$. Let us consider the
polynomial ${\mathcal{G}}=f(x)+g(y)+h(z)$, with $f(x)=x^3$,
$g(y)=y^2+y+1$ and $h(z)=z^3+1$, so that $\mathcal{G} \in
\mathcal{F}$.
Let  $I=\langle \mathcal{G}\rangle$ and
$J_{*}^{C,t}$ be  the ideal associated to the code
$C=C^{\perp}(I,L)$ over $\FF_4$  that can correct up
to $t=1$ error and with defining monomials $L=\{1,x,z,y\}$.\\
The ideal $J_{*}^{C,t} \subset \FF_4[s_1,s_2,s_3,s_4,x_1,y_1,z_1,e_1]$
is generated by $ \{x_1^4-x_1, y_1^4-y_1, z_1^4-z_1, e_1^3-1,
g+f+h, e_1-s_1,e_1z_1-s_3,e_1x_1-s_2, e_1y_1-s_4 \} $ and  the
reduced \GR \ basis $G$ with respect to  the lex ordering with $s_1<s_2<s_3<s_4<x_1<y_1<z_1<e_1$
is 
{\small
$$\{s_1^3+1, \: s_2^4+s_2,  \: s_3^4+s_3, \: s_4^2+s_4s_1+s_3^3s_1^2+s_2^3s_1^2,
\: {\bf y_1}+s_4s_1^2, \: {\bf x_1}+s_2s_1^2, \: {\bf z_1}+s_3s_1^2, \:
e_1+s_1\} \,,
$$
}
then
$$\mathcal{L}_{1}= {\bf x_1}+s_2s_1^2,   \quad
  \mathcal{L}_{2}= {\bf y_1}+s_4s_1^2,   \quad
  \mathcal{L}_{3}= {\bf z_1}+s_3s_1^2 \,.
$$

\subsection{SDG surfaces II}

We discuss codes from another family of surfaces introduced in \cite{CGC-cd-art-salazar}.

Let $\mathcal{F}=\{\beta x^2z+\beta^2xz^2+f(x)+g(y)+h(z)$ $\mid
\beta \neq 0, f,g,h {\mbox{ are }}
(\FF_4,\FF_2)$-polynomials, $\deg(f) \leq 2, \deg(h)
\leq 3, \deg(g)=2\}$. 
In \cite{CGC-cd-art-salazar} it is shown
that the family $\mathcal{F}$ has $576$ members and that each
member of this family has $32$ roots in $(\FF_4)^3$. Let us
consider the polynomial ${\mathcal{G}}=x^2z+xz^2+f(x)+g(y)+h(z)$,
with $\beta=1$, $f(x)=1$, $g(y)=y^2+y+1$ and $h(z)=z^3+1$, so that
$\mathcal{G} \in \mathcal{F}$.
Let  $I=\langle \mathcal{G}\rangle$
and $J_{*}^{C,t} $ be the ideal associated to the  code
$C=C^{\perp}(I,L)$ over $\FF_4$  that can correct one
error and with defining monomials $L=\{1,z, z^2, z^3, x, y\}$.\\
The ideal $J_{*}^{C,t} \subset
\FF_4[s_1,s_2,s_3,s_4,s_5,s_6,x_1,y_1,z_1,e_1]$ is generated by $
\{x_1^4-x_1,y_1^4-y_1,z_1^4-z_1,e_1^3-1, x_1^2z_1+x_1z_1^2+f+g+h,
e_1-s_1, e_1z_1-s_2,
e_1z_1^2-s_3,e_1z_1^3-s_4,e_1x_1-s_5,e_1y_1-s_6, \} $
and the reduced \GR \ basis $G$  with respect to  the lex ordering with
$s_1<s_2<s_3<s_4<s_5<s_6<x_1<y_1<z_1<e_1$ is
%
{\small
\begin{align*}
&  \{s_1^3+1,  s_2^4+s_2, s_3+s_2^2s_1^2, s_4+s_2^3s_1,s_5^4+s_5,
s_6^2+s_6s_1+s_5^2s_2s_1^2+s_5s_2^2s_1^2+s_2^3s_1^2+s_1^2, \\
& {\bf x_1}+s_5s_1^2, {\bf{y_1}}+s_6s_1^2,
\; {\bf{z_1}}+s_2s_1^2,\;
e_1+s_1\}
\end{align*}
}
and then
$$
  \mathcal{L}_{1}= {\bf x_1}+s_5s_1^2,    \quad
  \mathcal{L}_{2}= {\bf{y_1}}+s_6s_1^2,   \quad
  \mathcal{L}_{3}= {\bf{z_1}}+s_2s_1^2 \,.
$$

\subsection{Norm--trace curves}

We now give an example for codes coming from a family of curves 
(\cite{CGC-cd-art-Geil-1}), which are a natural generalization of Hermitian
curves.

Let $C=C^{\perp}(I,L)$ be the code from the norm--trace curve
$x^7=y^4+y^2+y$ over $\FF_8$ and with defining monomials
$\{1,x,x^2,y\}$.
This code (\cite{CGC-cd-art-Geil-1}) can correct $t=1$ error.
Let $J_{*}^{C,t}$ be the
ideal generated by:
\begin{equation*}
\{x_1^8-x_1,y_1^8-y_1,e_1^7-1,e_1-s_1,e_1x_1-s_2,e_1x_1^2-s_3,e_1y_1-s_4,x_1^7-y_1^4-y_1^2-y_1\}
\end{equation*}
and  the reduced \GR \ basis $G$  with respect to  the lex ordering with $s_1<s_2<s_3<s_4<x_1<y_1<e_1$  is {\small
$$\{s_1^7+1, \: s_2^8+s_2, \: s_3+s_2^2s_1^6, \:
s_4^4+s_4^2s_1^2+s_4s_1^3+s_2^7s_1^4, \: {\bf{x_1}}+s_2s_1^6, \:
{\bf{y_1}}+s_4s_1^6,
\: e_1+s_1\}.$$
}
Then
$$
  \mathcal{L}_{1}={\bf{x_1}}+s_2s_1^6,   \quad
  \mathcal{L}_{2}={\bf y_1}+s_4s_1^6 \;.
$$

Observe that in all our examples so far no stuffing was required, because we were considering the case $t=1$, which clearly cannot contain multiplicities.

\subsection{Hermitian curves}\label{HC}

Let $q$ be a power of a prime, then the Hermitian curve $\mathcal{H}$ over
$\FF_{q^2}$ is defined by the affine equation $\mathcal{G}: x^{q+1}=y^q+y$.
Each member of this family has $n=q^3$ points in  $\FF_{q^2}$ and
it is well-known that the function space is generated by monomials.

In Example \ref{exhq2} we considered the case $q=2$ and $t=2$,
we now consider the code $C$ corresponding to the case $q=3$ and $t=2$.
The defining monomials are $L=\{1,x,y,x^2,x y,y^2,x^3\}$. As before, we choose as ghost point $(1,1)$.

Our ideal $J_{*}^{C,2}$ is generated by
\vspace{-0.1cm}
{\scriptsize
\begin{align*}
&  \{ x_1^9-x_1, y_1^9-y_1, e_1^9-e_1,e_2^9-e_2, x_2^9-x_2, y_2^9-y_2,
y_1^3x_1-y_1^3+y_1x_1-y_1-x_1^5+x_1^4,\\
&y_2^3x_2-y_2^3+y_2x_2-y_2-x_2^5+x_2^4,
y_1^4-y_1^3+y_1^2-y_1-y_1x_1^4+x_1^4,y_2^4-y_2^3+y_2^2-y_2-y_2x_2^4+x_2^4,\\
&
e_1+e_2-s_1,
e_1 x_1+e_2 x_2-s_2,
e_1 y_1+e_2 y_2-s_3,
e_1 x_1^2+e_2 x_2^2-s_4,e_1 x_1 y_1+e_2 x_2 y_2-s_5,\\
&
e_1 y_1^2+e_2 y_2^2-s_6,
e_1 x_1^3+e_2 x_2^3-s_7,
e_1((x_1-1)^8-1)((y_1-1)^8-1), e_2((x_2-1)^8-1)((y_2-1)^8-1),\\
&
(e_1^8-1)(x_1-1),
(e_1^8-1)(y_1-1),(e_2^8-1)(x_2-1),
(e_2^8-1)(y_2-1),
e_1e_2((x_1-x_2)^8-1)((y_1-y_2)^8-1)
\}\,.
\end{align*}
}
We calculate the \Gr \ basis $G$ with respect to the usual lex ordering
with $s_1<\dots<s_7<x_2<y_2<x_1<y_1<e_2<e_1$.
The general error evaluator polynomial of $\mathcal{C}$ contains 134 monomials and it is reported in the Appendix.

The first weak locator ${\mathcal P}_{2}$ contains 172 monomials, while the second weak locators ${\mathcal P}_{1}$ contains 494 monomials (see Appendix for all polynomials). However,
these polynomials are by far not random.
Indeed, we can prove the following general structure result for
any $q\geq 2$ and $t=2$.

\begin{theorem}\label{evviva}
Let $p$ be any prime number and $m\in \NN$ such that $q=p^m\geq 2$. 
Let $C=C^{\perp}(I,L)$ be any Hermitian code with $t=2$ over $\FF_q$.
Then all sets of multi-dimensional general error locator polynomials
for $C$ are of the form
\begin{equation} \label{ovvio}
\begin{array}{c}
  \{ {\mathcal L}_{2}= {\mathcal L}_{x}= x^2+a x+b, 
  {\mathcal L}_{1}= {\mathcal L}_{xy}= y^2+c y+d \}  \\
  \{ {\mathcal L}_{2}= {\mathcal L}_{y}= y^2+A y+B, 
     {\mathcal L}_{1}= {\mathcal L}_{yx}= x^2+C x+D \} 
\end{array}
\end{equation}
with $a,b,A,B\in \FF_p[S]$, $c,d\in \FF_p[S,x]$ and $C,D\in \FF_p[S,y]$.\\
Moreover, 
\begin{eqnarray}
\label{formulinah} 
q\geq 2 &\implies &a s_2+b s_1=-s_4\,, \\
\label{formulinah2} 
q\geq 3 &\implies & A s_3+B s_1=-s_6\,.
\end{eqnarray}
\indent
Let $q\geq 2$ and $s_1=s_2=0$. We have $e_1=-e_2$, $x_1=x_2$, $b=x_1^2$,
 $a=2 x_1$.

Let $q\geq 3$ and $s_1=s_3=0$. We have $e_1=-e_2$, $y_1=y_2$, $B=y_1^2$,
 $A=2 y_1$.\\

All the results above hold also for any set of  weak
multi-dimensional general error locator polynomials
\begin{equation} \label{ovvio2}
\begin{array}{c}
  \{ {\mathcal P}_{2}= {\mathcal P}_{x}= x^2+a x+b, 
  {\mathcal P}_{1}= {\mathcal P}_{xy}= y^2+c y+d \}  \\
  \{ {\mathcal P}_{2}= {\mathcal P}_{y}= y^2+A y+B, 
     {\mathcal P}_{1}= {\mathcal P}_{yx}= x^2+C x+D \} 
\end{array}
\end{equation}
\end{theorem}
\begin{proof}
Let $H=J_{*}^{C,2}$ be the non-stuffed ideal for $C$ and $\widetilde H$ 
its stuffed ideal as in Theorem \ref{bombastaffato}.
There are two \Gr\ bases of $H$ and $\widetilde H$ that are relevant for us.
If the order has $S<x_2<y_2$ then we get $G_x$ for $H$ and $\widetilde G_x$
for $\widetilde H$. 
If the order has $S<y_2<x_2$ then we get $G_y$ for $H$ and $\widetilde G_y$
for $\widetilde H$.
As in Theorem \ref{bombastaffato}, $\widetilde G_x$ contains polynomials
${\sf p}_{x}\in \FF_q[S,x_2]$ and ${\sf p}_{x,y}\in \FF_q[S,x_2,y_2]$ such that,
once we replace $x_2$ with $x$ and $y_2$ with $y$, we get a set of locators
$ \{ {\mathcal L}_{2}= {\mathcal L}_{x}, {\mathcal L}_{1}= {\mathcal L}_{xy}
\}$.\\
The degree of ${\sf p}_{x}$ in $x_2$ is, {\em a priori}, $1$ or $2$.
However, since there are at least two points $\{P_1,P_2\}$
on the curve with two different
$x$, then  $\deg_{x_2}{\sf p}_{x}=2$, since ${\sf p}_x$ must have two 
distinct roots once evaluated
on a syndrome corresponding to a weight-$2$ error with $\{P_1,P_2\}$
as locations.\\
The degree of ${\sf p}_{x,y}$ in $y_2$ is, {\em a priori}, $1$ or $2$.
However, for any $\bar x\in \FF_q$ there are at least two points 
$\{ P_1=(\bar x,\bar y_1), P_2=(\bar x,\bar y_2) \}$
on the curve with $\bar y_1 \not= \bar y_2$.
Then  $\deg_{y_2}{\sf p}_{y}=2$, since it must have the two distinct roots 
$\{ \bar y_1,\bar y_2 \}$
once evaluated
on a syndrome corresponding to a weight-$2$ error with $\{P_1,P_2\}$
as locations.\\
The previous argument can be trivially adapted to show that
$\deg_{y_2}({\sf p}_{y})=2$ and $\deg_{x_2}({\sf p}_{y,x})=2$, where
${\sf p}_{y}\in \FF_q[S,y_2]$ and ${\sf p}_{y,x}\in \FF_q[S,y_2,x_2]$
come from $\widetilde G_y$, and so (\ref{ovvio}) is proved, except
for our claim that all the coefficients of these polynomials
actually lie in the base field $\FF_p$,
which follows from Remark \ref{basefield}.

To prove (\ref{formulinah}), we first claim that
\begin{equation}
\label{f2}
f\in H \implies f^2 \in \widetilde H.
\end{equation}
To see (\ref{f2}) we note that in the creation of $\widetilde H$ from $H$ we
only impose the vanishing of the first-order derivative at points
of ${\mathcal V}(H)$, but if we take any point $Q\in {\mathcal V}(H)$ 
we have (see Definition \ref{deftheta} for $\theta_1$)
$$
  \theta_1(f^2)= 2 f(Q)\theta_1(f) =  0 \theta_1(f)=0 \,.
$$
Since $s_1-e_1-e_2, s_2-e_1x_1-e_2x_2,s_4-e_1x_1^2-e_2x_2^2 \in H$,
we have that\\ $(s_1-e_1-e_2)^2,(s_2-e_1x_1-e_2x_2)^2,
(s_4-e_1x_1^2-e_2x_2^2)^2 \in \widetilde H$ for (\ref{f2}).
Passing from variables to values we observe that
\begin{eqnarray}\label{A}
    \bar{s}_1=\bar{e}_1+\bar{e}_2, \quad \bar{s}_2=\bar{e}_1\bar{x}_1+\bar{e}_2\bar{x}_2, \quad \bar{s}_4=\bar{e}_1\bar{x}_1^2+\bar{e}_2\bar{x}_2^2
\end{eqnarray}
and that
$$
  \bar{a}=a(\bar{S})=-(\bar{x}_1+\bar{x}_2), \quad  
  \bar{b}=\bar{x}_1\bar{x}_2 \,.
$$
So
$$
  -(\bar{x}_1+\bar{x}_2)
  (\bar{e}_1\bar{x}_1+\bar{e}_2\bar{x}_2)+
   \bar{x}_1\bar{x}_2(\bar{e}_1+\bar{e}_2)=
  -(\bar{e}_1\bar{x}_1^2+\bar{e}_2\bar{x}_2^2),\textrm{ which proves (\ref{formulinah}).}
$$
In the same way, we can calculate the set of locators
$ \{ {\mathcal L}_{2}= {\mathcal L}_{y}, {\mathcal L}_{1}= {\mathcal L}_{yx}\}$.
If $q\geq 3$, we have also  $s_1-e_1-e_2, s_3-e_1y_1-e_2y_2,s_6-e_1y_1^2-e_2y_2^2 \in H$,
so we have that $(s_1-e_1-e_2)^2,(s_3-e_1y_1-e_2y_2)^2,
(e_6-e_1y_1^2-e_2y_2^2)^2 \in \widetilde H$ for (\ref{f2}).
Again, we pass from variables to values, and we obtain
\begin{eqnarray}\label{B}
    \bar{s}_1=\bar{e}_1+\bar{e}_2, \quad \bar{s}_3=\bar{e}_1\bar{y}_1+\bar{e}_2\bar{y}_2, \quad \bar{s}_6=\bar{e}_1\bar{y}_1^2+\bar{e}_2\bar{y}_2^2
\end{eqnarray}
and that
$$
  \bar{A}=A(\bar{S})=-(\bar{x}_1+\bar{x}_2), \quad  
  \bar{B}=B(\bar{S})=\bar{x}_1\bar{x}_2 \,.
$$
So
$$
  -(\bar{x}_1+\bar{x}_2)
  (\bar{e}_1\bar{y}_1+\bar{e}_2\bar{y}_2)+
   \bar{y}_1\bar{y}_2(\bar{e}_1+\bar{e}_2)=
  -(\bar{e}_1\bar{y}_1^2+\bar{e}_2\bar{y}_2^2) \implies \bar{A}s_3+\bar{B}s_1=-s_6\,.
$$
The last part of theorem comes from direct computations, as follows.\\
From (\ref{A}), in the case $\bar{s}_1=\bar{s}_2=0$,  we note  $\bar{e}_1=-\bar{e}_2$, $\bar{x}_1=\bar{x}_2$. And so
\begin{itemize}
\item[1.] If $p=2$ then $\bar{a}=-(\bar{x}_1+\bar{x}_2)=2\bar{x}_1=0$ and $ \bar{b}=\bar{x}_1\bar{x}_2=\bar{x}_1^2$. 
\item[2.] If $p\not=2$ then $\bar{a}=-(\bar{x}_1+\bar{x}_2)=2\bar{x}_1\implies \bar{x}_1=\frac{\bar{a}}{2}$. 
\end{itemize}
From (\ref{B}), if $s_1=s_3=0 $ then $ e_1=-e_2$ and $y_1=y_2.$ And thus
\begin{itemize}
\item[1.] If $p=2$ then $\bar{A}=0$ e $\bar{B}=y_1^2$. 
\item[2.] If $p\not=2$ then $\bar{A}=2y_1\implies y_1=\frac{A}{2}$ and $\bar{B}=y_1^2$.
\end{itemize}
Since in the proof so far we have used the relations on the syndromes coming from the non-stuffed ideal $H$, everything that we proved up to now holds also for the weak locators.
\end{proof}

The locator $\mathcal{P}_2$ computed for the Hermitian code with $q=3$ and $t=2$ is indeed of the form $\mathcal P_{2}= \mathcal P_{x}= x^2+a x+b$, with $|a|=82$ and $|b|=91$, so, for example when $s_1\ne 0$, it is enough to evaluate $a(\bar S)$ and then we obtain $b(\bar S)$ as 
$$
b(\bar S) =-\frac{s_4+ a(\bar S) s_2}{s_1}.
$$ 
Also $\mathcal{P}_1$ is as above, that is, of the form $ \mathcal P_{1}=\mathcal P_{xy}= y^2+c y+d $.\\

Regrettably, we have not been able to compute explicitly $\mathcal L_{2}$ and $\mathcal L_{1}$ for $q=3$, due to the high computation cost of the stuffing procedure.


\section{Conclusions and open problems}
\label{conc}
Assuming we are able to compute the relevant \Gr\ basis, we have
identified a very easy decoding procedure for any affine-variety code:
we evaluate our polynomials $\{\mathcal L_{i}\}_{1\leq i\leq m}$ in the received syndromes and we use some simple
root--finding to get the error locations.
As it is traditional in coding theory, once we have the error locations
we can directly get the error values and hence the decoding problem
is completely solved.
This apparently idyllic situation is marred by two serious issues:
\begin{itemize}
\item the computation of the associated \Gr\ basis can be quite
      beyond present means already for medium-size codes;
\item even if we compute our locators, they could be so dense that
      their use would be impractical.
\end{itemize}
These two apparently different problems may have one common solution:
to identify our polynomials {\em without} computing any \Gr\ basis,
but using the ``structure of the code''.
This is indeed a desperate goal, if tried for general codes, but
we believe that some code families have locators which are easy to describe
explicitly and very sparse.
Our belief stems from our results in \cite{CGC-cd-art-teomaxmanu}
(and \cite{CGC-cd-prep-teomaxmanu}),
where we explicitly give locators for families of cyclic codes,
which apparently have no special structure,
simultaneously proving their sparsity (see also \cite{CGC-cd-art-YaotsChung10}, \cite{CGC2-cd-art-LeeChangChen10}, \cite{CGC2-cd-art-LeeChangJingChen10} for recent results on the structure of locators).\\
We then suggest the following problems.
\begin{problem}
For any $q$ and $t$ write formally $\{{\mathcal L}_{i}\}$
for the Hermitian code.
\end{problem}
\begin{problem}
For any $q$, $r$ and $t$ write formally $\{{\mathcal L}_{i}\}$
for the code from norm-trace curves.
\end{problem}
\begin{problem}
For any admissible parameter, write formally $\{{\mathcal L}_{i}\}$
for the codes from \cite{CGC-cd-art-salazar} curves.
\end{problem}
\begin{problem}
For any admissible parameter, write formally $\{{\mathcal L}_{i}\}$
for the codes from \cite{CGC-cd-art-salazar} surfaces I.
\end{problem}
\begin{problem}
For any admissible parameter, write formally $\{{\mathcal L}_{i}\}$
for the codes from \cite{CGC-cd-art-salazar} surfaces II.
\end{problem}
An interesting problem comes from the definition of the $t_i$'s 
in Definition~\ref{zeroaf}. Clearly, we have $t_i<t$ only if
any $t$ points on the variety have necessarily less than $t$ distinct
values for their $i$-th component.
For example, you might think of two parallel lines in the plane $(\FF_q)^2$, 
$x=a$ and $x=b$, any defining
a Reed Solomon code.
In this case, whatever $t\geq 2$ can be, we will always have $t_1=2$.
This example is very special, since the variety is reducible.
We then ask the following problem.
\begin{problem}
To identify (easy to check) conditions on a curve 
and on the function space such $t_i=t$ for any $i$.
\end{problem}
\noindent For special cases this is quite obvious. For example when $t=1$
this is always true.
It would be very nice to get a generalization of the former problem.
\begin{problem}
To identify conditions on a curve 
and on the function space such either $t_i=t$ for any $i$ or
to find a (projective?affine?) transformation of $(\FF_q)^m$
such that the same holds.
\end{problem}
$\,$\\
\indent In \cite{CGC-cd-art-gelp1} we studied also locators able to correct simultaneously errors and erasures (Definition \ref{zeronu}).
It is obvious how to extend Definition \ref{locDebole}, \ref{eval}, \ref{zeroaf} to cover also simultaneous error-and-erasure decoding.
A suitable ideal modified from $J_*^{C,t}$ will again be strongly multi-stratified and so Theorem \ref{evviva} can be extended accordingly, to cover the new case.
We do not give explicitly the related definitions and results, due to triviality of the extensions.\\

Our decoding works well with the case of one ghost point, which plays the same role  as that played by zero in the zero-free n-th root codes (being cyclic codes a special case).
However, there is no reason why we should  restrict to the use of \textit{one} ghost point, since a clever choice of multiple ghost points could give easier decoding. Indeed the geometric structure of the union of $\mathcal{V}(I)$ and the ghost points influences the shape of the \Gr\ bases of our ideals and hence the shape of our polynomials.

\begin{problem}
To identify a clever choice of (possibly) multiple ghost points, in order to  minimize  the corresponding locators.
\end{problem}
$\,$\\
\indent We have recently known of a promising new approach to the decoding of one-point geometric Goppa codes (\cite{CGC-cd-art-lax11}),
where  {\em generic} versions of locators are proposed. The nice idea behind it is that trying to correct all
correctable syndromes forces the locator polynomials to be dense, while it could be possible in some cases
to identify a very large subset of syndromes ({\em generic syndromes}) such that the locator for those is small.
\begin{problem}
To define rigorously sets of \textit{generic} multi-dimensional locators.
\end{problem}

\section*{Acknowledgements}
The first two authors would like to thank their supervisor: the third author.
For useful suggestions and discussions, the authors would like to thank:
P. Fitzpatrick, O. Geil, T. Mora, M. Piva and C. Traverso.

Suggestions and criticism by an anonymous referee helped us greatly
to improve the presentation.

This work has been partially presented at the Claude Shannon Institute
Workshop in Cork (Ireland), May 2007 and the Mini workshop on
error correcting codes and network coding in Aalborg (Denmark), September 2007. 

Some of these results are present in the second author's PHD thesis (\cite{CGC-cd-phdthesis-Emma})

We have run our computer simulations using the software package Singular
{\tt (http://www.singular.uni-kl.de)}.

\bibliography{RefsCGC}

\section*{Appendix}
\noindent In Example \ref{Degenere1} we have the following polynomial
{\scriptsize{
\begin{align*}
\mathcal{L}_{xy}=&\mathbf{y}^2+\mathbf{y}(2s_4^6+2s_4^5s_3+2s_4^5s_2s_1^6-2s_4^5s_2+2ys_4^4s_3^2-
2s_4^4s_2^2+2s_4^4s_2s_1+2s_4^3s_3^3-2s_4^3s_2^3+2s_4^3s_2s_1^2+\\
&
2s_4^2s_3^4-2s_4^2s_2^4+2s_4^2s_2s_1^3+2s_4s_3^5-s_4s_3s_2^4-s_4s_3s_2s_1^3+2s_4s_3s_1^4-
s_4s_2^5+s_4s_2s_1^4-2s_3^6+2s_3^5s_2s_1^6-\\
&
2s_3^5s_2+2s_3^4s_2^2-2s_3^4s_2s_1+3s_3^3s_2^3-3s_3^3s_2s_1^2+
3s_3^2s_2^6s_1^4+2s_3^2s_2^5s_1^5-3s_3^2s_2^4s_1^6-s_3^2s_2^4+2s_3^2s_2^3s_1-\\
&
s_3^2s_2^2s_1^2-2s_3^2s_1^4-s_3s_2^6s_1^5+s_3s_2^5s_1^6+s_3s_2^4s_1-2s_3s_2^3s_1^2+
s_3s_2^2s_1^3-3s_2^6-s_2^5s_1-s_2^4s_1^2-s_2^3s_1^3-s_2^2s_1^4-\\
&
s_2s_1^5-s_1^6)-s_4^6s_3^6+2s_4^6s_2^6+3s_4^6s_2s_1^5-2s_4^6s_1^6-s_4^5s_3-s_4^5
s_2s_1^6+s_4^5s_2-s_4^4s_3^2+s_4^4s_2^2-s_4^4s_2s_1-\\
&
s_4^3s_3^3+s_4^3s_2^3-s_4^3s_2s_1^2-s_4^2s_3^4+s_4^2s_2^4-s_4^2s_2s_1^3-s_4s_3^5+
2s_4s_3s_2^4-s_4s_3s_1^4-s_4s_2^5+s_4s_2s_1^4-3s_3^6s_2^6+\\
&
2s_3^6s_2s_1^5-s_3^6s_1^6-s_3^6-2s_3^5s_2s_1^6+2s_3^5s_2+3s_3^3s_2^3-3s_3^3s_2
s_1^2-s_3^2s_2^6s_1^4-s_3^2s_2^5s_1^5+s_3^2s_2^4s_1^6+3s_3^2s_2^3s_1-\\
&
s_3^2s_2^2s_1^2-3s_3^2s_2s_1^3+s_3^2s_1^4-3s_3s_2^6s_1^5+3s_3
s_2^5s_1^6-3s_3s_2^5+2s_3s_2^4s_1+2s_3s_2^3s_1^2+s_3s_2^2s_1^3-2s_3s_2s_1^4+\\
&
2s_2^6s_1^6+2s_2^6+2s_2^5s_1+2s_2^4s_1^2+2s_2^3s_1^3+2s_2^2s_1^4+2s_2s_1^5.
\end{align*}}} 
\noindent In Example \ref{Degenere2} we the have following polynomial
{\scriptsize{
\begin{align*}
\mathcal{L}_{xy}=&\mathbf{y}^2+\mathbf{y}(3s_4s_3s_2^4+2s_4s_3s_2s_1^3-3s_4s_2^5+3s_4s_2s_1^4+s_3^6
-s_3^5s_2s_1^6+s_3^5s_2+s_3^4s_2^2-s_3^4s_2s_1+s_3^3s_2^3-\\
&
s_3^3s_2s_1^2-3s_3^2s_2^5s_1^5-s_3^2s_2^4s_1^6-2s_3^2s_2^4
-2s_3^2s_2^3s_1-s_3^2s_2^2s_1^2-3s_3^2s_2s_1^3-s_3s_2^6s_1^5+2s_3s_2^5s_1^6-s_3s_2^4s_1-\\
&
2s_3s_2^3s_1^2-s_3s_2^2s_1^3+3s_3s_2s_1^4-3s_2^5s_1-3s_2^4s_1^2-3s_2^3s_1^3-3s_2^2s_1^4-
3s_2s_1^5-3s_1^6-3)+\\
&
\mathbf{x}(2s_3^6+2s_2^6+3s_2s_1^5+s_1^6-3)+3s_4s_3s_2^4-3s_4s_3s_2s_1^3-2s_4
s_3s_1^4-3s_4s_2^5+3s_4s_2s_1^4+2s_3^6s_2^6-\\
&
3s_3^6s_2s_1^5-3s_3^6s_1^6+3s_3^5s_2s_1^6-3s_3^5s_2-3s_3^4s_2^2+3s_3^4s_2s_1
-3s_3^3s_2^3+3s_3^3s_2s_1^2+2s_3^2s_2^6s_1^4-s_3^2s_2^5s_1^5+\\
&
s_3^2s_2^4+s_3^2s_2^3s_1+s_3^2s_2^2s_1^2+3s_3^2s_2s_1^3+2s_3^2s_1^4-2s_3
s_2^6s_1^5+3s_3s_2^5s_1^6+3s_3s_2^5-3s_3s_2^4s_1-3s_3s_2^3s_1^2-\\
&
2s_3s_2^2s_1^3-3s_3s_2s_1^4-2s_2^6s_1^6+3s_2^6-s_2^5s_1-s_2^4s_1^2-
s_2^3s_1^3-s_2^2s_1^4+3s_2s_1^5.
\end{align*}}} 
\noindent In Example \ref{EXHerm1} we have the following polynomials
{\scriptsize{
\begin{align*}
 f_{x}=& \mathbf{x}^2+\mathbf{x}s_4s_2^2+\mathbf{e}(s_4^3s_3^2s_2s_1^2+s_4^3s_3^2+s_4^3s_3s_2+
s_4^3s_3s_1+s_4^3s_2^3s_1^2+s_4^3s_2^2+s_4^3s_1^2+s_4^2s_3^2s_2^3s_1+s_4^2s_3^2s_2^2s_1^2+\\
&
s_4^2s_3s_2^3s_1^2+s_4^2s_3s_2^2+s_4s_3^2s_2^2+s_4s_3s_2s_1^2+s_4s_2^3s_1+s_4s_1
+s_3^2s_2^3+s_3^2+s_3s_2^3s_1+s_3s_1+s_2^3s_1^2+s_2^2+s_1^2)+\\
&
s_5^2s_3+s_5s_3s_2+s_4^3s_3^2s_2+s_4^3s_3^2s_1+s_4^3s_3s_2s_1+
s_4^3s_3s_1^2+s_4^3s_2^3+s_4^3s_2^2s_1+s_4^3+s_4^2s_3^3s_2+s_4^2s_3^2s_2^3s_1^2+\\
&
s_4^2s_3^2s_2^2+s_4^2s_3^2s_2s_1+s_4^2s_3s_2^3+s_4^2s_3s_2^2s_1+s_4^2s_3s_2s_1^2+s_4^2s_2s_1^3
+s_4^2s_2+s_4s_3^3s_1^2+s_4s_3^2s_2^2s_1+s_4s_3^2s_2s_1^2+\\
&
s_4s_3^2s_1^3+s_4s_3s_2^2s_1^2+s_4s_3s_2+s_4s_3s_1+s_3^3s_2^2s_1+s_3^2s_2^3s_1
+s_3^2s_2^2s_1^2+s_3^2s_1+s_3s_2^3s_1^2+s_3s_2^2s_1^3+s_3s_2^2+\\
&
s_3s_1^2+s_2^3+s_2^2s_1+1,\\
&  \\
g_{x}=&
{\mathbf x_1}+{\mathbf x_2}+s_4^3s_3^2s_2+s_4^3s_3^2s_1+s_4^3s_3s_2s_1+s_4^3s_3s_1^2+s_4^3s_2^3+s_4^3+
s_4^2s_3^2s_2^3s_1^2+s_4^2s_3^2s_2^2s_1^3+s_4^2s_3s_2^3+s_4^2s_3s_2^2s_1+\\
 &
s_4^2s_3s_1^3+s_4^2s_3+s_4^2s_1+s_4s_3^2s_2^2s_1+s_4s_3^2s_2s_1^2+
s_4s_3s_2^2s_1^2+s_4s_3s_2s_1^3+s_4s_2^2s_1^3+s_4s_2^2+s_3^2s_2^3s_1+s_3^2s_2s_1^3+ \\
 & 
s_3^2s_2+s_3^2s_1+s_3s_2^3s_1^2+s_3s_1^2+s_2^3s_1^3+s_2s_1^2+s_1^3
\end{align*}}} 
{\scriptsize{
\begin{align*}
 f_{y}=& \mathbf{y}^2+\mathbf{y}(s_4s_3s_2+s_2^3+1)+\mathbf{e}(s_4^3s_3^3s_2^3s_1^2+s_4^3s_3s_2^3
s_1+s_4^3s_2^3s_1^2+s_4^2s_3^3s_2^2s_1+s_4^2s_2^2s_1+s_4s_3^2s_2s_1+s_4s_3s_2s_1^2+\\
 &
s_3^2s_2^3+s_3s_2^3s_1+s_2^3s_1^2)+s_5^3+s_5s_4^2s_3^2s_2+
s_5s_3^3s_2^2+s_5s_2^2+s_4^3s_3^3s_2^3+s_4^3s_3^2s_2^3s_1+s_4^3s_3^2s_1+s_4^3s_3s_2^3s_1^2+\\
 &
s_4^3s_3s_1^2+s_4^3s_2^3+
s_4^2s_3^2s_2^2s_1^3+s_4^2s_3s_2^2s_1+s_4^2s_2^2s_1^2+s_4s_3^3s_2s_1+
s_4s_3s_2s_1^3+s_4s_3s_2+s_3s_2^3s_1^2+ s_2^3s_1^3+s_2^3,\\
&  \\
g_{y}=& \mathbf{y_1}+\mathbf{y_2}+s_4^3s_3^3+s_4^3s_3s_2^3s_1^2+s_4^3s_2^3+s_4^2s_3^3s_2^2s_1^2+
s_4^2s_3s_2^2s_1+s_4s_3^3s_2s_1+s_4s_3s_2+s_4s_2s_1+s_3^3s_2^3s_1^3+s_3^3+\\
&
s_3^2s_1+s_3s_1^2+s_2^3s_1^3+s_2^3+s_1^3.
\end{align*}}}
\noindent In example \ref{EXHerm2} we have the following polynomials
{\scriptsize{
\begin{align*}
\mathcal{L}_{x}=& \mathbf{x}^2+\mathbf{x}(s_2s_3^2s_4^3+s_1s_3^2s_4^3+s_1s_2s_3s_4^3+s_1^2s_3s_4^3+s_2^3s_4^3+
s_4^3+s_1^2s_2^3s_3^2s_4^2+s_1^3s_2^2s_3^2s_4^2+s_1^3s_2^3s_3s_4^2+s_1^4s_2^2s_3s_4^2+\\
&
s_2^4s_4^2+s_2s_4^2+s_1^4s_4^2+s_1s_2^2s_3^2s_4+s_1^2s_2s_3^2s_4+s_1^2s_2^2s_3s_4+
s_1^3s_2s_3s_4+s_1s_2^4s_4+s_1^3s_2^2s_4+s_2^2s_4+s_1s_2s_4+\\
&
s_2^4s_3^2+s_1s_2^3s_3^2+s_1^3s_2s_3^2+s_1^4s_3^2+
s_1s_2^4s_3+s_1^2s_2^3s_3+s_1s_2s_3+s_1^2s_3+s_1^3s_2^3+s_1^4s_2^2+s_1s_2^2+s_1^2s_2+s_1^3)+\\
&
s_3s_5^2+s_2s_3s_5+s_1^2s_2^2s_3^2s_4^3+s_1^3s_2s_3^2s_4^3+s_2^2s_3s_4^3+s_1s_2s_3s_4^3+
s_1^2s_2^4s_4^3+s_1^3s_2^3s_4^3+s_2^3s_4^3+s_1^2s_2s_4^3+s_2s_3^3s_4^2+\\
&
s_1s_2^4s_3^2s_4^2+s_1^2s_2^3s_3^2s_4^2+s_1s_2s_3^2s_4^2+s_1^2s_2^4s_3s_4^2+
s_1^3s_2^3s_3s_4^2+s_1^2s_2s_3s_4^2+s_2s_4^2+s_1^2s_3^3s_4+s_1^3s_2^3s_3^2s_4+\\
&
s_1^4s_2^2s_3^2s_4+s_1^3s_3^2s_4+s_1s_2^3s_3s_4+
s_1^2s_2^2s_3s_4+s_1s_3s_4+s_1s_2^4s_4+s_1^4s_2s_4+s_1^2s_4+s_1^3
s_2^3s_3^3+s_2^3s_3^3+s_1s_2^2s_3^3+\\
&
s_1^3s_3^3+s_3^3+s_1^2s_2^2s_3^2+s_1^4s_2s_3+s_1s_2s_3+s_1^4s_2^2+s_1^3+1.
\\
&\\
\mathcal{L}_{xy}=&\mathbf{y}^2+\mathbf{y}(s_2^3s_3^3s_4^3+s_1s_2^2s_3^3s_4^3+s_1^2s_2s_3^3s_4^3+
s_1s_3^2s_4^3+s_1^2s_2^3s_3s_4^3+s_2^2s_3s_4^3+s_1s_2s_3s_4^3+s_1^2s_3s_4^3+s_2^3s_4^3+\\
&
s_1s_2^3s_3^3s_4^2+s_1^2s_2^2s_3^3s_4^2+s_1^3s_2s_3^3s_4^2+s_1^3s_2^2s_3^2s_4^2
+s_2^3s_3s_4^2+s_1^2s_2s_3s_4^2+s_1^3s_3s_4^2+s_3s_4^2+s_1^2s_2^2s_4^2+\\
&
s_1^2s_2^3s_3^3s_4+s_1^3s_2^2s_3^3s_4+s_1s_2s_3^3s_4+s_1^2s_2s_3^2s_4+s_1s_2^3s_3s_4+
s_1^2s_2^2s_3s_4+s_1s_2s_4+s_1^3s_3^3+s_3^3+s_1s_2^3s_3^2+\\
&
s_1s_3^2+s_1^2s_2^3s_3+s_1^3s_2^2s_3+s_2^2s_3+s_1^2s_3+
s_1^3s_2^3+s_2^3)+x(s_2^2s_3s_4^3+s_1^2s_3s_4^3
+s_1s_2^2s_4^3+s_1^2s_2s_4^3+s_4^3+\\
&
s_1s_2^2s_3s_4^2+s_1^2s_2s_3s_4^2+s_1s_2^3s_4^2+s_1^2s_2^2s_4^2+
s_3^2s_4+s_1s_2^3s_3s_4+s_1^3s_2s_3s_4+s_1s_3s_4+s_1^3s_2^2s_4+s_1s_2s_4+\\
&
s_1^2s_4+s_1^2s_2^2s_3^2+s_1^2s_2^3s_3+s_2^2s_3+s_1^2s_3+s_2^3+s_1s_2^2+s_1^3+s_5^3)+s_2s_3^2s_4^2s_5+
s_3s_4s_5+s_2^2s_3^3s_5+s_2^2s_5+\\
&
s_1^2s_2^2s_3^2s_4^3+s_2s_3^2s_4^3+s_1s_3^2s_4^3+s_2^2s_3s_4^3+s_2^3s_4^3+
s_1s_2^2s_4^3+s_1^2s_2s_4^3+s_1^2s_2^3s_3^2s_4^2+s_1^3s_2^2s_3^2s_4^2+s_1s_2s_3^2s_4^2+\\
&
s_1s_2^2s_3s_4^2+s_1^2s_2s_3s_4^2+s_1s_2^3s_4^2+s_1^3s_2s_4^2+
s_1s_2s_3^3s_4+s_2^3s_3^2s_4+s_1s_2^2s_3^2s_4+s_1^2s_2s_3^2s_4+s_1^3s_3^2s_4+s_3^2s_4+\\
&
s_1s_2^3s_3s_4+s_1^3s_2s_3s_4+
s_2s_3s_4+s_1^2s_2^3s_4+s_1^3s_2^2s_4+s_1s_2s_4+s_1^3s_2^3s_3^3+s_3^3+s_1s_2^3s_3^2+
s_1^3s_2s_3^2+s_2s_3^2+\\
& 
s_1s_3^2+s_1^3s_2^2s_3+s_2^2s_3+s_1^2s_3+s_1^3+1.
\end{align*}}} 

\noindent In Subsection \ref{HC} we have the  following polynomials
{\scriptsize{
\begin{align*}
\mathcal{E}=&\mathbf{e}_2^2-\mathbf{e}_2s_1-s_7s_5^3s_2^6-s_7s_5-s_7s_4s_3^3
s_2^7s_1^6-s_7s_4s_3s_2^7-s_7s_4s_2^3s_1^5+s_7s_3^4s_2^5s_1^8-s_7s_3^4s_2^5-
s_7s_3^3s_2s_1^5-\\
& 
s_7s_3s_2s_1^7+s_7s_2^5s_1^4-s_6s_1-s_5^4s_4s_3^3s_2^6s_1^4-s_5^4s_4s_3s_2^6s_1^6+
s_5^4s_4s_2^2s_1^3+s_5^4s_3^3s_2^8s_1^3+s_5^4s_3^3s_1^3+s_5^4s_3s_2^8s_1^5+\\
& 
s_5^4s_3s_1^5+s_5^4s_2^4s_1^2+s_5^3s_4^2s_3^3s_2s_1+s_5^3s_4^2s_3s_2s_1^3-
s_5^3s_4^2s_2^5+s_5^3s_4s_3^4s_2^7s_1^3+s_5^3s_4s_3^2s_2^7s_1^5-s_5^3s_4s_3s_2^3s_1^2+\\
& 
s_5^3s_4s_2^7s_1^7-s_5^3s_3^6s_2s_1^8+s_5^3s_3^6s_2+s_5^3s_3^4s_2s_1^2
-s_5^3s_3^3s_2^5s_1^7+s_5^3s_3^2s_2s_1^4+s_5^3s_3s_2^5s_1+s_5^3s_2s_1^6-s_5^2s_4^5s_3^3s_1^8-\\
& 
s_5^2s_4^5s_3s_1^2+s_5^2s_4^5s_2^4s_1^7-s_5^2s_4^4s_3^3s_2^2s_1^7-s_5^2s_4^4s_3s_2^2s_1-
s_5^2s_4^4s_2^6s_1^6+s_5^2s_4^3s_3^6s_1^7-s_5^2s_4^3s_3^4s_1+s_5^2s_4^3s_3^2s_1^3+\\
& 
s_5^2s_4^3s_2^8s_1^5-s_5^2s_4s_3^3s_1^4-s_5^2s_4s_3s_1^6-s_5^2s_3^6s_2^6s_1^4+
s_5^2s_3^4s_2^6s_1^6-s_5^2s_3^2s_2^6s_1^8+s_5^2s_2^6s_1^2-s_5s_4^5s_3^4s_2s_1^7-\\
& 
s_5s_4^5s_3^2s_2s_1+s_5s_4^5s_3s_2^5s_1^6+s_5s_4^5s_2s_1^3-s_5s_4^4s_3^4s_2^3s_1^6+
s_5s_4^4s_3^3s_2^7s_1^3-s_5s_4^4s_3^2s_2^3s_1^8+s_5s_4^3s_3^7s_2s_1^6-\\
& 
s_5s_4^3s_3^5s_2s_1^8+s_5s_4^3s_3^3s_2s_1^2+s_5s_4^3s_3s_2s_1^4+s_5s_4^3s_2^5s_1+
s_5s_4^2s_2^7s_1^8-s_5s_4^2s_2^7+s_5s_4s_3^6s_2s_1+s_5s_4s_3^3s_2^5s_1^8-\\
& 
s_5s_4s_3^2s_2s_1^5-s_5s_5s_4s_3s_2^5s_1^2-s_3^7s_2^7s_1^3-s_5s_3^6s_2^3s_1^8+
s_5s_3^6s_2^3+s_5s_3^5s_2^7s_1^5-s_5s_3^4s_2^3s_1^2-s_5s_3^2s_2^3s_1^4-\\
& 
s_4^8s_3^6s_1^4+s_4^8s_3^4s_1^6-s_4^8s_3^3s_2^4s_1^3-s_4^8s_3^2s_1^8-s_4^8s_3s_2^4s_1^5-s_4^8s_2^8
s_1^2+s_4^8s_1^2-s_4^7s_3^6s_2^2s_1^3+s_4^7s_3^4s_2^2s_1^5-s_4^7s_3^3s_2^6s_1^2-\\
& 
s_4^7s_3^2s_2^2s_1^7-s_4^7s_3s_2^6s_1^4-s_4^6s_3^6s_2^4s_1^2+s_4^6s_3^4s_2^4s_1^4-
s_4^6s_3^3s_2^8s_1+s_4^6s_3^3s_1-s_4^6s_3^2s_2^4s_1^6-s_4^6s_3s_2^8s_1^3+s_4^6s_3s_1^3-\\
& 
s_4^6s_2^4s_1^8+s_4^6s_2^4-s_4^5s_3^5s_2^2s_1^6-s_4^5s_3^3s_2^2s_1^8-s_4^5s_3^3s_2^2+
s_4^5s_3^2s_2^6s_1^5+s_4^5s_3s_2^2s_1^2-s_4^4s_3^6-s_4^4s_3^5s_2^4s_1^5-s_4^4s_3^4s_2^8s_1^2+\\
& 
s_4^4s_3^4s_1^2+s_4^4s_3^2s_2^8s_1^4-s_4^4s_3^2s_1^4+s_4^4s_3s_2^4s_1+s_4^4s_2^8s_1^6+
s_4^4s_1^6+s_4^3s_3^8s_2^2s_1^5-s_4^3s_3^2s_2^2s_1^3+s_4^3s_3s_2^6s_1^8+s_4^3s_3s_2^6+\\
& 
s_4^3s_2^2s_1^5+s_4^2s_3^4s_2^4s_1^8+s_4^2s_3^2s_2^4s_1^2-s_4^2s_3s_2^8s_1^7-s_4^2s_2^4s_1^4
+s_4s_3^7s_2^2s_1^8+s_4s_3^7s_2^2-s_4s_3^6s_2^6s_1^5+s_4s_3^5s_2^2s_1^2-\\
& 
s_4s_3^3s_2^2s_1^4+s_4s_3s_2^2s_1^6+s_4s_2^6s_1^3-s_3^8s_2^8s_1^2-s_3^6s_2^8s_1^4+
s_3^5s_2^4s_1-s_3^4s_2^8s_1^6-s_3^3s_2^4s_1^3-s_3^2s_2^8s_1^8+s_3^2-s_3s_2^4s_1^5.
\end{align*}}} 
{\scriptsize{
\begin{align*}
\mathcal{P}_{x}=& 
\mathbf{x}_2^2+\mathbf{x}_2(s_7s_4s_3^3s_1^3+s_7s_4s_3s_1^5-s_7s_4s_2^4s_1^2-s_7s_3^3s_2^2s_1^2-
s_7s_3s_2^2s_1^4+s_7s_2^6s_1-s_5^3s_4^5-s_5^3s_4^4s_2^2s_1^7+\\
&
s_5^3s_4^3s_3^3s_1^7+s_5^3s_4^3s_3s_1+s_5^3s_4^3s_2^4s_1^6+s_5^3s_4^2s_2^6s_1^5-s_5^3s_4s_3^6s_1^6+
s_5^3s_4s_3^4-s_5^3s_4s_3^3s_2^4s_1^5-s_5^3s_4s_3^2s_1^2-\\
&
s_5^3s_4s_3s_2^4s_1^7+s_5^3s_3^6s_2^2s_1^5-s_5^3s_3^4s_2^2s_1^7+s_5^3s_3^2s_2^2
s_1-s_5s_4^6s_1+s_5s_4^5s_2^2s_1^8-s_5s_4^5s_2^2+s_5s_4^4s_3^3s_1^8+s_5s_4^4s_3s_1^2-\\
&
s_5s_4^4s_2^4s_1^7-s_5s_4^3s_3^3s_2^2s_1^7-s_5s_4^3s_3s_2^2s_1-s_5s_4^2s_3^6
s_1^7+s_5s_4^2s_3^4s_1-s_5s_4^2s_3^3s_2^4s_1^6-s_5s_4^2s_3^2s_1^3-s_5s_4^2s_3s_2^4s_1^8-\\
&
s_5s_4^2s_2^8s_1^5-s_5s_4s_3^6s_2^2s_1^6+s_5s_4s_3^4s_2^2+s_5s_4s_3^3s_2^6s_1^5
-s_5s_4s_3^2s_2^2s_1^2+s_5s_4s_3s_2^6s_1^7-s_5s_3^6s_2^4s_1^5+s_5s_3^4s_2^4s_1^7-\\
&
s_5s_3^2s_2^4s_1+s_4^8s_1^8+s_4^7s_2^3s_1^6+s_4^7s_2^2s_1^7-s_4^6s_3^3s_2s_1^6+s_4^6s_2^4s_1^6+
s_4^5s_3^4s_2^7s_1^8-s_4^5s_3^4s_2^7-s_4^5s_3s_2^3s_1^7+s_4^5s_2^7s_1^4+\\
&
s_4^5s_2^6s_1^5+s_4^4s_3^6s_2s_1^5+s_4^4s_3^4s_2s_1^7+s_4^4s_3^3s_2^5s_1^4+
s_4^4s_3s_2^5s_1^6+s_4^4s_2^8s_1^4-s_4^3s_3^6s_2^3s_1^4+s_4^3s_3^2s_2^3s_1^8+s_4^3s_3s_2^7s_1^5+\\
&
s_4^3s_2^3s_1^2+s_4^3s_2^2s_1^3+s_4^2s_3^7s_2s_1^6-s_4^2s_3^5s_2s_1^8+s_4^2s_3^4s_2^5s_1^5+
s_4^2s_3^3s_2s_1^2+s_4^2s_3^2s_2^5s_1^7-s_4^2s_3s_2s_1^4-s_4^2s_2^5s_1+s_4^2s_2^4s_1^2+\\
&
s_4s_3^8s_2^7s_1^8-s_4s_3^8s_2^7+s_4s_3^4s_2^7s_1^4+s_4s_3^2s_2^7s_1^6-s_4s_2^7
s_1^8+s_4s_2^7+s_4s_2^6s_1-s_3^7s_2^5s_1^4+s_3^5s_2^5s_1^6-s_3^3s_2^5s_1^8+s_2^8-1)-\\
&
s_7s_4s_3^3s_2s_1^2-s_7s_4s_3s_2s_1^4+s_7s_4s_2^5s_1+s_7s_3^3s_2^3s_1+s_7s_3s_2^3s_1^3-s_7s_2^7+
s_5^3s_4^5s_2s_1^7+s_5^3s_4^4s_2^3s_1^6-s_5^3s_4^3s_3^3s_2s_1^6-\\
&
s_5^3s_4^3s_3s_2-s_5^3s_4^3s_2^5s_1^5-s_5^3s_4^2s_2^7s_1^4+s_5^3s_4s_3^6s_2s_1^5
-s_5^3s_4s_3^4s_2s_1^7+s_5^3s_4s_3^3s_2^5s_1^4+s_5^3s_4s_3^2s_2s_1+s_5^3s_4s_3s_2^5s_1^6-\\
&
s_5^3s_3^6s_2^3s_1^4+s_5^3s_3^4s_2^3s_1^6-s_5^3s_3^2s_2^3-s_5^2s_3^6s_2^8s_1^8+
s_5^2s_3^6s_2^8+s_5^2s_3^6s_1^8-s_5^2s_3^6+s_5s_4^6s_2s_1^8-s_5s_4^4s_3^3s_2s_1^7-s_5s_4^4s_3s_2s_1+\\
&
s_5s_4^4s_2^5s_1^6+s_5s_4^3s_3^3s_2^3s_1^6-s_5s_4^3s_3s_2^3s_1^8-s_5s_4^3s_3s_2^3+
s_5s_4^2s_3^6s_2s_1^6-s_5s_4^2s_3^4s_2s_1^8+s_5s_4^2s_3^3s_2^5s_1^5+s_5s_4^2s_3^2s_2s_1^2+\\
&
s_5s_4^2s_3s_2^5s_1^7+s_5s_4^2s_2s_1^4+s_5s_4s_3^6s_2^3s_1^5-s_5s_4s_3^4s_2^3s_1^7-
s_5s_4s_3^3s_2^7s_1^4+s_5s_4s_3^2s_2^3s_1-s_5s_4s_3s_2^7s_1^6+s_5s_3^7s_2^8s_1^8-\\
&
s_5s_3^7s_2^8-s_5s_3^7s_1^8+s_5s_3^7+s_5s_3^6s_2^5s_1^4-s_5s_3^4s_2^5s_1^6-s_5s_3^2
s_2^5s_1^8-s_5s_3^2s_2^5-s_4^8s_2s_1^7+s_4^7s_3s_1^8-s_4^7s_3-s_4^7s_2^4s_1^5-\\
&
s_4^7s_2^3s_1^6+s_4^6s_3^3s_2^2s_1^5-s_4^6s_2^5s_1^5+s_4^5s_3s_2^4s_1^6-s_4^5s_2^8s_1^3-s_4^5s_2^7
s_1^4-s_4^4s_3^6s_2^2s_1^4-s_4^4s_3^4s_2^2s_1^6-s_4^4s_3^3s_2^6s_1^3+s_4^4s_3^2s_2^2s_1^8-\\
&
s_4^4s_3^2s_2^2-s_4^4s_3s_2^6s_1^5-s_4^4s_2s_1^3+s_4^3s_3^6s_2^4s_1^3-s_4^3s_3^5s_1^8+
s_4^3s_3^5-s_4^3s_3^2s_2^4s_1^7-s_4^3s_3s_2^8s_1^4-s_4^3s_2^4s_1-s_4^3s_2^3s_1^2-\\
&
s_4^2s_3^7s_2^2s_1^5+s_4^2s_3^5s_2^2s_1^7-s_4^2s_3^4s_2^6s_1^4-s_4^2s_3^3s_2^2s_1-
s_4^2s_3^2s_2^6s_1^6+s_4^2s_3s_2^2s_1^3+s_4^2s_2^6-s_4^2s_2^5s_1-s_4s_3^4s_2^8s_1^3-\\
&
s_4s_3^2s_2^8s_1^5-s_4s_2^7s_1^8-s_4s_1^7+s_3^7s_2^6s_1^3-s_3^6s_2^2s_1^8+s_3^6s_2^2-
s_3^5s_2^6s_1^5+s_3^3s_2^6s_1^7.
\end{align*}}} 
{\scriptsize{
\begin{align*}
\mathcal{P}_{xy}=& 
\mathbf{y}_2^2-\mathbf{y}_2(s_6^2s_3^3s_1^3-s_6^2s_3s_1^5+s_6^2s_2^4s_1^2-s_6s_3^7+
s_6s_3^6s_1-s_6s_3^4s_2^4s_1^7-s_6s_3^4s_1^3+s_6s_3^3s_2^4+s_6s_3^2s_1^5-s_6s_3s_2^8s_1^6+\\
&
s_6s_3s_2^4s_1^2+s_6s_2^8s_1^7+s_6s_2^4s_1^3-s_6s_1^7-s_3^8+s_3^6s_2^4s_1^6-s_3^6s_1^2
-s_3^5s_2^4s_1^7-s_3^5s_1^3+s_3^4s_2^4+s_3^4s_1^4+s_3^3s_2^8s_1^5-s_3^3s_1^5-\\
&
s_3^2s_2^8s_1^6+s_3^2s_2^4s_1^2-s_3^2s_1^6-s_3s_2^8s_1^7+s_3s_2^4s_1^3+s_2^8+s_1^8+1)
+\mathbf{x}_2(s_7s_4s_3^3s_2^6s_1^5+s_7s_4s_3s_2^6s_1^7-s_7s_4s_2^2s_1^4-\\
&
s_7s_3^3s_2^8s_1^4-s_7s_3s_2^8s_1^6+s_7s_2^4s_1^3+s_5^4s_2^3s_1+s_5^3s_4^4s_1-
s_5^3s_4^3s_2^2-s_5^3s_4^2s_3^7s_1^4+s_5^3s_4^2s_3^6s_1^5+s_5^3s_4^2s_3^5s_1^6-
s_5^3s_4^2s_3^4s_2^4s_1^3-\\
&
s_5^3s_4^2s_3^4s_1^7+s_5^3s_4^2s_3^3s_2^4s_1^4-s_5^3s_4^2s_3^2s_2^4s_1^5+s_5^3s_4^2s_3^2s_1-
s_5^3s_4^2s_3s_2^8s_1^2+s_5^3s_4^2s_3s_2^4s_1^6-s_5^3s_4^2s_3s_1^2+s_5^3s_4^2s_2^8s_1^3-\\
&
s_5^3s_4^2s_2^4s_1^7-s_5^3s_4^2s_1^3-s_5^3s_4s_3^7s_2^2s_1^3-s_5^3s_4s_3^6s_2^6+
s_5^3s_4s_3^6s_2^2s_1^4+s_5^3s_4s_3^5s_2^2s_1^5-s_5^3s_4s_3^4s_2^2s_1^6+
s_5^3s_4s_3^3s_2^6s_1^3-\\
&
s_5^3s_4s_3^3s_2^2s_1^7+s_5^3s_4s_3^2s_2^6s_1^4+s_5^3s_4s_3^2s_2^2+
s_5^3s_4s_3s_2^6s_1^5-s_5^3s_3^7s_2^4s_1^2+s_5^3s_3^6s_2^8s_1^7+s_5^3s_3^6s_2^4s_1^3
-s_5^3s_3^6s_1^7+s_5^3s_3^5s_2^4s_1^4-\\
&
s_5^3s_3^5s_1^8+s_5^3s_3^5+s_5^3s_3^4s_2^8s_1-s_5^3s_3^4s_2^4s_1^5+s_5^3s_3^4
s_1+s_5^3s_3^3s_2^8s_1^2+s_5^3s_3^2s_2^4s_1^7-
s_5^3s_3^2s_1^3+s_5^3s_3s_2^8s_1^4-s_5^3s_3s_2^4s_1^8+
\end{align*}}} 
{\scriptsize{
\begin{align*}
&
s_5^3s_3s_2^4+s_5^3s_1^5+s_5^2s_4^3s_3^3s_2^3s_1^5+s_5^2s_4^3s_3s_2^3s_1^7-
s_5^2s_4^3s_2^7s_1^4+s_5^2s_4s_2^3s_1^2-
s_5^2s_2^5s_1+s_5s_4^7s_3s_1^7-s_5s_4^7s_1^8-\\
&
s_5s_4^6s_3s_2^2s_1^6+s_5s_4^6s_2^6s_1^3+
s_5s_4^6s_2^2s_1^7-s_5s_4^5s_3^4s_1^6+s_5s_4^5s_3^3s_1^7-s_5s_4^5s_3^2s_1^8+
s_5s_4^5s_3s_2^4s_1^5+s_5s_4^5s_3s_1-s_5s_4^5s_2^4s_1^6-\\
&
s_5s_4^5s_1^2-s_5s_4^4s_3^4s_2^2s_1^5+
s_5s_4^4s_3^3s_2^2s_1^6-s_5s_4^4s_3^2s_2^2s_1^7-s_5s_4^4s_3s_2^6s_1^4+s_5s_4^4
s_3s_2^2s_1^8+s_5s_4^4s_2^6s_1^5-s_5s_4^4s_2^2s_1-\\
&
s_5s_4^3s_3^7s_1^5+s_5s_4^3s_3^6s_1^6+s_5s_4^3s_3^5s_1^7-s_5s_4^3s_3^4s_2^4s_1^4-
s_5s_4^3s_3^4s_1^8-s_5s_4^3s_3^3s_2^4s_1^5+
s_5s_4^3s_3^3s_1-s_5s_4^3s_3^2s_2^4s_1^6+\\
&
s_5s_4^3s_3^2s_1^2+s_5s_4^3s_3s_2^8s_1^3-s_5s_4^3s_3s_2^4s_1^7+s_5s_4^3s_3s_1^3
+s_5s_4^3s_2^8s_1^4+s_5s_4^3s_2^4s_1^8-s_5s_4^3s_2^4+
s_5s_4^3s_1^4-s_5s_4^2s_3^6s_2^6s_1-\\
&
s_5s_4^2s_3^4s_2^6s_1^3-s_5s_4^2s_3^3s_2^6s_1^4-s_5s_4^2s_3^3s_2^2s_1^8-
s_5s_4^2s_3s_2^6s_1^6+s_5s_4^2s_3s_2^2s_1^2+s_5s_4^2s_2^2s_1^3-
s_5s_4s_3^6s_2^8+s_5s_4s_3^6s_1^8-\\
&
s_5s_4s_3^4s_2^8s_1^2+s_5s_4s_3^4s_1^2-s_5s_4s_3^3s_2^8s_1^3+s_5s_4s_3^3s_1^3-
s_5s_4s_3^2s_2^4s_1^8+s_5s_4s_3^2s_2^4-s_5s_4s_3s_2^8s_1^5-s_5s_4s_3s_2^4s_1+\\
&
s_5s_4s_3s_1^5+s_5s_4s_2^8s_1^6-
s_5s_4s_2^4s_1^2+s_5s_3^7s_2^6s_1^2-s_5s_3^6s_2^6s_1^3+s_5s_3^6s_2^2s_1^7-
s_5s_3^5s_2^6s_1^4+s_5s_3^5s_2^2s_1^8-s_5s_3^5s_2^2+\\
&
s_5s_3^4s_2^6s_1^5-s_5s_3^4s_2^2s_1-s_5s_3^2s_2^6s_1^7+s_5s_3^2s_2^2s_1^3+s_5
s_3s_2^6+s_5s_2^6s_1-s_5s_2^2s_1^5+s_4^8s_3s_2^3s_1^4-s_4^8s_2^3s_1^5-
s_4^7s_3^2s_2s_1^6+\\
&
s_4^7s_3s_2^5s_1^3+s_4^7s_3s_2s_1^7-s_4^7s_2^5s_1^4-s_4^7s_2s_1^8-s_4^7s_2+
s_4^6s_3^4s_2^3s_1^3-s_4^6s_3^3s_2^3s_1^4-s_4^6s_3^2s_2^3s_1^5-s_4^6s_3s_2^7
s_1^2+s_4^6s_3s_2^3s_1^6+\\
&
s_4^6s_2^3s_1^7+s_4^5s_3^7s_2s_1^3+s_4^5s_3^6s_2^5s_1^8
-s_4^5s_3^6s_2^5-s_4^5s_3^6s_2s_1^4+s_4^5s_3^4s_2^5s_1^2-
s_4^5s_3^3s_2^5s_1^3+s_4^5s_3^2s_2-s_4^5s_2^5s_1^6-s_4^5s_2s_1^2-\\
&
s_4^4s_3^7s_2^3s_1^2
+s_4^4s_3^6s_2^7s_1^7+s_4^4s_3^6s_2^3s_1^3-s_4^4s_3^5s_2^7s_1^8+
s_4^4s_3^5s_2^7-s_4^4s_3^5s_2^3s_1^4+s_4^4s_3^4s_2^7s_1+s_4^4s_3^4s_2^3s_1^5+s_4^4s_3^3s_2^7s_1^2
-\\
&
s_4^4s_3^3s_2^3s_1^6+s_4^4s_3^2s_2^7s_1^3+s_4^4s_3s_2^3+s_4^3s_3^8s_2s_1^4-s_4^3s_3^7s_2s_1^5+
s_4^3s_3^5s_2s_1^7+s_4^3s_3^4s_2^5s_1^4+s_4^3s_3^4s_2s_1^8+s_4^3s_3^4s_2-\\
&
s_4^3
s_3^3s_2^5s_1^5+s_4^3s_3^3s_2s_1+s_4^3s_3^2s_2^5s_1^6+s_4^3s_3^2s_2s_1^2-s_4^3s_3s_2^5s_1^7-
s_4^3s_2^5s_1^8-s_4^2s_3^8s_2^3s_1^3+s_4^2s_3^7s_2^7s_1^8+s_4^2s_3^7s_2^3s_1^4-\\
&
s_4^2s_3^6s_2^3s_1^5-s_4^2s_3^5s_2^3s_1^6-s_4^2s_3^4s_2^7s_1^3+s_4^2s_3^4s_2^3s_1^7+
s_4^2s_3^3s_2^7s_1^4-s_4^2s_3^2s_2^7s_1^5-s_4^2s_3^2s_2^3s_1+s_4^2s_3s_2^7s_1^6+
s_4^2s_3s_2^3s_1^2-\\
&
s_4^2s_2^3s_1^3-s_4s_3^8s_2^5s_1^2+s_4s_3^7s_2^5s_1^3-s_4s_3^7s_2s_1^7-
s_4s_3^6s_2^5s_1^4-s_4s_3^5s_2^5s_1^5-s_4s_3^4s_2^5s_1^6+s_4s_3^4s_2s_1^2+s_4s_3^2s_2^5s_1^8
+\\
&
s_4s_3^2s_2^5+s_4s_3^2s_2s_1^4-s_4s_3s_2s_1^5+
s_4s_2^5s_1^2-s_4s_2s_1^6+s_3^8s_2^7s_1-s_3^7s_2^7s_1^2-s_3^6s_2^7s_1^3-
s_3^6s_2^3s_1^7+s_3^5s_2^7s_1^4+\\
&
s_3^5s_2^3s_1^8+s_3^5s_2^3-s_3^4s_2^3s_1-s_3^3s_2^7s_1^6+s_3^2s_2^7s_1^7+
s_3s_2^7s_1^8+s_3s_2^7-s_3s_2^3s_1^4-s_2^7s_1+s_2^3s_1^5)-\\
& s_7s_4s_3^3s_2^7s_1^4-s_7s_4s_3s_2^7s_1^6+s_7s_4s_2^3s_1^3+s_7s_3^3s_2s_1^3+
s_7s_3^2s_2^5s_1^8-s_7s_3^2s_2^5+s_7s_3s_2s_1^5-s_7s_2^5s_1^2+s_6^2s_3^6+\\
&
s_6^2s_3^3s_2^4s_1^7-s_6^2s_3^2s_1^4+s_6^2s_2^8s_1^6-s_6^2s_1^6-s_6s_3^8s_1^7-
s_6s_3^7s_1^8-s_6s_3^5s_2^4s_1^6+s_6s_3^5s_1^2-s_6s_3^4s_2^4s_1^7-s_6s_3^3s_2^4-s_6s_3^3s_1^4-\\
&
s_6s_3^2s_2^8s_1^5-s_6s_3^2s_2^4s_1-s_6s_3s_2^8s_1^6+s_6s_3s_1^6+
s_6s_2^8s_1^7+s_6s_1^7-s_5^4s_3^4s_2^8s_1^8+s_5^4s_3^4s_2^8+s_5^4s_3^4s_1^8-s_5^4s_3^4-
s_5^4s_2^4s_1^8-\\
&
s_5^3s_4^4s_2+s_5^3s_4^3s_3^6s_2^3s_1-s_5^3s_4^3s_3^4s_2^3s_1^3+s_5^3s_4^3s_3^3s_2^7+
s_5^3s_4^3s_3^2s_2^3s_1^5+s_5^3s_4^3s_3s_2^7s_1^2+s_5^3s_4^3s_2^3s_1^7+s_5^3s_4^2s_3^7s_2s_1^3-\\
&
s_5^3s_4^2s_3^6s_2s_1^4-s_5^3s_4^2s_3^5s_2s_1^5+s_5^3s_4^2s_3^4s_2^5s_1^2+s_5^3s_4^2s_3^4
s_2s_1^6-s_5^3s_4^2s_3^3s_2^5s_1^3+s_5^3s_4^2s_3^2s_2^5s_1^4-s_5^3s_4^2s_3^2s_2-s_5^3s_4^2s_3s_2^5s_1^5-\\
&
s_5^3s_4^2s_3s_2s_1+s_5^3s_4^2s_2^5s_1^6+s_5^3s_4s_3^7s_2^3s_1^2+s_5^3s_4s_3^6s_2^7s_1^7-
s_5^3s_4s_3^6s_2^3s_1^3-s_5^3s_4s_3^5s_2^3s_1^4+s_5^3s_4s_3^4s_2^3s_1^5-s_5^3s_4s_3^3s_2^7s_1^2+\\
&
s_5^3s_4s_3^3s_2^3s_1^6-s_5^3s_4s_3^2s_2^7s_1^3-s_5^3s_4s_3^2s_2^3s_1^7-s_5^3s_4s_3s_2^7s_1^4
-s_5^3s_3^8s_2^5s_1^8+s_5^3s_3^8s_2^5+s_5^3s_3^7s_2^5s_1-s_5^3s_3^6s_2^5s_1^2-s_5^3s_3^6s_2s_1^6-\\
&
s_5^3s_3^5s_2^5s_1^3+s_5^3s_3^4s_2^5s_1^4-s_5^3s_3^4s_2-s_5^3s_3^3s_2^5s_1^5
-s_5^3s_3^3s_2s_1-s_5^3s_3^2s_2^5s_1^6-s_5^3s_3s_2^5s_1^7-s_5^3s_3s_2s_1^3-s_5^3s_2s_1^4-
s_5^2s_4^3s_3^3s_2^4s_1^4-\\
&
s_5^2s_4^3s_3s_2^4s_1^6+s_5^2s_4^3s_2^8s_1^3-s_5^2s_4s_2^4s_1-s_5^2s_3^8s_2^6s_1^8+s_5^2s_3^8s_2^6
-s_5^2s_3^4s_2^2s_1^8+s_5^2s_3^4s_2^2-s_5^2s_2^6s_1^8-s_5^2s_2^6-s_5s_4^7s_3s_2s_1^6+\\
&
s_5s_4^7s_2s_1^7+s_5s_4^6s_3s_2^3s_1^5-
s_5s_4^6s_2^7s_1^2-s_5s_4^6s_2^3s_1^6+s_5s_4^5s_3^4s_2s_1^5-s_5s_4^5s_3^3s_2s_1^6+
s_5s_4^5s_3^2s_2s_1^7-s_5s_4^5s_3s_2^5s_1^4-\\
&
s_5s_4^5s_3s_2s_1^8+s_5s_4^5s_2^5s_1^5+
s_5s_4^5s_2s_1-s_5s_4^4s_3^6s_2^3s_1^2-s_5s_4^4s_3^4s_2^3s_1^4-s_5s_4^4s_3^3s_2^7s_1
-s_5s_4^4s_3^3s_2^3s_1^5-s_5s_4^4s_3s_2^3s_1^7-\\
&
s_5s_4^4s_2^7s_1^4-s_5s_4^4s_2^3s_1^8-s_5s_4^4s_2^3+s_5s_4^3s_3^7s_2s_1^4
+s_5s_4^3s_3^6s_2^5s_1-s_5s_4^3s_3^6s_2s_1^5
-s_5s_4^3s_3^5s_2s_1^6+s_5s_4^3s_3^4s_2s_1^7+\\
&
s_5s_4^3s_3^3s_2^5s_1^4-s_5s_4^3s_3^3s_2s_1^8+s_5s_4^3s_3^3s_2-s_5s_4^3s_3^2s_2^5s_1^5-s_5s_4^3s_3^2s_2s_1+s_5s_4^3s_3s_2^5s_1^6-
s_5s_4^3s_3s_2s_1^2+s_5s_4^3s_2s_1^3+\\
&
s_5s_4^2s_3^6s_2^7s_1^8+s_5s_4^2s_3^4s_2^7s_1^2+
s_5s_4^2s_3^3s_2^7s_1^3+s_5s_4^2s_3^3s_2^3s_1^7+s_5s_4^2s_3^2s_2^3s_1^8-s_5s_4^2s_3^2s_2^3+
s_5s_4^2s_3s_2^7s_1^5-s_5s_4^2s_3s_2^3s_1-\\
&
s_5s_4^2s_2^3s_1^2+s_5s_4s_3s_2^5s_1^8+
s_5s_4s_2^5s_1-s_5s_4s_2s_1^5+s_5s_3^8s_2^7s_1^8-s_5s_3^8s_2^7-s_5s_3^7s_2^7s_1
+s_5s_3^6s_2^7s_1^2-s_5s_3^6s_2^3s_1^6+\\
&
s_5s_3^5s_2^7s_1^3-s_5s_3^4s_2^7s_1^4+s_5s_3^4s_2^3s_1^8+
s_5s_3^2s_2^7s_1^6-s_5s_3^2s_2^3s_1^2-s_5s_3s_2^7s_1^7+s_5s_2^7s_1^8+s_5s_2^7+s_5s_2^3s_1^4-
s_4^8s_3s_2^4s_1^3+\\
&
s_4^8s_2^4s_1^4+s_4^8s_1^8-s_4^8-s_4^7s_3^3s_2^6s_1^8+s_4^7s_3^3s_2^6+s_4^7s_3^2s_2^2s_1^5-
s_4^7s_3s_2^6s_1^2-s_4^7s_3s_2^2s_1^6+s_4^7s_2^6s_1^3-s_4^7s_2^2s_1^7-s_4^6s_3^6s_2^4s_1^8-\\
&
s_4^6s_3^3s_2^8s_1^7+
s_4^6s_3^3s_2^4s_1^3+s_4^6s_3^2s_1^8-s_4^6s_3^2-s_4^6s_3s_2^4s_1^5-s_4^6s_2^4s_1^6-s_4^5s_3^7s_2^2s_1^2+
s_4^5s_3^6s_2^6s_1^7+s_4^5s_3^6s_2^2s_1^3+s_4^5s_3^4s_2^6s_1+\\
&
s_4^5s_3^3s_2^6s_1^2+s_4^5s_3^3s_2^2s_1^6+s_4^5s_3^2s_2^6s_1^3-s_4^5s_3^2s_2^2s_1^7+s_4^5s_3s_2^2s_1^8+
s_4^5s_2^6s_1^5+s_4^5s_2^2s_1-s_4^4s_3^8s_2^4s_1^8+s_4^4s_3^8s_2^4-s_4^4s_3^7s_2^4s_1-\\
&
s_4^4s_3^6s_2^4s_1^2+s_4^4s_3^6s_1^6-s_4^4s_3^4s_2^8-s_4^4s_3^4s_2^4s_1^4+s_4^4s_3^4s_1^8+
s_4^4s_3^4-s_4^4s_3^3s_2^8s_1+s_4^4s_3^3s_2^4s_1^5+s_4^4s_3^2s_2^8s_1^2+s_4^4s_3^2s_1^2+\\
&
s_4^4s_3s_2^4s_1^7+s_4^4s_2^8s_1^4-s_4^4s_1^4-s_4^3s_3^8s_2^2s_1^3+s_4^3s_3^7s_2^6s_1^8-
s_4^3s_3^7s_2^6+s_4^3s_3^7s_2^2s_1^4+s_4^3s_3^6s_2^2s_1^5-s_4^3s_3^5s_2^2s_1^6-s_4^3s_3^4s_2^6s_1^3-\\
&
s_4^3s_3^3s_2^6s_1^4-s_4^3s_3^3s_2^2s_1^8-s_4^3s_3^2s_2^6s_1^5-s_4^3s_3s_2^6s_1^6+
s_4^3s_2^6s_1^7+s_4^2s_3^8s_2^4s_1^2-s_4^2s_3^7s_2^8s_1^7-s_4^2s_3^7s_2^4s_1^3+s_4^2s_3^5s_2^4s_1^5+\\
&
s_4^2s_3^4s_2^8s_1^2+s_4^2s_3^3s_2^8s_1^3+s_4^2s_3^2s_2^8s_1^4+s_4^2s_3s_2^8s_1^5-s_4^2s_3s_2^4s_1+
s_4^2s_2^4s_1^2+s_4s_3^8s_2^6s_1-s_4s_3^7s_2^6s_1^2+s_4s_3^7s_2^2s_1^6+s_4s_3^6s_2^6s_1^3+\\
&
s_4s_3^5s_2^6s_1^4+s_4s_3^4s_2^6s_1^5-s_4s_3^4s_2^2s_1+s_4s_3^2s_2^6s_1^7
-s_4s_3^2s_2^2s_1^3-s_4s_3s_2^6s_1^8+s_4s_3s_2^6+s_4s_3s_2^2s_1^4-s_4s_2^6s_1+s_4s_2^2s_1^5+\\
&
s_3^8s_2^8s_1^8+s_3^8s_2^8+s_3^8s_1^8+s_3^7s_2^8s_1-s_3^7s_2^4s_1^5+s_3^7s_1
-s_3^6s_2^8s_1^2-s_3^6s_2^4s_1^6-s_3^5s_2^8s_1^3+s_3^5s_2^4s_1^7-s_3^5s_1^3+s_3^4s_2^8s_1^4-s_3^4s_2^4s_1^8-\\
&
s_3^4s_2^4-s_3^4s_1^4-s_3^3s_2^8s_1^5+s_3^3s_2^4s_1+s_3^3s_1^5-s_3^2s_2^8s_1^6-s_3^2s_2^4s_1^2-
s_3^2s_1^6-s_3s_2^4s_1^3-s_3s_1^7-s_2^8s_1^8-s_2^8-s_2^4s_1^4-s_1^8+1.
\end{align*}}}

\end{document}